\documentclass[11pt]{report}
\usepackage{blindtext, amssymb, bbm, hyperref, graphicx, geometry, amsmath, amsthm, xcolor, tikz, multicol, float, extarrows}
\usepackage[onehalfspacing]{setspace}
\usepackage[T1]{fontenc}
\usepackage[utf8]{inputenc}

\def\Gal{\operatorname{Gal}}
\newcommand{\pfrac}[2]{%
  \vcenter{\tabskip0pt\offinterlineskip\halign{%
  \strut##&\hfill\hskip1pt##\hskip1pt\hfill&##\cr
  &$#1$&\vrule\cr\noalign{\hrule}\vrule&$#2$&\cr}}%
}
\newtheorem{lemma}{Lemma}[section]
\newtheorem{theorem}[lemma]{Theorem}
\newtheorem{prop}[lemma]{Proposition}

\newtheorem{obs}[lemma]{Observation}
\newtheorem{corollary}[lemma]{Corollary}
\theoremstyle{definition}
\newtheorem{notation}[lemma]{Notation}
\newtheorem{definition}[lemma]{Definition}
\newtheorem{example}[lemma]{Example}
\newtheorem{recipe}[lemma]{Recipe}

\counterwithin{equation}{section}
\begin{document}
\begin{titlepage}
\begin{center}
    \begin{figure}
        \centering
        \includegraphics[width = 15cm]{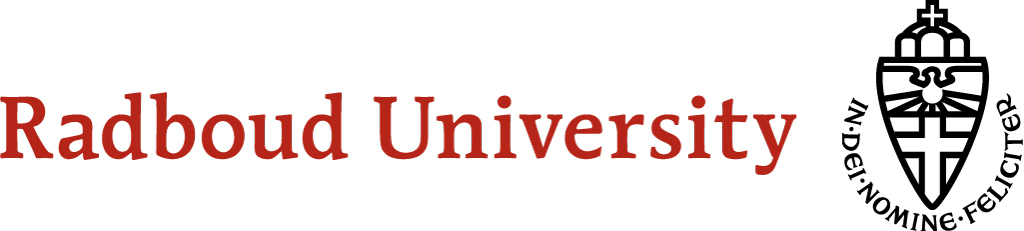}
    \end{figure}
    \large{Master's Thesis in Mathematics}\\[1cm]
    \hrule\vspace*{0.4cm}
    \huge{\textbf{Fractions, Functions and Folding}}\\[0.4cm]
    \large{\textbf{A Novel Link between Continued Fractions, Mahler Functions and Paper Folding}}\\[0.4cm]
    \hrule\vspace*{0.4cm}
\begin{minipage}{0.4\textwidth}
\begin{flushleft} \large
\emph{Author:}\\
\Large{Joris \textsc{Nieuwveld}} 
\end{flushleft}
\end{minipage}
~
\begin{minipage}{0.4\textwidth}
\begin{flushright} \large \large
\emph{Supervisor:} \\
Prof. Wadim \textsc{Zudilin} \\[0.2cm]
\emph{Second reader:} \\
Dr. Wieb \textsc{Bosma} 
\end{flushright}
\end{minipage}\\[1cm]
\begin{figure}[h]
    \centering
    \includegraphics[width = 8cm]{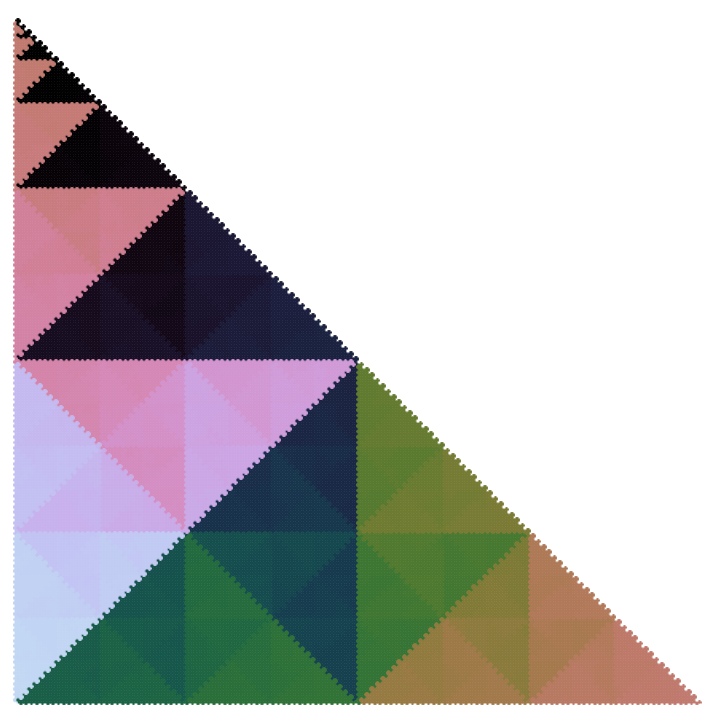}
\end{figure}
\vspace*{\fill}
\emph{Date: }\today
\end{center}
\end{titlepage}
\newpage
\begin{flushright}
\begin{minipage}{7 cm}
\textit{The front page illustration depicts the two folding curves described in Subsection \ref{paperfolding curve rho subsection} after 16 iterations. They merge to form a triangle. The first folding curve starts brownish pink and turns into pale blue and belongs to $-\overleftarrow{\boldsymbol{w}_{14}}$, while the second starts black and fades to blue, green and copper and belongs to $\boldsymbol{w}_{15}$.}
\end{minipage}
\end{flushright}
\vspace*{12cm}
\subsection*{Abstract}
Repeatedly folding a strip of paper in half and unfolding it in straight angles produces a fractal: the dragon curve. Shallit, van der Poorten and others showed that the sequence of right and left turns relates to a continued fraction that is also a simple infinite series. We construct a Mahler function from two functions of Dilcher and Stolarsky with similar properties. It produces a predictable irregular continued fraction that admits a regular continued fraction and a shape resembling the dragon curve. Furthermore, we discuss numerous variations on this theme.
\tableofcontents

\chapter*{Introduction}
\addcontentsline{toc}{chapter}{Introduction}
The simple functional equation $H(x) = H(x^2) + x H(x^4)$ seems to have few interesting properties. A solution has a power series expansion
$$H(x) = 1+x+x^2+x^4+x^5+x^8+x^9+x^{10} +x^{16}+x^{17}+x^{18}+x^{20}+x^{21}+x^{32} + \dotsm$$
whose coefficients are all 0 and 1. The series contains long strings of zeros that end at powers of 2, and the number of coefficients 1 relates to the Fibonacci numbers. However, a more profound beauty is hidden inside $H$. When investigating the functional equation at $x \mapsto x^{-1}$, one finds out that its solutions ultimately are two functions from Dilcher and Stolarsky \cite{dilcher2009stern}. Mimicking these two functions, an irregular continued fraction in a spirit of Ramanujan's \cite[(9.29)]{borwein2014neverending} can be found for $\rho(x):= \frac{H(x)}{H(x^2)}$:
$$
1 + \cfrac{x}{1 + \cfrac{x^2}{1 + \cfrac{x^4}{1 + \cfrac{x^8}{\ddots}}}}.
$$
This gives a recipe for defining $\rho$ at many points on the unit circle, where $H$ remains undefined. Only, most strikingly, $\rho$ also allows to be written as a \emph{regular} continued fraction whose partial quotients are, apart for possibly one, all $\pm x$ and form a sequence related to the paperfolding dragon. The present thesis is a mathematical exposition of numerous features of $H$ and $\rho$ and their close relatives.\medskip

Chapter 1 introduces four intertwining Mahler functions that exist within the unit disk. Dilcher and Stolarsky \cite{dilcher2009stern} recently constructed the first two, $F$ and $G$. The third, $H$, emerges when studying their system of Mahler equations at infinity. Getting only one power series instead of two to solve the system is already an interesting phenomenon. The fourth, $I$, is the generating function of the Baum-Sweet sequence and satisfies the never before observed relation $xF(x^3) + G(x^3) = I(x)$. Next, we construct two continued fractions. The first, $\lambda$, was already studied by Dilcher and Stolarsky and relates to $F$ and $G$; the second continued fraction is $\rho$. We prove that $\lambda(x) = x\rho(x^{-3})$ whenever both sides make sense and that both assume algebraic values at certain roots of unity.\medskip

Chapter 2 begins with recalling Shallit's and van der Poorten's link between continued fractions, $f(x) = \sum_{n=0}^{\infty}x^{2^n-1}$ and the paperfolding dragon \cite{van1992folded}. This earliest example of a folded continued fraction will illustrate how $\rho(x)$ above corresponds to two regular continued fractions whose partial quotients are all, except for possibly one, equal to $x$ and $-x$. These continued fractions coincide within the unit disk but differ outside. Like the paperfolding dragon, the signs of the partial quotients produce fractals, and their generating functions induce algebraically independent Mahler functions. The second half of the chapter focuses on more general folded continued fractions. Examples like the paperfolding dragon and $\rho$ are rare because many others converge to algebraic functions near the origin. We discuss several cases. Finally, an observation analogous to a result of Cohn \cite{cohn1996symmetry} as well as a simple method to find identities involving Fibonacci numbers are presented.\medskip

In Chapter 3, we study Mahler functions in combination with the Hadamard product, which is the termwise multiplication of power series. Many well-known spaces of functions are closed under this operation, including many large subspaces of Mahler functions. We recall a few of these subspaces, expand them slightly and show that the entire space of Mahler functions does not have this property. Finally, we classify the rational functions whose Hadamard product with each Mahler function is Mahler.

\subsection*{Acknowledgements}
First and foremost, I am extremely grateful to Wadim Zudilin for his excellent guidance, willingness to read and respond to my dozens of ``quick notes'' and enthusiasm for the various unconventional tangents I chose within this research. I also want to thank my former roommate Dion for his ability to listen to my ramblings and encouragement. Finally, I thank my parents for their continued support and the interest they have taken in (the graphical parts) of my thesis.
\chapter{Four curious power series}
In this chapter, we investigate a rabbit hole of functions and the remarkable connections between them. We introduce the functions $F$, $G$, $H$ and $I$, proof a few simple properties and connect these functions with continued fractions.

\section{Defining four intertwining power series}

In this section, four so-called Mahler functions are introduced and linked together. Before defining what a Mahler function is, the first two functions are given. 

\subsection{Two power series}
The two power series called $F$ and $G$ are defined as the unique power series around $q=0$ that satisfy $F(0) = G(0) = 1$ and the system of equations
\begin{align}\label{defFG}\begin{split}
        F(q) &= G(q^2) + qF(q^4),\\
        G(q) &= qF(q^2) + G(q^4).
\end{split}\end{align}
These power series are
\begin{align}\begin{split}
    F(q) &= \sum_{n=0}^{\infty} a_nq^n = 1+q+q^2+q^5+q^6+q^8+q^9+q^{10}+ \dotsm \quad\text{and}\\
    G(q) &= \sum_{n=0}^{\infty} b_nq^n = 1+q+q^3+q^4+q^5+q^{11}+q^{12}+q^{13} + \dotsm.
\end{split}\end{align}
They were constructed by Dilcher and Stolarsky in \cite{dilcher2009stern} as limits of so-called Stern polynomials. Here we use the system of equations to define $F$ and $G$ rather than the original construction because the system has a unique power series solution. Dilcher and Stolarsky also proved in \cite[Proposition 5.1]{dilcher2009stern} that the mixed system of $F$ and $G$ can be split into two systems:
\begin{align}\label{FG long form}\begin{split}
    F(q) &= (1+q+q^2)F(q^4) - q^{4}F(q^{16}) \quad \text{and}\\
    qG(q) &= (1+q+q^2)G(q^4) - G(q^{16}).\end{split}
\end{align}
This makes $F$ and $G$ transparently in a known class functions because of the following definition:
\begin{definition}\label{definition Mahler function}
Let $k \ge 2$. An analytic function $f \in \mathbb{C}[[q]]$ is called a $\boldsymbol{k}$\textbf{-Mahler function} if there exist a $d \ge 0$ and polynomials $A(q),A_0(q),A_1(q),\dots,A_d(q) \in \mathbb{C}[q]$ such that $A_0(q)A_d(q) \ne 0$ and 
\begin{align}\label{GeneralMahlerForm}
    A(q) + A_0(q)f(q) + A_1(q)f(q^k) + \dots + A_d(q)f\big(q^{k^d}\big) = 0.
\end{align}
The (linear) functional equation above, solved by a $k$-Mahler function, is called a \textbf{$\boldsymbol{k}$-Mahler equation}. Here $d$ is said to be the \textbf{degree} of $f$. The $k$-Mahler equation is called \textbf{homogeneous} if $A(q) = 0$ and  \textbf{inhomogeneous} otherwise. If $k$ is clear from the context, it is often omitted from the terminology.
\end{definition}
From \eqref{FG long form}, we see that $F$ and $G$ satisfy homogeneous $4$-Mahler equations of degree 2. Note that this implies that $F$ and $G$ also satisfy $2$-Mahler equations of degree $4$.\medskip

By comparing the coefficients of these power series, the uniqueness of the expansions can be shown from the system \eqref{defFG} through a recursion for the coefficients:

\begin{align}\label{CoefficientFG}
  a_n =
  \begin{cases}
    b_{\frac{n}{2}} & \text{if }n \equiv 0 \mod 2, \\
    a_{\frac{n-1}{4}} &\text{if } n \equiv 1 \mod 4, \\
    0 & \text{if }n \equiv 3 \mod 4
  \end{cases} \quad\text{and} \quad
    b_n =
  \begin{cases}
    a_{\frac{n-1}{2}} &\text{if } n \equiv 1 \mod 2, \\
    b_{\frac{n}{4}} &\text{if } n \equiv 0 \mod 4, \\
    0 & \text{if }n \equiv 2 \mod 4.
  \end{cases}
\end{align}

The recursions show that the coefficients of $F$ and $G$ are automatic, a property we do not study in this thesis. In particular, the recursions imply that all coefficients of $F$ and $G$ are either $0$ or $1$ by induction. This property is given its own name:
\begin{definition}
A \textbf{$\boldsymbol{\{0, 1\}}$-power series} is a power series whose coefficients are all $0$ or $1$.
\end{definition}
Furthermore, by \cite[Proposition 5.2]{dilcher2009stern}, $F$ and $G$  satisfy \begin{align}\label{Cross relation F and G}
G(q)G(q^2) - qF(q)F(q^2) = 1 \quad\text{and}\quad F(q)G(q^4) - qG(q)F(q^4) = 1.
\end{align} 
Such relations falsely suggest there are more algebraic relations between $F(q)$, $G(q)$, $F(q^2)$, $G(q^2)$. 
\begin{definition}
Let $S$ be a set of functions over $\mathbb{C}$ in variable $q$. The \textbf{transcendence degree} of $S$ is the cardinality of the largest subset $S' \subseteq S$ such that the fields $\mathbb{C}(q)(S)$ and $\mathbb{C}(q)(S')$ coincide. The notation used for this is $\operatorname{tr\,deg}(S)$.
\end{definition}
A more general notion of transcendence degree exists but is not used here. Obtaining the transcendence degree involving all possible $F(q^d)$ and $G(q^d)$ for all positive integers $d$ turns out to be a rather difficult question, but when only studying powers of $2$, this is somewhat doable. For example, $F\big(q^{2^d}\big)$ can be written as a linear combination over $\mathbb{C}[q]$ of $F(q), G(q), F(q^2)$ and $G(q^2)$ by simply iterating the equations of the system $\eqref{defFG}$. Therefore,
\begin{align*}
    &\operatorname{tr\,deg}\big(F(q), G(q), F(q^2), G(q^2), F(q^4), G(q^4), F(q^4), G(q^4),F(q^8), G(q^8),\dots\big)\\
    &\quad= \operatorname{tr\,deg}\big(F(q), G(q), F(q^2), G(q^2)\big) \le 4.
\end{align*}
Adding $G(q)G(q^2) - qF(q)F(q^2) = 1$ \eqref{Cross relation F and G} to this conclusion lowers the upper bound to 3. In 2015, Bundschuh and V\"a\"an\"anen proved in \cite{bundschuh2015transcendence} that the transcendence degree is indeed equal to 3, and that there are no algebraic relations between any three of $F(q)$, $G(q)$, $F(q^2)$ and $G(q^2)$. There is a related result on the transcendence of these functions \cite{brent2016algebraic}: $$\operatorname{tr\,deg}\big(F(q), F(q^2), F'(q), F'(q^2)\big) = 4.$$

To conclude the subsection, notice that $F$ and $G$ are defined everywhere within the unit disk as $\{0, 1\}$-power series. However, both cannot be analytically continued any further. That is, the unit circle forms an impassable boundary \cite[Lemma 1]{brent2016algebraic}.
\subsection[Defining the third function]{Defining the third function, \texorpdfstring{$\boldmath{H}$}{H}}
As the unit circle forms an impassable boundary for $F$ and $G$, what happens with the system of equations at the other side of the unit circle seems interesting. To do this, view the system \eqref{defFG} at infinity by formally defining $\tilde{F}(q) := F(q^{-1})$ and $\tilde{G}(q) := G(q^{-1})$ for all $0 < |q| < 1$. This produces the system of equations
\begin{align*}
    \tilde{F}(q) &= \tilde{G}(q^2) + q^{-1}\tilde{F}(q^4), \\
    \tilde{G}(q) &= q^{-1}\tilde{F}(q^2) + \tilde{G}(q^4),
\end{align*}
which does not have a solution in power series. Fortunately, it has a power series solution in $q^{\frac{1}{3}}$, which one extracts by defining $\hat{F}(q) := q^{\frac{2}{3}}\tilde{F}(q)$ and $\hat{G}(q) := q^{\frac{1}{3}}\tilde{G}(q)$. Then 
\begin{align*}
        \hat{F}(q) &= \hat{G}(q^2) + q\hat{F}(q^4), \\
    \hat{G}(q) &= \hat{F}(q^2) + q\hat{G}(q^4).
\end{align*}
This system has a lot of symmetry. Like $F$ and $G$, we demand $\hat{F}(0) = \hat{G}(0) = 1$ and that both have a power series expansion. Then there is a unique power series solution $H(q)=\hat{F}(q) = \hat{G}(q)$. This can be most easily seen as follows: Define $A(q) = \sum_{n=0}^{\infty}x_nq^n = \hat{F}(q) - \hat{G}(q)$. Then $A(q) = -A(q^2) +qA(q^4)$, which gives the following recursion for $x_n$:
\begin{align*}
  x_n =
  \begin{cases}
    -x_{\frac{n}{2}} &\text{if } n \equiv 0 \mod 2, \\
    x_{\frac{n-1}{4}} & \text{if }n \equiv 1 \mod 4, \\
    0 & \text{if }n \equiv 3 \mod 4.
  \end{cases}
\end{align*}
As $x_0 = A(0) = \hat{F}(0)-\hat{G}(0) = 0$, all $x_n$ are 0 by induction, giving that $A(q)$ is the zero function. Thus, $H(q)=\hat{F}(q) = \hat{G}(q)$ is indeed well-defined and satisfies
\begin{align}\label{DefH}
    H(q) = H(q^2) + qH(q^4).
\end{align}
This gives us
\begin{align*}
    H(q) = \sum_{n=0}^{\infty}c_nq^n = 1+q+q^2+q^4+q^5+q^8+q^9+q^{10} +q^{16}+\dotsm.
\end{align*}
The $2$-Mahler solution of equation \eqref{DefH} that satsifies $H(0)=1$ is taken as definition of $H(q)$. In particular, $H$ is a Mahler function. Like for $F$ and $G$, the Mahler equation induces a recursion for the coefficients of the power series:
\begin{align}\label{CoeffiecientsH}
  c_n =
  \begin{cases}
    c_{\frac{n}{2}} &\text{if } n \equiv 0 \mod 2, \\
    c_{\frac{n-1}{4}} &\text{if } n \equiv 1 \mod 4, \\
    0 &\text{if } n \equiv 3 \mod 4.
  \end{cases}
\end{align}
From this recursion, it follows that $H(q)$ is also a $\{0, 1\}$-power series. The sequence $(c_n)_{n=0}^{\infty}$ is already known: $c_n = 1$ if and only if $n$ is in A003714 on the OEIS \cite{OEIS}, the so-called `Fibbinary' sequence, the numbers without consecutive 1's in their binary expansion.  This can easily be verified with the recursive properties in \eqref{CoeffiecientsH}. Another property of this sequence is that $c_n + \binom{3n}{n} \equiv 1 \mod 2$ \cite{OEIS}. Again, mirroring $F$ and $G$, $H$ can be written as a 4-Mahler function.
\begin{prop}
For all $q$ such that $|q| < 1$, $H$ also satisfies
\begin{align*}
    H(q) = (1 +q+q^2)H(q^4)-q^6H(q^{16}).
\end{align*}
\end{prop}
\begin{proof}
From the system \eqref{DefH}, it follows that $H(q^2) = H(q^4) + q^2H(q^8)$ and $H(q^4) = H(q^8) + q^4H(q^{16})$, so \begin{align*}
    H(q)& = (H(q^4) + q^2H(q^8)) + qH(q^4) + q^2(H(q^4) - H(q^8) - q^4H(q^{16})) \\
    &= (1 +q+q^2)H(q^4) -q^6H(q^{16}).\qedhere
\end{align*}
\end{proof}
The sole difference between $H$ and $F$ is the presence of a factor $-q^6$ in place of $-q^4$ in the 4-Mahler equations. This difference  changes the power series expansions, and thus the functions, fundamentally.\medskip

By iterating equation \eqref{DefH}, it follows that $\operatorname{tr\,deg}\big(H(q), H(q^2), H(q^4), H(q^8),\dots\big) \le 2$. This is remarkable, as the system on the other side of the unit circle has transcendence degree 3. It turns out that this upper bound is sharp:
\begin{prop}\label{prop H has }$\operatorname{tr\,deg}\big(H(q), H(q^2), H(q^4), H(q^8),\dots\big) = 2$.
\end{prop}
Proposition \ref{prop H has } can be proven with a conventional theorem as in \cite{nishioka1990new}, but a short-cut exists within the same paper for this particular case. This short-cut is given in the next subsection.
\subsection[The fourth power series: I]{The fourth power series: $I$}
The existence of $H$ and its accompanying transcendence degree is remarkable but can be explained by studying it once more at infinity. Thus, let $\tilde{H}(q)$ satisfy
$$
\tilde{H}(q) = \tilde{H}(q^2) + q^{-1}\tilde{H}(q^4).
$$
Then there is no power series in $q$ that satisfies this equation, but there is one in $q^{\frac{1}{3}}$. To extract this solution, define $I(q) = q\tilde{H}(q^3)$ such that
\begin{align*}
    qI(q) = q^{2}I(q^2) +q^{-3+4}I(q^4).
\end{align*}
Cleaning this up gives that the newly found Mahler equation
\begin{align}\label{defI}
    I(q) = qI(q^2) + I(q^4),
\end{align}
which has a unique power series solution when setting $I(0)=1$:
\begin{align*}
    I(q) &= \sum_{n=0}^{\infty} d_nq^n = 1 + q + q^3 + q^4 + q^7 + q^9 + q^{12} + q^{15} + q^{16} + \dotsm.
\end{align*}
The coefficients of $I$ satisfy 
\begin{align}\label{CoeffiecientsI}
  d_n =
  \begin{cases}
    d_{\frac{n-1}{2}} &\text{if } n \equiv 1 \mod 2, \\
    d_{\frac{n}{4}} &\text{if } n \equiv 0 \mod 4, \\
    0 &\text{if } n \equiv 3 \mod 4,
  \end{cases}
\end{align}
and so $I$ is also a $\{0, 1\}$-power series.\medskip

In contrast to $H$, $I$ is a well-known function: It is the generating function of the Baum-Sweet sequence \cite[A086747]{OEIS}. This sequence is defined as ``$d_n = 1$ if the binary representation of $n$ contains no block of consecutive zeros of odd length; otherwise $d_n = 0$.'' For $n=0$, this definition is ambiguous, as the binary representation of 0 can be defined as the empty string or 0. The OEIS is in the minority defining $d_0=0$ while Baum and Sweet \cite{baum1976continued}, Nishioka \cite{nishioka1990new} and Wikipedia \cite{wikiBaumSweet} use $d_0=1$. Hence the following definition is chosen.
\begin{definition}
The \textbf{Baum-Sweet sequence} $(d_n)_{n=0}^{\infty}$ is defined by $d_0 = 1$, and for all $n \ge 1$, $d_n = 0$ when the binary representation contains a block of zeros of odd length and 1 otherwise.
\end{definition}
Using equation \eqref{defI}, it is easily verified that $I$ is indeed the generating function of the Baum-Sweet sequence. We remark that Baum and Sweet introduced their sequence as the unique root of $I(q)^3 +q^{-1}I(q) + 1 = 0$ in $\mathbb{F}_2((q^{-1}))$ \cite{baum1976continued}.\medskip

Like the three other functions, $I$ can also be written in terms of $I(q^4)$ and $I(q^{16})$.
\begin{prop}\label{propBiggerFormI}
For all $q$ such that $|q| < 1$, $I(q)$ satisfies $$q^3I(q)  = (1+q^3+q^6)I(q^4) - I(q^{16}).$$
\end{prop}
\begin{proof}
From \eqref{defI}, extract $qI(q^2) = q^3I(q^4) + qI(q^8)$ and $qI(q^8) = q^{-3}I(q^4) - q^{-3}I(q^{16})$. Thus,
\begin{align*}
    I(q) = q^3I(q^4) + qI(q^8) + I(q^4) + q^{-3}I(q^4) - q^{-3}I(q^{16}) - qI(q^8) = (q^{-3}+1+q^3)I(q^4) - q^{-3}I(q^{16})
\end{align*}
and by multiplying with $q^3$, the result follows.
\end{proof}
While $H(q)$ and $F(q^3)$ had almost the same 4-Mahler equation, $I(q)$ and $G(q^3)$ satisfy exactly the same 4-Mahler equation. Moreover, $qF(q^3)$ also satisfies this 4-Mahler equation. This is remarkable, but also paves the way for an identity that, to the best of our knowledge, has not yet been observed.
\begin{prop}\label{propFGH}
For all $q$ such that $|q| < 1$, $I(q) = qF(q^3) + G(q^3)$.
\end{prop}
\begin{proof}
As $F(q) = G(q) + qF(q^4)$ and $G(q) = qF(q) + G(q^4)$, $qF(q^3)+G(q^3) = qG(q^6) + q^4F(q^{12})+ q^3F(q^6) + G(q^{12}) = q(q^2F(q^6)+G(q^6)) + (q^4F(q^{12})+G(q^{12}))$. Then $qF(q^3) + G(q^3)$ follows the same Mahler equation as $I(q)$. As \eqref{defI} gives a unique power series when $I(0)=1$ is fixed, verifying the constant terms on both sides gives the proposition.
\end{proof}
Due to the fame of the Baum-Sweet sequence, transcendence results are already known for $I$.
\begin{theorem}[Example in \cite{nishioka1990new}]
$I(q)$ and $I(q^2)$ are algebraically independent over $\mathbb{C}[q]$.
\end{theorem}
This theorem gives the short-cut needed to reach the transcendence degree of $H(q)$ and $H(q^2)$. Define the Mahler function $\tilde{I}(q)$ as the unique power series whose linear term is 1 and that satisfies
\begin{align*}
    \tilde{I}(q) = q^{-1}\tilde{I}(q^2) + \tilde{I}(q^4).
\end{align*}
The transcendence degree of this function can easily be established:
\begin{lemma}\label{lemma tilde I alg ind}
$\operatorname{tr\,deg}\big(\tilde{I}(q), \tilde{I}(q^2)\big) = 2$.
\end{lemma}
\begin{proof}
Apply the proof of the last example in \cite{nishioka1990new} with $z = q^{-1}$. 
\end{proof}
From here, $H$ can be reached directly.
\begin{proof}[Proof of Proposition \ref{prop H has }]
We have $qH(q^3) = q^{-1} \cdot q^2H(q^6) + q^4H(q^{12})$, and the linear term of $qH(q^3)$ is $H(0) = 1$, hence $\tilde{I}(q) = qH(q^3)$. If $H(q)$ and $H(q^2)$ were not algebraically independent over $\mathbb{C}[q]$, then there would exist a polynomial $P \in \mathbb{C}[q, H(q), H(q^2)]$ such that $P(q, H(q), H(q^2)) = 0$. Then $$P(q^3, H(q^3), H(q^6)) = P(q^3, q^{-1}\tilde{I}(q), q^{-2}\tilde{I}(q^2)) = 0,$$
which contradicts Lemma \ref{lemma tilde I alg ind}.\end{proof}
\section{Continued fractions}
In their work on the power series $F$ and $G$, Dilcher and Stolarsky \cite{dilcher2009stern} already noticed that $F$ and $G$ have a connection with continued fractions. For $H$ and $I$ similar constructions exist, but the one for $H$ has another remarkable property -- a form as a folded continued fraction that will be described in Chapter 2. To state and prove all these connections, a crash course on continued fractions is given. As $q$ is often used in this context, $x$ serves as the main variable in the upcoming sections. All results and notations are from \cite{borwein2014neverending}.
\subsection{A crash course on continued fractions}
\begin{definition}
A \textbf{finite continued fraction} is an expression of the form \begin{align*}
    a_0+\dfrac{1}{a_1+\dfrac{1}{a_2+\dfrac{1}{\phantom{\bigg(}^{^{\ddots}}+\dfrac{1^{\phantom{\big(}}}{ a_{n-1} + \dfrac{1}{a_n} }}}}
\end{align*}
with $n \ge 0$, $a_0 \in \mathbb{Z}$, $a_i \in \mathbb{Z}_{> 0}$ for $i = 1,2, \dots, n$. This is denoted by $[a_0;\:a_1,\:a_2,\:\dots,\:a_n]$, and the $a_i$ are called the \textbf{partial quotients} of the continued fraction.
\end{definition}
For example, the continued fraction $[1 ;\: 2,\: 3]$ is $ 1 + \frac{1}{2+\frac{1}{3}}= 1 + \frac{3}{7} = \frac{10}{7}$. Instead of $\mathbb{Z}$ and $\mathbb{Z}_{> 0}$, other sets of numbers or even polynomial rings can be used but issues with dividing by $0$ may then arise. If there are none, the entire construction can still be used. The advantage of using $\mathbb{Z}$ and $\mathbb{Z}_{>0}$ is that there is now a unique continued fraction with $a_n \ne 1$ for each rational number and no division by zero problems occurs \cite{borwein2014neverending}.\medskip

The partial quotients can also be seen as variables. Let $p_n = p_n(a_0, a_1,\dots,a_n)$ and $q_n = q_n(a_0, a_1,\dots,a_n)$ be polynomials in $n+1$ variables defined by $p_0 = a_0$, $q_0 = 1$, 
\begin{align*}
     p_n &= a_0p_{n-1}(a_1, a_2,\dots,a_{n})+q_{n-1}(a_1, a_2,\dots,a_{n}) \quad\text{and}\\
     q_n &= p_{n-1}(a_1, a_2,\dots,a_{n-1}).
\end{align*}
Then one can show by induction that  $\frac{p_n}{q_n} = [a_0;\: a_1,\: a_2,\:\dots,\: a_n]$. To simplify some statements, define $p_{-1} = 1$ and $q_{-1} = 0$. The following lemma will often be used.
\begin{lemma}[Key Lemma, Lemma 2.8 in \cite{borwein2014neverending}]\label{lemma key lemma continued fractions}
For each $n \ge 1$,
\begin{align*}
    \begin{pmatrix}
        p_n & p_{n-1} \\q_n &q_{n-1} 
    \end{pmatrix}=
    \begin{pmatrix}
        a_0 & 1 \\ 1 & 0
    \end{pmatrix}
        \begin{pmatrix}
        a_1 & 1 \\ 1 & 0
    \end{pmatrix}
    \dots
        \begin{pmatrix}
        a_{n-1} & 1 \\ 1 & 0
    \end{pmatrix}
        \begin{pmatrix}
        a_n & 1 \\ 1 & 0
    \end{pmatrix}.
\end{align*}
\end{lemma}
The matrix $\begin{pmatrix} p_n & p_{n-1} \\q_n &q_{n-1}  \end{pmatrix}$ is called the \textbf{matrix of continuants} of the continued fraction. A few important consequences are:
\begin{lemma}[Lemma 2.9 in \cite{borwein2014neverending}]\label{lemma recursion p_n and q_n}
Let $n\ge 1$. Then $p_n = a_np_{n-1} + p_{n-2}$ and $q_n = a_nq_{n-1} + q_{n-2}$.
\end{lemma}
\begin{lemma}[Theorem 2.14 in \cite{borwein2014neverending}]\label{lemma cross product p_n and q_n}
For all $n \ge 0$, $q_np_{n-1} - p_nq_{n-1} = (-1)^n$.
\end{lemma}
The next ingredient is \textbf{infinite continued fractions}. These are simply defined as the limit of a sequence of their finite truncations. Famously, the continued fraction $[1;\:1,\:1,\:1,\:1,\:1,\:1,\:\dots]$ is equal to $\frac{1 + \sqrt{5}}{2}$, the golden ratio. An infinite continued fraction such that $a_0 \in \mathbb{Z}$ and $a_i \in \mathbb{Z}_{\ge 1}$ for all $i \ge 1$ always converges. In the opposite direction, for each real irrational number, there is a unique infinite continued fraction that converges to it \cite{borwein2014neverending}. Less conventionally, we also define that a sequence of truncated continued fractions $(b_n)_{n=0}^{\infty} = ([a_0;\:a_1,\:\dots,\:a_n])_{n=0}^{\infty}$ is \textbf{parity partial converging} if the limits $(b_{2n})_{n=0}^{\infty}$ and $(b_{2n+1})_{n=0}^{\infty}$ exist (but perhaps differ).\medskip

To describe a continued fraction associated to $H$, a slightly more general notion has to be introduced: the irregular continued fraction.
\begin{definition}
Let $n \ge 0$. Then an \textbf{irregular continued fraction} is an expression of the from 
\begin{align*}
    a_0+\dfrac{b_1}{a_1+\dfrac{b_2}{a_2+\dfrac{b_3}{\phantom{\bigg(}^{^{\ddots}}+\dfrac{b_{n-1}^{\phantom{\big(}}}{ a_{n-1} + \dfrac{b_n}{a_n} }}}}
\end{align*}
 denoted by 
\begin{align*}
    a_0+\pfrac{b_1}{a_1} + \pfrac{b_2}{a_2}+ \pfrac{b_3}{a_3}
+\dots+\pfrac{b_n}{a_n}.\end{align*}\end{definition}
As for usual continued fractions, we similarly define infinite irregular fractions, when the limit of their truncations exists, and the same holds for parity partial convergence.
\subsection[Continued fractions related to F, G, H and I]{Continued fractions related to $\boldmath{F}$, $\boldmath{G}$, $\boldmath{H}$ and $\boldmath{I}$}
For all $n \ge 0$, define the two continued fractions
\begin{align}\label{defCn}\begin{split}
    \lambda_n(x) &:= [x;\: x^2,\: x^4,\: x^8, \:\dots,\: x^{2^{n}}] \quad \text{and}\\
\lambda^+_n(x) &:= [x;\: x^2,\: x^4,\: x^8, \:\dots,\: x^{2^{n}}+1]
\end{split}\end{align}
and the irregular, `upside down', continued fraction 
\begin{align}\label{defrhon}
    \rho_n(x) := 1+\pfrac{x}{1}+\pfrac{x^{2}}{1}+\pfrac{x^{4}}{1} + \dots +\pfrac{x^{2^{n-2}}}{1} + \pfrac{x^{2^{n-1}}}{1}.
\end{align}
These two continued fractions turn out to be connected with the four Mahler functions discussed before. Both $\lambda_n$ and $\rho_n$ are defined for many $n$ on a large subset of $\mathbb{C}$ but not everywhere.
\begin{example}\label{counterexample C_n zeta_12 undefined}
Let $\zeta_{12}$ denote $\exp(\frac{2\pi i}{12})$. Then $$\lambda_3(\zeta_{12} ) = \zeta_{12} +\pfrac{1}{\zeta_{12} ^2}+\pfrac{1}{\zeta_{12}^4} = \zeta_{12} +\pfrac{\zeta_{12} ^4}{\zeta_{12}^6+1} = \zeta_{12}  +\frac{\zeta_{12}^4}{0}$$
is not defined.
\end{example}
Naturally, only assign a value to $\lambda_n(x)$ and $\rho_n(x)$ when such problems do not occur for a particular $x$. The partial convergents of $\lambda_n$ and $\rho_n$ take elegant forms. To find them, define, for all $n \ge 0$,
\begin{align*}
    F_n(x) := \sum_{n=0}^{2^n-1} {a_n}x^n, &\quad\quad G_n(x) := \sum_{n=0}^{2^n-1} {b_n}x^n,\\
    H_n(x) := \sum_{n=0}^{2^n-1} {c_n}x^n,&\quad\quad I_n(x) := \sum_{n=0}^{2^n-1} {d_n}x^n.
\end{align*}
If $n < 0$, set $F_n(x) = G_n(x) = H_n(x) = I_n(x) = 1$. Clearly, $\lim_{n \to \infty} F_n(x)$ equals $F(x)$ for $|x| < 1$, and such limits similarly exist for the other three cases. We have the following Mahler-like recursions.
\begin{lemma}\label{lemma recursions for F_n, G_n, H_n, and I_n}
Let $n \ge 1$. Then
\begin{align*}
    F_n(x) = G_{n-1}(x^2) + xF_{n-2}(x^4), &\quad G_n(x) = xF_{n-1}(x^2) + G_{n-2}(x^4), \\
    H_n(x) = H_{n-1}(x^2) + xH_{n-2}(x^4), &\quad I_n(x) = xI_{n-1}(x^2) + I_{n-2}(x^4).
\end{align*}\end{lemma}\begin{proof}
All these relations follow from the recursion formulas \eqref{CoefficientFG}, \eqref{CoeffiecientsH} and \eqref{CoeffiecientsI}.\end{proof}
With these expressions, the convergents of the three continued fractions can be written explicitly.
\begin{prop}\label{PropDivisionContFrac} For all $n \ge 0$,
\begin{enumerate}
    \item $\rho_n(x) = \frac{H_{n}(x)}{H_{n-1}(x^2)}$.
    \item $\lambda_n^+ = \frac{I_{n}(x)}{I_{n-1}(x^2)}$.
    \item $\lambda_n(x)$ equals $\frac{xF_n(x^3)}{G_{n-1}(x^6)}$ for even $n$ and $\frac{G_n(x^3)}{x^2 F_{n-1}(x^6)}$ for odd $n$.
\end{enumerate}
\end{prop}\begin{proof}\begin{enumerate}
\item For $n=0$, $\rho_0(x) = 1 = \frac{H_0(x)}{H_{-1}(x^2)}$. Then use induction and Lemma \ref{lemma recursions for F_n, G_n, H_n, and I_n} for the truncated sums:
\begin{align*}
    \rho_{n+1}(x) = 1 + \frac{x}{\rho_{n}(x^2)} = 1 + \frac{x}{H_n(x^2)/H_{n-1}(x^4)} = \frac{xH_{n-1}(x^4) +H_n(x^2)}{H_n(x^2)} = \frac{H_{n+1}(x)}{H_n(x^2)}.
\end{align*}
\item For $n=0$, $\lambda^+_0(x) = 1 = \frac{I_0(x)}{ I_{-1}(x^2)}$; again by induction and Lemma \ref{lemma recursions for F_n, G_n, H_n, and I_n},
\begin{align*}
    \lambda^+_{n+1}(x) = x + \frac{1}{\lambda^+_{n}(x^2)} = x + \frac{1}{I_n(x^2)/I_{n-1}(x^4)} = \frac{I_{n-1}(x^4) +xI_n(x^2)}{I_n(x^2)} = \frac{I_{n+1}(x)}{I_n(x^2)}.
\end{align*}
\item The statement holds true for $n=0$ by construction as $F_0(x) = G_{-1}(x) = 1$ and $\lambda_0(x) = x$. Assume that $n > 0$ and $\lambda_{n}(x) = \frac{p_n(x)}{q_n(x)}$ and apply Lemma \ref{lemma recursions for F_n, G_n, H_n, and I_n}:
\begin{align*}
    \lambda_{n+1}(x) = x + \cfrac{1}{\lambda_n(x^2)} = x + \cfrac{q_n(x^2)}{p_n(x^2)} = \cfrac{xp_n(x^2)+q_n(x^2)}{p_n(x^2)}.
\end{align*}
For even $n$, $q_{n+1}(x) = p_n(x^2) = G_{n-1}(x^6)$ and 
\begin{align*}
    p_{n+1}(x) = x\cdot x^3F_n(x^6) + G_{n-1}(x^{12}) = G_{n+1}(x^3)
\end{align*}
and, for odd $n$, $q_{n+1}(x) = p_n(x^2) = x^2F_{n-1}(x^6)$ and
\begin{align*}
    p_{n+1}(x) = xG_n(x^6) + x^4F_{n-1}(x^{12}) = xF_{n+1}(x^3).
\end{align*}
Thus, by induction the statement follows.\qedhere\end{enumerate}\end{proof}

Now, wherever possible, one can also form (irregular) infinite continued fractions by passing to the limits $$\rho(x) := \lim_{n \to \infty}\rho_n(x), \quad \lambda^+(x) := \lim_{n \to \infty}\lambda^+_n(x)\quad \text{and}\quad \lambda(x) := \lim_{n \to \infty}\lambda_n(x).$$ As $F(x), G(x), H(x)$ and $I(x)$ are defined only for $|x|<1$, the limits of the first two exist for all $|x|<1$ and are equal to $\frac{H(x)}{H(x^2)}$ for $\rho(x)$ and $\frac{I(x)}{I(x^2)}$ for $\lambda^+(x)$. Dilcher and Stolarsky found a parity partial limit for $\lambda(x)$.
\begin{prop}[Propositions 6.3 and 7.1 of \cite{dilcher2009stern}]\label{proppartialconv} For all $x$ such that $0 < |x| < 1$,
\begin{align*}
    \lim_{\substack{n \to \infty \\ \text{$n$ even}}} \lambda_n(x) &= \frac{xF(x^3)}{G(x^6)} \quad\text{and}\\
    \lim_{\substack{n \to \infty \\ \text{$n$ odd}}} \lambda_n(x)&=\frac{G(x^3)}{x^2F(x^6)}.
\end{align*}
Moreover, if $|x|>1$, $\lim_{n \to \infty} \lambda_n(x)$ also exists and is unique.
\end{prop}
Thus, the extra +1 in the last partial quotient of $\lambda^+_n$ creates usual convergence instead of parity partial convergence if $|x| < 1$! On the other side of the unit circle, for $|x|>1$, its extra role is lost as $x^{2^n}$ grows quickly, giving that $\lambda(x) = \lambda^+(x)$ for these $|x|>1$. So the oddity of different behaviour inside and outside the unit circle of the Mahler functions is carried out in the continued fractions.\medskip

Due to their continued-fraction forms, $\rho$, $\lambda^+$ and $\lambda$ also satisfy Mahler-like identities:
\begin{align*}
    \rho(x) = 1 +\frac{x}{\rho(x^2)}, \quad \lambda^+(x) = x +\frac{1}{\lambda^+(x^2)} \quad \text{and}\quad {\lambda}(x) = x +\frac{1}{{\lambda}(x^2)}.
\end{align*}
These are not authentic Mahler equations, and as the examples of $\lambda$ and $\lambda^+$ show, the solutions to the equations are often not unique. Adamczewski studied $\lambda$ in \cite{adamczewski2010non}, where he denoted it with $\mathcal{C}$.\medskip

As mentioned before, $\rho$ and $\lambda$ are also closely related, and now this can be clearly stated. This gives an explicit link between $H$ and $F, G$, which were originally defined on disjoint domains.
\begin{theorem}\label{theorem rho and c equal}
Let $x \in \mathbb{C}$ such that $x \ne 0$ and $\lambda(x)$ exists. Then $\lambda(x) = x\rho(x^{-3})$ where equality is understood as that for $0 < |x| < 1$, $$\lim_{n \to \infty}\lambda_{2n}(x) = \lim_{n \to \infty} x\rho_{2n}(x^{-3})\quad\text{and}\quad\lim_{n \to \infty}\lambda_{2n+1}(x) = \lim_{n \to \infty} x\rho_{2n+1}(x^{-3})$$ and for $|x| > 1$, $$\lim_{n \to \infty}\lambda_{n}(x) = \lim_{n \to \infty} x\rho_{n}(x^{-3}).$$
\end{theorem}
This theorem also explains the dichotomy of the two parity partial limits of $\lambda(x)$ on one side of the unit circle. This correlates with the fact that $\operatorname{tr\,deg}\big(F(x), G(x), F(x^2), G(x^2)\big) = 3$ and $\operatorname{tr\,deg}\big(H(x), H(x^2)\big) = 2$. To prove this theorem rigorously, some work is required.
\begin{lemma}\label{lemmaHn(q)Fn(q-1)}
Let $n \ge 0$. Then for all $x \in \mathbb{C}\setminus\{0\}$,
\begin{align*}
    H_n(x) = \begin{cases}
    x^{2\frac{2^n-1}{3}}F_n(x^{-1}) & \text{if } n \text{ is even},\\
    x^{\frac{2^{n+1}-1}{3}}G_n(x^{-1}) & \text{if } n \text{ is odd}.
    \end{cases}
\end{align*}\end{lemma}
\begin{proof}
As always, use induction on $n$. For $n = 0$, $H_0(x) = 1 = F_0(x)$, and so the claim holds. Similarly, for $n=1$, $H_1(x) = 1 + x = x(1 + x^{-1}) = xG_1(x^{-1})$. Now apply the identities from Lemma \ref{lemma recursions for F_n, G_n, H_n, and I_n}
to get, for even $n$, 
\begin{align*}
    H_{n+1}(x) &= H_{n}(x^2) + xH_{n-1}(x^4) = (x^2)^{2\frac{2^n-1}{3}}F_n(x^{-2}) + x(x^4)^{\frac{2^{n-1+1}-1}{3}}G_n(x^{-4}) \\&= 
    x^{-1}x^{\frac{2^{n+2}-1}{3}}F_n(x^{-2}) +  x^{\frac{2^{n+2}-1}{3}}G_n(x^{-4}) = x^{\frac{2^{n+2}-1}{3}}G_{n+1}(x^{-1})
\end{align*}
and, for odd $n$,
\begin{align*}
    H_{n+1}(x) &= H_{n}(x^2) + x H_{n-1}(x^4) = (x^2)^{2\frac{2^{n+1}-1}{3}}G_n(x^{-2}) + x(x^4)^{2\frac{2^{n-1}-1}{3}}F_n(x^{-4}) \\&= 
    x^{2\frac{2^{n+1}-1}{3}}G_n(x^{-2}) +  x^{-1}\cdot x^{2\frac{2^{n+1}-1}{3}}F_n(x^{-4}) = x^{2\frac{2^{n+1}-1}{3}}F_{n+1}(x^{-1}).\qedhere
\end{align*}\end{proof}
\begin{prop}\label{proposition rho_n and C_n}
Let $n\ge 0$. Then
\begin{align*}
    \rho_n(x) = \frac{H_n(x)}{H_{n-1}(x^2)} = \begin{cases}
    \dfrac{F_n(x^{-1})}{G_{n-1}(x^{-2})} & \text{if }n \text{ is even}, \\
    \dfrac{G_n(x^{-1})}{x F_{n-1}(x^{-2})} & \text{if }n \text{ is odd}.
    \end{cases}
\end{align*}\end{prop}
\begin{proof}
The statement follows directly from Lemma \ref{lemmaHn(q)Fn(q-1)} as for even $n$ and odd $n$, respectively,
\begin{align*}
    \frac{H_n(x)}{H_{n-1}(x^2)} &= \frac{x^{2\frac{2^{n}-1}{3}}F_n(x^{-1})}{(x^2)^{\frac{2^{n-1+1}-1}{3}}G_{n-1}(x^{-2})} = \frac{F_n(x^{-1})}{G_{n-1}(x^{-2})} \quad \text{and} \\
    \frac{H_n(x)}{H_{n-1}(x^2)} &=\frac{x^{\frac{2^{n+1}-1}{3}}G_n(x^{-1})}{(x^2)^{2\frac{2^{n-1}-1}{3}}F_n(x^{-2})} =\frac{G_n(x^{-1})}{x F_{n-1}(x^{-2})}. \qedhere
\end{align*}\end{proof}
\begin{proof}[Proof of Theorem \ref{theorem rho and c equal}]
Apply Propositions \ref{PropDivisionContFrac} and \ref{proposition rho_n and C_n}. For even $n$, 
\begin{align*}
    \lambda_n(x) = \frac{x F_n(x^3)}{G_{n-1}(x^6)} = x\rho_n(x^{-3})
\end{align*}
and, for odd $n$,
\begin{align*}
    \begin{split}\lambda_n(x) = \frac{G_n(x^3)}{x^2F_{n-1}(x^6)} = \frac{G_n(x^3)}{x^2 x^{-3}F_{n-1}(x^6)} = x\rho_n(x^{-3})\end{split}.   \qedhere
\end{align*}\end{proof}
For another proof, one can also produce an explicit equivalence transformation of irregular continued fractions. To show that this property is truly inherited from $F$, $G$ and $H$, this method was chosen. As $\lambda_n(x)$ only converges along $n$ of the same parity for $|x| < 1$, an immediate consequence is obtained.
\begin{corollary}\label{corr different limits rho}
For all $|x| <1$, $(\rho_n(x))_{n=0}^{\infty}$ converges, and for $|x|>1$, $(\rho_n(x))_{n=0}^{\infty}$ is only parity partial convergent.
\end{corollary}
As already mentioned, $H$ and $I$ are cubic roots in $\mathbb{F}_2((x^{-1}))$, and Adamczewski \cite{adamczewski2010non} claims that $\lambda(x)$ is the unique root of $\lambda(x)^3 +x\lambda(x)^2+1$ in $\mathbb{F}_2((x^{-1}))$. Hence, by filling in Theorem \ref{theorem rho and c equal}, $\rho(x)$ is the unique root of $\rho(x)^3+\rho(x)^2+x=0$ over the same field. The symmetries between these polynomials coming from $H$ and $I$ and from $\rho$ and $\lambda$ are evident.\medskip

Then there is another helpful property of $\rho$:
\begin{prop}\label{propSumRhoTo2}
Let $|x| \ne 1$ and $n \ge 0$. Then $\rho_n(x) + \rho_n(-x) = 2$. In particular, if limits over the same parity are taken, then $\rho(x) + \rho(-x) = 2$.
\end{prop}
\begin{proof}
This is a simple writing-out exercise: $\rho_n(x) + \rho_n(-x) = 1 + \frac{x}{\rho_n(x^2)} + 1 + \frac{-x}{\rho_n(x^2)} = 2$.
\end{proof}
In the same spirit, a result for $\lambda$ can be obtained.
\begin{prop}\label{propSumCTo2z}
If $|x| \ne 1$ and $n \ge 0$, $\lambda_n(x) + \lambda_n(-x) = 2$. Particularly, if limits over the same parity are taken, then $\lambda(x) + \lambda(-x) = 2$.
\end{prop}
\begin{prop}\label{propproductrho}
 If $|x|<1$, $\prod_{k=0}^{\infty}\rho\big(x^{2^k}\big) = H(x)$ and $\prod_{k=0}^{\infty}\lambda\big(x^{2^k}\big) = I(x)$. 
\end{prop}
\begin{proof}
Define $f(x) := \prod_{k=0}^{\infty}\rho\big(x^{2^k}\big)$. Then $f(0) = 1$ and $$f(x) = \rho(x)\rho(x^2)f(x^4) = \bigg(1+\frac{x}{\rho(x^2)}\bigg)\rho(x^2)f(x^4) = (\rho(x^2)+x)f(x^4) = f(x^2) + xf(x^4).$$
As such, if $|x|<1$, $f(x) = H(x)$, and the case of $\lambda$ and $I$ follows the same argument.\end{proof}
On another note, Schwartz Reflection Principle \cite[Theorem IX.1.1]{lang2013complex} ensures that $H(x) = \overline{H(\overline{x})}$ within the unit disk, and so for all $|x|<1$, $\rho(x) = \overline{\rho(\overline{x})}$ as $\rho(x) = \frac{H(x)}{H(x^2)}$. To conclude this subsection, we include some numerical findings:
\begin{obs}
\begin{enumerate}
    \item Within the unit disk, $\rho$ has two simple zeros, which are near $-0.440049\pm0.65651142i$. As such, $H$ and $\rho$ have each two zeros, and $\rho$ has two poles within the unit disk.
    \item On the real interval $(0, 1)$, $\rho(\zeta_3x)$ is an oscillating function.
\end{enumerate}
\end{obs}
\subsection[The continued fractions rho and lambda evaluated at roots of unity]{The continued fractions $\boldsymbol{\rho}$ and $\boldsymbol{\lambda}$ evaluated at roots of unity}
In this subsection, we study the continued fractions $\rho$ and $\lambda$ at special roots of unity. As this is exactly the overlap of the closures of the domains where the limits as $n$ goes to infinity are unique and parity partial, these roots are interesting. Since $F$, $G$, $H$ and $I$ cannot be defined on the unit circle due to their Mahler equations, it is not expected that $\rho$ and $\lambda$ can be defined there. Yet, for example, both $\rho(1)$ and $\lambda(1)$ can be computed easily: The continued fraction $[1;\:1,\:1,\:1,\:1,\:\dots]$ approaches $\frac{1+\sqrt{5}}{2}$, the golden ratio.\medskip

Denote the root of unity $\exp(2\pi i \frac{a}{n})$ by $\zeta_{n}^a$ and all primitive $2^n$th  roots of unity by $X_n$: $$X_n := \{\zeta_{2^n}^a : 1 \le a \le 2^n \text{ and }a \text{ odd}\}.$$
The union of all these roots of unity is denoted by $X$: $$X := \bigcup_{n \ge 0}X_n.$$
The set $X$ consists exactly of all complex numbers $x$ such that $x^{2^k}  =1$ for sufficiently large $k$ and is a dense subset of the unit circle. For all $x\in X$, all values of $\rho(x)$ and $\lambda(x)$ have to do with $\sqrt{5}$:
\begin{prop}\label{prop rho lambda equal on X}
For all $x \in X$, $\rho(x)$ and $\lambda(x)$ are defined and $x\rho(x^{-3}) = \lambda(x)$. In particular, if $x = \exp(2 \pi i \frac{a}{2^n})$ with $a$ odd and $n\ge 0$, then $x$ lies in $\mathbb{Q}(\zeta_{2^n}, \sqrt{5})$ but, if $n \ge 2$, $x$ is not in $\mathbb{Q}(\zeta_{2^{n-1}}, \sqrt{5})$ nor in $\mathbb{Q}(\zeta_{2^n})$.
\end{prop}
\begin{proof}
For $n=0$, there is only one $n$th root of unity: 1, and
\begin{align*}
    \rho_n(1) = \lambda_n(1) = 1 + \overbrace{\pfrac{1}{1} + \pfrac{1}{1} +\dots+\pfrac{1}{1}}^{n-1\text{ terms}}.
\end{align*}
This continued fraction converges to $\frac{1+\sqrt{5}}{2}$, which verifies the statement for $n= 0$. Now let $x = \zeta_{2^n}^a$ with $a$ odd. Then the statement follows from induction on $n$, since
\begin{align*}
    \begin{split}
        \lambda(x) = x + \frac{1}{\lambda(x^2)} = x + \frac{1}{x^2\rho(x^{-6})}=  x\bigg(1 + \frac{x^{-3}}{\rho(x^{-6})}\bigg)  = x\rho(x^{-3}).
    \end{split} \qedhere 
\end{align*}\end{proof}
\begin{prop}
Let $n \ge 2$. Then $\sum_{x \in X_n} \rho(x) = 2^{n-1}$.
\end{prop}
\begin{proof}
Apply Proposition \ref{propSumRhoTo2} to get $2\sum_{x \in X_n} \rho(x) = \sum_{x \in X_n} \big(\rho(x)+\rho(-x)\big) = 2^{n-1}\cdot 2$.
\end{proof}

The maps $\rho$ and $\lambda$ have some more well-behaved properties when viewed as maps from $X_n$ to $\mathbb{Q}(\zeta_{2^n}, \sqrt{5})$. A little bit of Galois theory is required for the next part. Let $\Gal(L:K)$ denote the Galois group of the field extension $L : K$.
\begin{lemma}
For all $n \ge 0$, $\mathbb{Q}(\zeta_{2^n}) \cap \mathbb{Q}(\sqrt{5}) = \mathbb{Q}$.
\end{lemma}
\begin{proof}Because $2$ is the only ramified prime of $\mathcal{O}_{\mathbb{Q}(\zeta_{2^n})}$ \cite[Theorem 3.12]{stevenhagen2012number} and $5$ is the discriminant of $\mathbb{Q}(\sqrt{5})$ \cite[Exercise 4.9]{stevenhagen2012number} and so the only ramified prime of $\mathcal{O}_{\mathbb{Q}(\sqrt{5})}$ \cite[Theorem 4.14]{stevenhagen2012number}, we have $\mathcal{O}_{\mathbb{Q}(\sqrt{5})}\not\subset \mathcal{O}_{\mathbb{Q}(\zeta_{2^n})}$. Thus, $\sqrt{5} \notin \mathcal{O}_{\mathbb{Q}(\zeta_{2^n})}\subset\mathbb{Q}(\zeta_{2^n})$. Now, as $\#\Gal(\mathbb{Q}(\sqrt{5}):\mathbb{Q})=2$, $\mathbb{Q}(\sqrt{5})$ has only two subfields, $\mathbb{Q}$ and $\mathbb{Q}(\sqrt{5})$, and so $\mathbb{Q}(\zeta_{2^n}) \cap \mathbb{Q}(\sqrt{5}) = \mathbb{Q}$.
\end{proof}
\begin{prop}\label{propGalFiniteCase}
Let $\sigma \in \Gal(\mathbb{Q}(\zeta_{2^n}, \sqrt{5}) :\mathbb{Q}(\sqrt{5}))$ and $\bar{\sigma}$ its restriction to $\Gal(\mathbb{Q}(\zeta_{2^n}):\mathbb{Q})$. Then $\sigma(\rho(x)) = \rho(\bar{\sigma}(x))$ for all $x \in X$.
\end{prop}
\begin{proof}
As ${\sigma}$ acts trivial on $\mathbb{Q}(\sqrt{5})$ and is a ring homomorphism,
\begin{align*}
    \rho(\bar{\sigma}(x)) &= 1 + \pfrac{\bar{\sigma}(x)}{1} + \pfrac{\bar{\sigma}(x)^2}{1} + \pfrac{\bar{\sigma}(x)^4}{1} + \dots + \pfrac{\bar{\sigma}(x)^{2^{n-2}}}{1} + \pfrac{-1}{\phi}\\
    &= 1 + \pfrac{{\sigma}(x)}{1} + \pfrac{{\sigma}(x^2)}{1} + \pfrac{{\sigma}(x^4)}{1} + \dots + \pfrac{{\sigma}\big(x^{2^{n-2}}\big)}{1} + \pfrac{-1}{\sigma(\phi)}\\
    &= \sigma(\rho(x)).\qedhere
\end{align*}\end{proof}
\begin{corollary}
Let $\sigma \in \Gal(\mathbb{Q}(X, \sqrt{5}) : \mathbb{Q}(\sqrt{5}))$ and $\bar{\sigma}$ be its restriction to $\Gal(\mathbb{Q}(X) : \mathbb{Q})$. Then for all $x \in X$, $\sigma(\rho(x)) = \rho(\bar{\sigma}(x))$.
\end{corollary}
\begin{proof}
This follows from verifying the statement for all $x \in X_n$ on the restrictions $\sigma|_{\mathbb{Q}(\zeta_{2^n}, \sqrt{5}):\mathbb{Q}(\sqrt{5})}$ and $\bar{\sigma}|_{\mathbb{Q}(\zeta_{2^n}):\mathbb{Q}}$ and from Proposition \ref{propGalFiniteCase}.
\end{proof}
\begin{corollary}
Let $n \ge 0$ and $z \in X_n$. Then the Galois conjugates of $\rho(z)$ in $\mathbb{Q}(X_n, \sqrt{5}):\mathbb{Q}(\sqrt{5})$ are all in the set $\{\rho(x):x \in X_n\}$.
\end{corollary}
Thus, on $X$, the continued fraction $\rho$ has a pleasant form and is well-defined. From the identity $\lambda(x) = x\rho(x^{-3})$, it follows that similar results are valid for $\lambda(x)$. Then, as $X$ is dense on the unit circle, $\rho$ and $\lambda$ are well-defined on a dense subset of the unit circle.\medskip

On the other hand, Example \ref{counterexample C_n zeta_12 undefined} demonstrates that $\rho$ is not defined at all but finite points of $\zeta_9X$ and $\lambda$ at all but finite points of $\zeta_3X$. Therefore, both $\rho(x)$ and $\lambda(x)$ are also undefined on a dense subset of the unit circle.
\chapter{Folded continued fractions}
As mentioned in the previous chapter, $H$ and $\rho$ have an exceptional connection with continued fractions. Not only can $\rho$ be manipulated into an easy irregular continued fraction but, peculiarly, also into a predictable regular continued fraction. More specifically, it is an example of a so-called folded continued fraction. For comparison, the original folded continued fraction are given. Then the folded continued fraction of $\rho$ is presented and studied, and a few other examples of \emph{algebraic} functions are discussed. At the end of the chapter, two related results are discussed.
\section{An introduction to folded continued fractions}\label{Section Intro into folded cfs}
To be able to deal with folded continued fractions, more notation is required.
\begin{notation} Let $\boldsymbol{w} = [a_0;\: a_1,\: \dots,\: a_n]$ be a continued fraction. Then it can also be seen as a finite sequence, which will be called a \textbf{word}. Let $\boldsymbol{w}$ be such a word.
\begin{itemize}
    \item If $b$ is a number and $\boldsymbol{v}$ another word, then $[\boldsymbol{w},\: b]$ and $[\boldsymbol{w},\: \boldsymbol{v}]$ denote the concatenation of $\boldsymbol{w}$ with $b$ and $\boldsymbol{v}$, respectively.
    \item The word $-\boldsymbol{w}$ denotes the negation of $\boldsymbol{w}$. That is, $[-a_0;\:-a_{1},\:\dots,\:-a_n]$.
    \item The word $\overleftarrow{\boldsymbol{w}}$ denotes the reverse of $\boldsymbol{w}$. That is, $[a_n;\:a_{n-1},\:\dots,\:a_0]$.
    \item The length of $\boldsymbol{w}$ is denoted by $\#\boldsymbol{w}$ and is equal to $n+1$.
\end{itemize}
The \textbf{empty word} is the word containing no elements and is denoted by $[\:]$. It's length is 0.
\end{notation}
The following lemma will be used constantly:
\begin{lemma}\label{Lemma Negative Continued Fraction}
For a continued fraction $\boldsymbol{w} = [a_0;\: a_1,\: \dots,\: a_n]$, we have $[-a_0;\: -a_1,\: \dots, \:-a_n] = -[a_0;\: a_1,\: \dots,\: a_n]$ and $-\overleftarrow{\boldsymbol{w}} = \overleftarrow{-\boldsymbol{w}}$.
\end{lemma}
\begin{proof}
The first statement follows by induction and the second one by writing out the definition.
\end{proof}
Thus, for a continued fraction $\boldsymbol{w}$, writing $-\boldsymbol{w}$ is unambiguous.
\subsection{An example of a folded continued fraction}
In the late 1970s, Shallit published the first paper \cite{shallit1979simple} on the continued fraction of $\sum_{n=0}^{\infty} x^{-2^k}$ for integers $x$ greater or equal to $3$, which, strangely enough, contained almost only $\pm x$. In the 1980s and 1990s, further work was done by, amongst others, Shallit, van der Poorten, Mend\`ez France and Dekking. In 1992, Shallit and van der Poorten coined the name folded continued fraction. One central result is the so-called Folding Lemma:

\begin{theorem}[Folding Lemma \cite{borwein2014neverending}]
Let $n \ge 0$, $[a_0; a_1, a_2, \dots, a_n]$ be a continued fraction, $\boldsymbol{w} = [a_1, a_2, \dots, a_n]$, $t$ not zero, and $\begin{pmatrix}p_n & p_{n-1} \\q_n & q_{n-1}\end{pmatrix}$ the matrix of continuants of this continued fraction. Then $$p_{2n+1} = q_np_nt + (-1)^n, \quad q_{2n+1} = tq_n^2\quad\text{and}\quad[a_0; \boldsymbol{w}, t, -\overleftarrow{\boldsymbol{w}}] = \frac{p_n}{q_n} + \frac{(-1)^n}{tq_n^2}.$$
\end{theorem}

The simplest function that can be written as a folded continued fraction is $f(x) := x \sum_{n=0}^{\infty}x^{-2^n}$ due to Shallit and van der Poorten \cite{van1992folded}. It coincides with the continued fraction $\lim_{n \to \infty} [1;\: \boldsymbol{p}_n]$ where $\boldsymbol{p}_0$ is the empty word and $\boldsymbol{p}_{n} = [\boldsymbol{p}_{n-1},\: x,\: -\overleftarrow{\boldsymbol{p}_{n-1}}] $ for $n \ge 1$. Now observe that after each folding iteration, the continued fraction has length $2^n$, and that $p_{0} = q_{0} = 1$. Then, by induction and applying the Folding Lemma with $t = x$, it follows that for all $n \ge 1$,
\begin{align*}
    p_{2^n-1} &= xp_{m}q_{m} +(-1)^{2^n}= x\bigg(\sum_{i=0}^{n-1} x^{2^{n-1}-2^i}\bigg)x^{m-1} +1= \sum_{i=0}^{n} x^{2^{n}-2^i}\quad \text{and}\\
    q_{2^n-1} &= xq_{m}^2 = x\cdot x^{m-1} = x^{2^n-1},
\end{align*}
where $m = 2^{n-1}-1$. Then the continued fraction $[1;\: \boldsymbol{p}_n]$ approaches 
\begin{align*}
    \lim_{n \to \infty}\frac{p_{2^n-1}}{q_{2^n-1}} = \lim_{n \to \infty} \frac{\sum_{i=0}^{n} x^{2^{n}-2^i}}{x^{2^n-1}} = f(x).
\end{align*}
Note that $f$ is also a $2$-Mahler function as $f(x) = xf(x^2) + 1$.\medskip

The recursion for $\boldsymbol{p}_n$ starts with $\boldsymbol{p}_{n-1}$, making the first $2^{n-1}$ terms of $\boldsymbol{p}_n$ and $\boldsymbol{p}_{n-1}$ coincide. Thus, the sequence of words $\boldsymbol{p}_n$ converges to an infinite continued fraction as $n$ goes to infinity:
\begin{align*}
    \lim_{n \to \infty} \boldsymbol{p}_n = [x; x, -x, x, x, -x, -x, x, x, x, -x, -x, x, -x, -x, \dots].
\end{align*}
As all terms are either $x$ or $-x$, knowing the signs of the partial quotients is sufficient to understand the entire continued fraction. The corresponding sequence of signs is famous in popular mathematics: the (regular) \textbf{paperfolding sequence} or dragon curve sequence. Hence the name `folded continued fraction'. In Subsection \ref{subsection spec}, there is a method presented to remove the $-x$ from the continued fraction to achieve a regular continued fraction with only positive partial quotients.\medskip
\subsection{Paperfolding sequences as curves}
The paperfolding sequence is  better known as a fractal, graphically the \textbf{paperfolding dragon} or dragon curve, than as a Mahler function or automatic sequence. As the construction is unknown to many modern mathematicians \cite{tabachnikov2014dragon}, it is explained here again, following the paper of Tabachnikov \cite{tabachnikov2014dragon}.

\begin{recipe}

Take a strip of paper and fold it $n$ times in half in the same direction. After unfolding, $2^n-1$ hills and trenches appear. For example, the result after three folds is drawn in Figure \ref{figure paper folded three times}.\medskip

\begin{figure}[!ht]
    \centering
    \begin{tikzpicture}
        \draw (0.1, 0) -- (0.8, 0);
        \draw (0.8, 0) .. controls (0.95,0.075) and (0.95,0.075) ..  (1, 0.15);
        \draw (1.2, 0) .. controls (1.05,0.075) and (1.05,0.075) ..  (1, 0.15);
        \draw (1.2, 0) -- (1.8, 0);
        \draw (1.8, 0) .. controls (1.95,0.075) and (1.95,0.075) ..  (2, 0.15);
        \draw (2.2, 0) .. controls (2.05,0.075) and (2.05,0.075) ..  (2, 0.15);
        \draw (2.2, 0) -- (2.8, 0);
        \draw (2.8, 0) .. controls (2.95,-0.075) and (2.95,-0.075) ..  (3, -0.15);
        \draw (3.2, 0) .. controls (3.05,-0.075) and (3.05,-0.075) ..  (3, -0.15);
        \draw (3.2, 0) -- (3.8, 0);
        \draw (3.8, 0) .. controls (3.95,0.075) and (3.95,0.075) ..  (4, 0.15);
        \draw (4.2, 0) .. controls (4.05,0.075) and (4.05,0.075) ..  (4, 0.15);
        \draw (4.2, 0) -- (4.8, 0);
        \draw (4.8, 0) .. controls (4.95,0.075) and (4.95,0.075) ..  (5, 0.15);
        \draw (5.2, 0) .. controls (5.05,0.075) and (5.05,0.075) ..  (5, 0.15);
        \draw (5.2, 0) -- (5.8, 0);
        \draw (5.8, 0) .. controls (5.95,-0.075) and (5.95,-0.075) ..  (6, -0.15);
        \draw (6.2, 0) .. controls (6.05,-0.075) and (6.05,-0.075) ..  (6, -0.15);
        \draw (6.2, 0) -- (6.8, 0);
        \draw (6.8, 0) .. controls (6.95,-0.075) and (6.95,-0.075) ..  (7, -0.15);
        \draw (7.2, 0) .. controls (7.05,-0.075) and (7.05,-0.075) ..  (7, -0.15);
        \draw (7.2, 0) -- (7.9, 0);
    \end{tikzpicture}
    \caption{A strip of paper is folded three times and then unfolded.}
    \label{figure paper folded three times}
\end{figure}
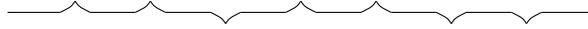
Associating the hills with a $1$ and the trenches with a $-1$ creates a word. For example, after three folds, this word is $[1,\:1,\:-1,\:1,\:1,\:-1,\:-1]$. The paperfolding sequence is also the limit sequence when iterating indefinitely. Recall that the recursive rule of the words is $\boldsymbol{p}_n = [\boldsymbol{p}_{n-1},\: 1,\: -\overleftarrow{\boldsymbol{p}_{n-1}}]$. The first $\boldsymbol{p}_{n-1}$ is the bottom of the paper that remains at the same position during the fold, the 1 is the fold, and the second $\boldsymbol{p}_{n-1}$ is the half of the paper that is folded on top. The minus sign comes from flipping that half upside down and the reverse arrow from laying it in the opposite direction. Thus, the two constructions coincide. When laying the hills and trenches of the folds in right angles, a shape emerges. For example, after three folds, Figure \ref{figure curve three folds} is obtained.\medskip

\begin{figure}[!ht]
    \centering
    \begin{tikzpicture}[scale = 0.5]
        \draw[draw = red, fill=red] (0,0) circle [radius=0.1];
        \draw (0, 0) -- (0, 1);
        \draw (0, 1) -- (1, 1);
        \draw (1, 1) -- (1, 0);
        \draw (1, 0) -- (2, 0);
        \draw (2, 0) -- (2, -1);
        \draw (2, -1) -- (1, -1);
        \draw (1, -1) -- (1, -2);
        \draw (1, -2) -- (2, -2);
    \end{tikzpicture}
    \caption{A strip of paper is folded three times and then unfolded, starting at the red dot.}
    \label{figure curve three folds}
\end{figure}
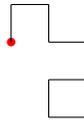
For each number of folds, one can make such a graph, and surprisingly, it is never \textbf{self-crossing}. That is, not a single line is drawn twice, and the curve never self-intersects. If one takes the limit of these folds, infinitely many lines are drawn, and if one takes four of these collections of lines, it covers the entire grid $\mathbb{Z}^2$. Thus, each line between two neighbouring points is covered exactly once. In Figure \ref{fig:Curve5-9-17}, the paperfolding dragon is shown after 5, 9 and 17 iterations, respectively. Due to the unusual shape, this curve was named the dragon curve. Such drawings are called \textbf{folding curves}.
\begin{figure}[H]
\centering
  \includegraphics[width=150mm]{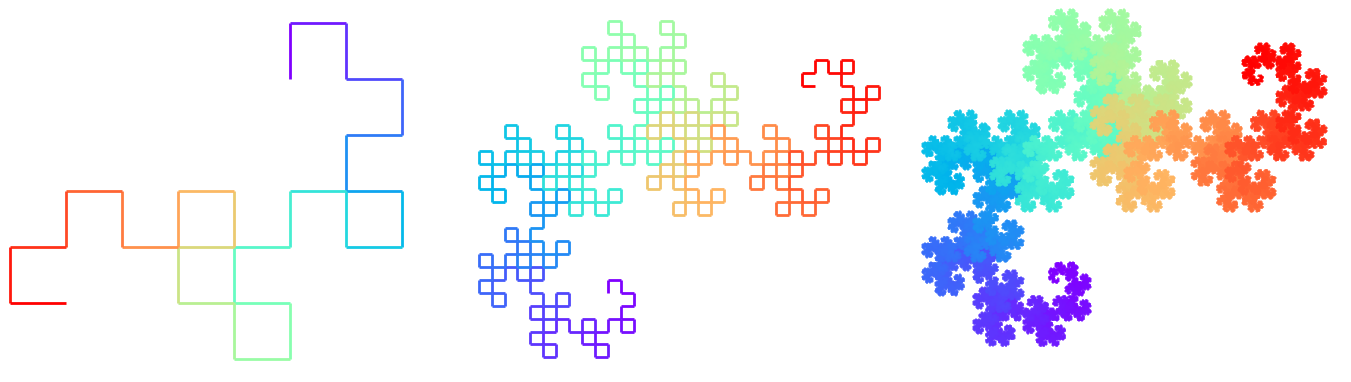}
  \caption{The dragon curve after 5, 9 and 17 iterations, respectively. It starts purple and slowly fades to red following the colours of the rainbow.}
  \label{fig:Curve5-9-17}
\end{figure}
\end{recipe}
\section[The folded continued fraction of rho]{The folded continued fraction of $\boldsymbol{\rho}$}
As seen for the regular paperfolding sequence, Mahler functions and folded continued fractions can sometimes live in the same world. So far, only infinite sums and products seem to have been studied in the literature, such as $\sum_{n=0}^{\infty}2^{-F_n}$ where $F_n$ are the Fibonacci numbers \cite{borwein2014neverending}. Because $\rho(x)$ is defined as an irregular continued fraction, such a regular folded continued fraction is yet unknown.

\subsection[Computing a folded continued fraction for rho]{Computing a folded continued fraction for $\boldsymbol{\rho}$}
For $\rho(x)$, a folded continued fraction exists with a more intricate fold. First, recall that 
\begin{align*}
    \rho_n(x) := 1+\pfrac{x}{1}+\pfrac{x^{2}}{1}+\pfrac{x^{4}}{1} + \dots +\pfrac{x^{2^{n-2}}}{1} + \pfrac{x^{2^{n-1}}}{1}.
\end{align*}
Evaluating the regular continued fractions of these rational functions $\rho_n$ in \textit{Magma} \cite{Magma} gives
\begin{align*}
    \rho_0(x) &= [1],\\
    \rho_1(x) &= [1+x],\\
    \rho_2(x) &= [1;\: x, \:x],\\
    \rho_3(x) &= [x+1;\: -x,\: -x,\: x,\: x],\\
    \rho_4(x) &= [ 1;\:x,\:x,\:x,\:-x,\: -x,\:x,\:-x,\:-x,\:x,\:x],\\
    \rho_5(x) &= [x + 1; -x,-x,x,x,-x,-x,-x, x,x,-x,x,x,x,-x,-x,x,-x,-x,x,x].
\end{align*}
It is already remarkable that almost all partial quotients are $\pm x$. Moreover, the words $\boldsymbol{w}_n$ seem to converge along even indices $n$ and along odd indices $n$. The parity partial convergence of $\rho(x)$ in Corollary \ref{corr different limits rho} for $|x| > 1$ explains the existence of the two distinct continued fractions. A careful look suggests that $\rho(x) = [s_n(x);\: \boldsymbol{w}_n(x)]$ with the recursion
\begin{align}\label{recursion w_n}
    \boldsymbol{w}_{n} =  [\boldsymbol{w}_{n-2},\: (-1)^nx,\: - \overleftarrow{\boldsymbol{w}_{n-2}}, \:(-1)^nx,\: \boldsymbol{w}_{n-1}]
\end{align}
for $\boldsymbol{w}_n = \boldsymbol{w}_n(x)$, where
\begin{align}\label{recursion s_n}
    s_n(x) = \begin{cases}
        1       &\text{if }n \text{ even}, \\
        x +1    &\text{if }n \text{ odd}.
    \end{cases}
\end{align}
This form resembles the folding of $\boldsymbol{p}$. To begin proving the observed recursions \eqref{recursion w_n} and \eqref{recursion s_n}, a few technical lemmas need to be set.
\begin{lemma}\label{lemma H_n nonlinear relation}
For all $n \ge 1$,
\begin{align*}
H_{n-2}(x^2)H_n(x)-H_{n-1}(x)H_{n-1}(x^2) = (-1)^{n-1}x^{2^n-1}.
\end{align*}
\begin{proof}
Use induction. For $n = 1$, the left side evaluates to $1\cdot(x+1) - 1\cdot 1 = x$, which is correct. Now assume that $n \ge 2$. Using $H_n(x) = H_{n-1}(x^2) + xH_{n-2}(x^4)$ for all $n \ge 1$, we obtain
\begin{align*}
& H_{n-2}(x^2)H_n(x)-H_{n-1}(x)H_{n-1}(x^2)\\
   &\quad =H_{n-2}(x^2)\big(H_{n-1}(x^2)+qH_{n-2}(x^4)\big)-\big(H_{n-2}(x^2)+qH_{n-3}(x^4)\big)H_{n-1}(x^2)\\
    &\quad =xH_{n-2}(x^2)H_{n-2}(x^4)-xH_{n-3}(x^4)H_{n-1}(x^2) = -x (-1)^{n-2}(x^2)^{2^{n-1}-1} = (-1)^{n-1}x^{2^n-1}.\qedhere
\end{align*}\end{proof}
\end{lemma}
\begin{lemma}\label{Lemma Hn Combinatorial}
For all $n \ge 1$,
\begin{align*}
    H_n(x) &= H_{n-1}(x) +x^{2^{n-1}}H_{n-2}(x).
\end{align*}
\end{lemma}
\begin{proof}
Recall that $H_n(x) = \sum_{m = 0}^{2^{n}-1}c_mq^m$, where $c_m = 0$ if the binary expansion of $m$ contains two neighbouring 1's and $c_m=1$ otherwise. Let $m< 2^n$ such that $c_m = 1$. If $m < 2^{n-1}$, then $c_m$ is already in $H_{n-1}$. If $m \ge 2^{n-1}$, the binary expansion of $m$ starts with $10$, and removing these terms gives a number below $2^{n-2}$ corresponding to $H_{n-2}$. This leads to the formula.
\end{proof}
Let $(k_n)_{n=0}^{\infty}$ be the sequence defined by $k_0 = k_1 = 0$ and $k_n = 2k_{n-2} + k_{n-1} + 2$ for all $n\ge 2$.
\begin{lemma}\label{Lemma defining r_n} For all $n \ge 0$, $k_n = \frac{2^{n+1}\pm(-1)^n}{3} - 1$ and $k_n$ is even.
\end{lemma}
\begin{proof}
Both the closed formula and each term being an even integer follow from induction.
\end{proof}
The terms $k_n+1$ form the Jacobsthal sequence, sequence A001045 on the OEIS \cite{OEIS}. This sequence is also relevant to the construction of $F$ and $G$ using Stern polynomials, determines where the long sequences of zero coefficients in their expansions begin and end, and to the degree of $H_n(x)$. This is not coincidental. By construction, $k_{n} = \deg{H_n(x^2)}$, which is the degree of the denominator of $\rho_n(x)$. At this point, only the main theorem remains to be proven.
\begin{theorem}\label{theorem Continued Fraction rho x, -x}
Let $n \ge 0$ and $x$ be a real number. Then $\rho_n(x)$ is expressed as a regular continued fraction by writing $\rho_n(x) = [s_n ;\: \boldsymbol{w}_n]$ for
\begin{align*}
    s_n = \begin{cases}
    1 & \text{if }n \text{ even},\\
    1+x & \text{if }n \text{ odd},
    \end{cases}
\end{align*}
with $\boldsymbol{w}_n$ defined recursively by
\begin{align*}
    \boldsymbol{w}_{n} =
    [\boldsymbol{w}_{n-2},\: (-1)^nx, \:- \overleftarrow{\boldsymbol{w}_{n-2}}, \:(-1)^nx,\:\boldsymbol{w}_{n-1}] \quad\text{for $n \ge 2$},
\end{align*}
$\boldsymbol{w}_{0} = \boldsymbol{w}_{1}$ are empty, the length of $\boldsymbol{w}_n$ is equal to $k_n$, $p_{k_n} = H_n(x)$ and $q_{k_n} = H_{n-1}(x^2)$.
\end{theorem}
\begin{proof}
The theorem holds true for $n = 0$ and $1$, and then apply induction. Assume that $n \ge 2$. From the recursion, the length of $\boldsymbol{w}_n$ is equal to twice the length of $\boldsymbol{w}_{n-2}$, the length of $\boldsymbol{w}_{n-1}$ and 2 added together. Thus, the length of $\boldsymbol{w}_n$ is indeed $k_n$.\medskip

Now let $l = k_{n-1}$ and $m = k_{n-2}$. Verifying that $\boldsymbol{w}_n$ satisfies this recursion is the most difficult part. Start along the lines of the proof of the Folding Lemma in \cite[Lemma 6.3]{borwein2014neverending}. Define
\begin{align*}
    \rho_{n-1} = [a_0;\: a_1,\: a_2,\:\dots,\:a_l] \quad\text{and}\quad \rho_{n-2} = [b_0;\: b_1,\: b_2,\:\dots,\:b_m],
\end{align*}
where $a_0 = s_{n-1}$ and $b_0 = s_{n-2}$ such that
\begin{align*}
    \boldsymbol{w}_{n-1} = [a_1;\: a_2,\:a_3,\:\dots,\:a_l]\quad\text{and}\quad \boldsymbol{w}_{n-2} = [b_1;\: b_2,\: b_3,\:\dots,\:b_m].
\end{align*}
Then the Key Lemma (Lemma \ref{lemma key lemma continued fractions}) gives
\begin{align*}
    \begin{pmatrix}
    a_0 & 1 \\ 1 & 0
    \end{pmatrix}
        \begin{pmatrix}
    a_1 & 1 \\ 1 & 0
    \end{pmatrix}
    \dots
    \begin{pmatrix}
    a_l & 1 \\ 1 & 0
    \end{pmatrix}
    =
    \begin{pmatrix}
    p_{l} & p_{l-1} \\ q_{l} & q_{l-1}
    \end{pmatrix}.
\end{align*}
Set $t = (-1)^nx$, and compute $[\boldsymbol{w}_{n-2},\: t,\: - \overleftarrow{\boldsymbol{w}_{n-2}}]$ by first dealing with $[- \overleftarrow{\boldsymbol{w}_{n-2}},\: -b_0]$ as follows:
\begin{align*}
    &\begin{pmatrix}
    -b_m & 1 \\ 1 & 0
    \end{pmatrix}\dots
    \begin{pmatrix}
    -b_1 & 1 \\ 1 & 0
    \end{pmatrix}
    \begin{pmatrix}
    -b_0 & 1 \\ 1 & 0
    \end{pmatrix}=
    \Bigg(\begin{pmatrix}
    -b_0 & 1 \\ 1 & 0
    \end{pmatrix}
    \begin{pmatrix}
    -b_1 & 1 \\ 1 & 0
    \end{pmatrix}
    \dots
    \begin{pmatrix}
    -b_m & 1 \\ 1 & 0
    \end{pmatrix} 
    \Bigg)^T \\&\quad =
        \begin{pmatrix}
    (-1)^{m+1}p_m & (-1)^mq_m \\ (-1)^mp_{m-1} & (-1)^{m-1}q_{m-1}
    \end{pmatrix}
    =\begin{pmatrix}
    -p_m & q_m \\ p_{m-1} & -q_{m-1}
    \end{pmatrix},
\end{align*}
because $m = k_{n-1}$ is even by Lemma \ref{Lemma defining r_n}. Then $-\overleftarrow{\boldsymbol{w}_{n-2}}$ is computed as
\begin{align*}
    \begin{pmatrix}
    -b_m & 1 \\ 1 & 0
    \end{pmatrix}
    \dots
    \begin{pmatrix}
    -b_1 & 1 \\ 1 & 0
    \end{pmatrix}= 
    \begin{pmatrix}
    -p_m & q_m \\ p_{m-1} & -q_{m-1}
    \end{pmatrix}
        \begin{pmatrix}
    0 & 1 \\ 1 & b_0
    \end{pmatrix} =
    \begin{pmatrix}
    q_m & b_0q_m-p_m  \\
    -q_{m-1}& p_{m-1} -b_0q_{m-1}
    \end{pmatrix}.
\end{align*}
By the Key Lemma for $[\boldsymbol{w}_{n-2},\: t,\: - \overleftarrow{\boldsymbol{w}_{n-2}}]$, we obtain
\begin{align*}
    &    \begin{pmatrix}
    b_0 & 1 \\ 1 & 0
    \end{pmatrix}
    \dots
    \begin{pmatrix}
    b_m & 1 \\ 1 & 0
    \end{pmatrix}
    \begin{pmatrix}
    t & 1 \\ 1 & 0
    \end{pmatrix}
        \begin{pmatrix}
    -b_m & 1 \\ 1 & 0
    \end{pmatrix}
    \dots
    \begin{pmatrix}
    -b_1 & 1 \\ 1 & 0
    \end{pmatrix}\\
    &\quad =
    \begin{pmatrix}
    p_m & p_{m-1} \\ q_m & q_{m-1}
    \end{pmatrix}
    \begin{pmatrix}
    t & 1 \\ 1 & 0
    \end{pmatrix}
    \begin{pmatrix}
    q_m & b_0q_m-p_m  \\
    -q_{m-1}& p_{m-1} -b_0q_{m-1}
    \end{pmatrix}\\
    &\quad =
    \begin{pmatrix}
    tp_m+p_{m-1}& p_m \\ tq_m+q_{m-1} & q_m
    \end{pmatrix}
    \begin{pmatrix}
    q_m & b_0q_m-p_m  \\
    -q_{m-1}& p_{m-1} -b_0q_{m-1}
    \end{pmatrix}\\
    &\quad =
    \begin{pmatrix}
    (tp_m+p_{m-1})q_m-q_{m-1}p_m& 
    (tp_m+p_{m-1})(b_0q_m-p_m) + p_m(p_{m-1} -b_0q_{m-1})\\
    (tq_m+q_{m-1})q_m-q_{m-1}q_m &
    (tq_m+q_{m-1})(b_0q_m-p_m) +q_m(p_{m-1} -b_0q_{m-1})
    \end{pmatrix}\\
    &\quad =
    \begin{pmatrix}
    tp_mq_m+(-1)^m & tb_0p_mq_m - tp_m^2 +(-1)^mb_0\\
    tq_m^2 & tb_0q_m^2 - tq_mp_m+1
    \end{pmatrix},
\end{align*}
where again $(-1)^m = 1$. Now the Key Lemma is applied to $[t,\: \boldsymbol{w}_{n-1}]$ to deduce
\begin{align*}
    \begin{split}
    &\begin{pmatrix}
    t & 1 \\ 1 & 0
    \end{pmatrix}
    \begin{pmatrix}
    a_1 & 1 \\ 1 & 0
    \end{pmatrix}
        \begin{pmatrix}
    a_2 & 1 \\ 1 & 0
    \end{pmatrix}
    \dots
    \begin{pmatrix}
    a_l & 1 \\ 1 & 0
    \end{pmatrix}
    =
    \begin{pmatrix}
    t & 1 \\ 1 & 0
    \end{pmatrix}
    \begin{pmatrix}
    a_0 & 1 \\ 1 & 0
    \end{pmatrix}^{-1}
    \begin{pmatrix}
    p_l & p_{l-1} \\ q_l & q_{l-1}
    \end{pmatrix}\\
    &\quad =
    \begin{pmatrix}
    t & 1 \\ 1 & 0
    \end{pmatrix}
        \begin{pmatrix}
    0 & 1 \\ 1 & -a_0
    \end{pmatrix}
    \begin{pmatrix}
    p_l & p_{l-1} \\ q_l & q_{l-1}
    \end{pmatrix}
    =
    \begin{pmatrix}
    1 & t-a_0 \\ 0 & 1
    \end{pmatrix}
    \begin{pmatrix}
    p_l & p_{l-1} \\ q_l & q_{l-1}
    \end{pmatrix}=            
    \begin{pmatrix}
    p_l+q_l(t-a_0) & *\\ q_l & *
    \end{pmatrix}.
    \end{split}
\end{align*}
Here, the $*$-terms denote expressions that do not matter for the end result. Note that $s_{n-1} - s_{n-2} = x$ for even $n$ and $s_{n-1} - s_{n-2} = -x$ for odd $n$. It means that $s_{n-1} - s_{n-2} = a_0 - b_0 = (-1)^nx=t$, and hence $t - a_0 = -b_0$. Now the entire recursion $[\boldsymbol{w}_{n-2},\: (-1)^nx,\: - \overleftarrow{\boldsymbol{w}_{n-2}} ,\:(-1)^nx,\:\boldsymbol{w}_{n-1}]$ is computed by multiplying the matrices found for $[\boldsymbol{w}_{n-2},\: t,\: - \overleftarrow{\boldsymbol{w}_{n-2}}]$ and $[t,\: \boldsymbol{w}_{n-1}]$:
\begin{align*}
   &\begin{pmatrix}
    tp_mq_m+1 & tb_0p_mq_m-tp_m^2 +b_0\\
    tq_m^2 & tb_0q_m^2-tq_mp_m+1
    \end{pmatrix}
    \begin{pmatrix}
    p_l-q_lb_0& *\\ q_l & *
    \end{pmatrix}\\
    &\quad =
    \begin{pmatrix}
    ( tp_mq_m+1)(p_l-q_lb_0) +(tp_mq_mb_0 - tp_m^2 + b_0)q_l &*\\
    tq_m^2(p_l - q_lb_0) + (tq_m^2b_0 - tq_mp_m + 1)q_l & *
    \end{pmatrix}\\
    &\quad =
    \begin{pmatrix}
    p_l+tp_mq_mp_l -tp_m^2q_l&*\\
    q_l+tq_m^2p_l-tq_mp_mq_l&*
    \end{pmatrix} =     \begin{pmatrix}
    p_l+(-1)^nxp_m(q_mp_l -p_mq_l)&*\\
    q_l+(-1)^nxq_m(q_mp_l-p_mq_l)&*
    \end{pmatrix}.
\end{align*}
As $l = k_{n-1}$ and $m = k_{n-2}$, the induction hypothesis gives that $p_l = H_{n-1}(x)$, $p_m = H_{n-2}(x)$, $q_l = H_{n-2}(x^2)$ and $p_l = H_{n-3}(x^2)$. Since $n \ge 2$, one can apply Lemma \ref{lemma H_n nonlinear relation} for $n-1$:
\begin{align*}
    q_mp_l-p_mq_l = H_{n-3}(x^2)H_{n-1}(x)-H_{n-2}(x)H_{n-2}(x^2) = -(-1)^{n-1}x^{2^{n-1}-1} = (-1)^nx^{2^{n-1}-1}.
\end{align*}
By Lemma \ref{Lemma Hn Combinatorial},
\begin{align*}
    p_{k_n} = p_l+(-1)^nxp_m(q_mp_l-p_mq_l) &= H_{n-1}(x) + x^{2^{n-1}}H_{n-2}(x) = H_n(x)\quad \text{and}\\
    q_{k_n} = q_l+(-1)^nxq_m(q_mp_l-p_mq_l) &= H_{n-2}(x^2) + x^{2^{n-1}}H_{n-3}(x^2) = H_{n-1}(x^2).
\end{align*}
Thus, the formulas for $p_{k_n}$ and $q_{k_n}$ are valid, and by Proposition \ref{PropDivisionContFrac}, $\rho_n = \frac{p_{k_n}}{q_{k_n}}$. This means that the recursion for $\boldsymbol{w}_n$ and $s_{n-2} = s_n$ are valid.
\end{proof}
\subsection[Specializing the folded continued fraction for rho]{Specializing the folded continued fraction for $\boldsymbol{\rho}$}\label{subsection spec}
For the regular finite continued fraction of a rational number, $[a_0; a_1,\dots, a_l]$, $a_i$ is commonly a positive integer for all $ 1 \le i \le l$. The same holds for all $i \ge 1$ for an infinite continued fraction $[a_0;a_1,a_2,\dots]$ as then there is a one-to-one correspondence between the infinite continued fractions and irrational numbers. Theorem \ref{theorem Continued Fraction rho x, -x} does not give this traditional shape for integer values $x$ as $-x$ shows up too. But one can specialise it into the preferred form as described in \cite{borwein2014neverending} with the following helpful lemma.
\begin{lemma}\label{lemma Ripple 1}
For all $x$, $y$ and $z$ such that both sides of the equation make sense, 
\begin{align*}
    [x;\: -y,\:z] = [x - 1;\: 1,\: y-1,\:-z].
\end{align*}\end{lemma}
\begin{proof}By the Key Lemma, and the matrices of partial convergents of the two continued fractions are
\begin{align*}
    \begin{pmatrix}
    x & 1 \\ 1 & 0 
    \end{pmatrix}
        \begin{pmatrix}
    -y & 1 \\ 1 & 0 
    \end{pmatrix}
        \begin{pmatrix}
    z & 1 \\ 1 & 0 
    \end{pmatrix} &= 
    \begin{pmatrix}
    x+z-xyz & 1-xy \\ 1-yz & -y
    \end{pmatrix} \quad \text{and}\\
    \begin{pmatrix}
    x - 1 & 1 \\ 1 & 0 
    \end{pmatrix}
        \begin{pmatrix}
    1 & 1 \\ 1 & 0 
    \end{pmatrix}
        \begin{pmatrix}
    y - 1 & 1 \\ 1 & 0 
    \end{pmatrix}
        \begin{pmatrix}
    -z & 1 \\ 1 & 0 
    \end{pmatrix} &=  \begin{pmatrix}
    x+z-xyz & xy-1 \\ 1-yz & y
    \end{pmatrix}.
\end{align*}
As the two left entries of the matrices coincide, the lemma follows.
\end{proof}
Now the regular continued fraction can be constructed. A continued fraction $[a_0; a_1, \dots, a_l]$ such that for each $1 \le i \le l$, $a_i \in \{-x,\: x\}$ can be written as a regular continued fraction that only contains $a_0, a_0-1, 1, x-2, x-1$ and $x$ by using Lemma \ref{Lemma Negative Continued Fraction} and Lemma \ref{lemma Ripple 1} repeatedly. For example,
\begin{align*}
    \rho_4(x) &= [ 1;\:x,\:x,\:x,\:-x,\: -x,\:x,\:-x,\:-x,\:x,\:x]\\
    &= [ 1;\:x,\:x,\:x-1,\:1,\:x-1,\:-[-x,\:x,\:-x,\:-x,\:x,\:x]]\\
    &= [ 1;\:x,\:x,\:x-1,\:1,\:x-1,\:x,\:-x,\:x,\:x,\:-x,\:-x]\\
    &= [ 1;\:x,\:x,\:x-1,\:1,\:x-1,\:x-1,\:1,\:x-1,\:-x,\:-x,\:x,\:x]\\
    &= [ 1;\:x,\:x,\:x-1,\:1,\:x-1,\:x-1,\:1,\:x-2,\:1,\:x-1,\:x,\:-x,\:-x]\\
    &= [ 1;\:x,\:x,\:x-1,\:1,\:x-1,\:x-1,\:1,\:x-2,\:1,\:x-1,\:x-1,\:1,\:x-1,\:x].
\end{align*}
This process can also be generalised in three steps:
\begin{enumerate}
    \item For all $0 \le i \le l - 1$ such that $a_i \ne a_{i+1}$, insert a 1 between $a_i$ and $a_{i+1}$.
    \item Replace every $-x$ by $x$.
    \item For each term $x$ in the new sequence, subtract $1$ for each neighbour being equal to $1$.
\end{enumerate}
The length of the continued fraction increases by the number of sign changes in the original sequence, which is bounded by $l-1$. Thus, the new continued fraction has less than $2l$ partial quotients. Also note that for all $m \ge 3$, this method gives a `proper' continued fraction for $\rho_n(m)$ that only contains $1, m-2, m-1$ and $m$. In other words, the two parity partial convergents of the continued fraction of $\rho(m)$ are all among four possible numbers. For example, $m = 5$ gives two continued fractions
\begin{align*}
    [1, 5, 5, 4, 1, 4, 4, 1, 3, 1, 4, 4, 1, 4, 5, 4, 1, 4, 4, 1,\dots]& \quad \text{for even $n$ and} \\
    [5; 1, 4, 4, 1, 4, 4, 1, 4, 5, 4, 1, 4, 4, 1, 3, 1, 4, 5, 4,\dots]& \quad \text{for odd $n$.}
\end{align*}
\subsection[The folding curve of rho]{The folding curve of $\boldsymbol{\rho}$}\label{paperfolding curve rho subsection}
The construction of the folded continued fraction of $\rho$ shares many features with Shallit's and van der Poorten's example discussed in Section \ref{Section Intro into folded cfs}. Hence it is natural to investigate whether this new folded continued fraction shares other properties with the paperfolding sequence. In this subsection, the folding curves of $\rho$ are tackled. As $\rho$ has two convergents, there are two different curves. In Figure \ref{fig:Rho 6,10,14}, a few examples are present.

\begin{figure}[H]
 \centering
  \includegraphics[width=140mm]{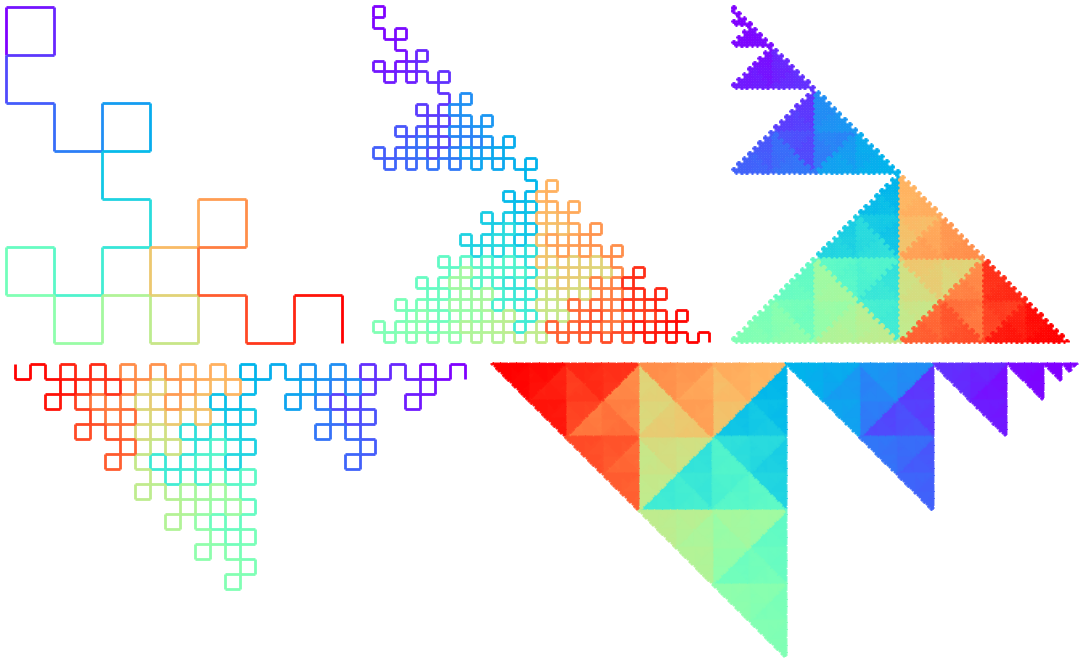}
  \caption{The curves derived from $\rho$ for several values of $n$. On top, after 6, 10 and 14 iterations, respectively, and, on the bottom after 7 and 13 iterations. It starts purple and slowly fades to red following the colours of the rainbow.}
  \label{fig:Rho 6,10,14}
\end{figure}
These curves resemble right isosceles triangles that increase in size. They are not perfect triangles, as the diagonals contain `hooks'. These hooks are small, local features that do not appear on a global scale. At the limit, they resemble true triangles. For one parity of $n$, the side lengths of the triangles grow by a factor of two and one triangle is added at the iteration $n+2$. The side lengths of the triangles of opposite parity differ by $\sqrt{2}$.

\begin{figure}[H]
\centering
  \includegraphics[width=90mm]{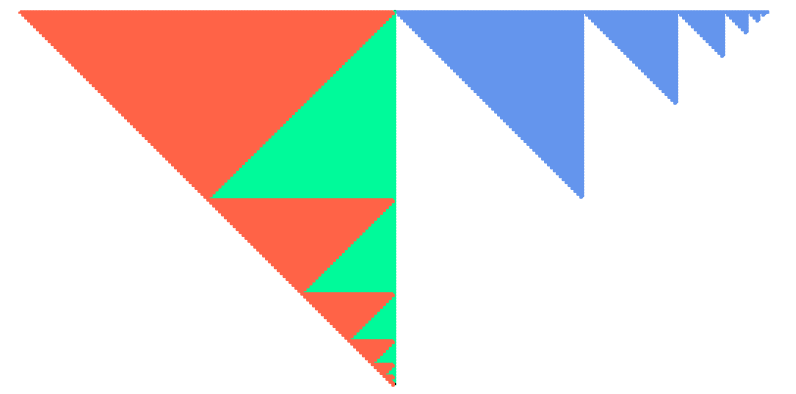}
  \caption{The curve for $\rho_{15}$. Blue is the term $\boldsymbol{w}_{13}$, green is $-\protect\overleftarrow{\boldsymbol{w}_{13}}$ and red is $\boldsymbol{w}_{14}$. Thus, the folding curve  first goes to the right, then up, and the last term completes the largest triangle.}
  \label{fig:Rho fold}
\end{figure}
Computational evidence suggests that these curves are not self-crossing for any $n \ge 0$. In Figure \ref{fig:Rho fold}, the way the fold works is drawn. The pattern is clear.
\begin{obs}
For all $n \ge 0$, the curve induced by the parity of the partial quotients of $\rho$ is not self-crossing. Moreover, by taking eight copies of the two limits and rotating and flipping them, the entire $\mathbb{Z}^2$ grid can be covered.
\end{obs}

The appearance of triangles is not surprising as the folding curve of the recursion defined by $\boldsymbol{q}_0 = [\:]$ and $\boldsymbol{q}_n = [\boldsymbol{q}_{n-1},\: (-1)^nx,\: -\overleftarrow{\boldsymbol{q}_{n-1}}]$ converges to a single triangle. Similarly, recursions giving gradually growing dragon curves exist, for example, the folding curve of the recursion defined by $\boldsymbol{v}_0 = \boldsymbol{v}_1 = [\:]$ and $[\boldsymbol{v}_{n-2},\: x, \:- \overleftarrow{\boldsymbol{v}_{n-2}} ,\:(-1)^nx,\:\boldsymbol{v}_{n-1}]$. Both of these folding curves are presented in Figure \ref{fig:alternative fold}. Thus, there is a strong relation between the folding curves of the paperfolding sequence and $\boldsymbol{w}$.
\begin{figure}[H]
\centering
  \includegraphics[width=120mm]{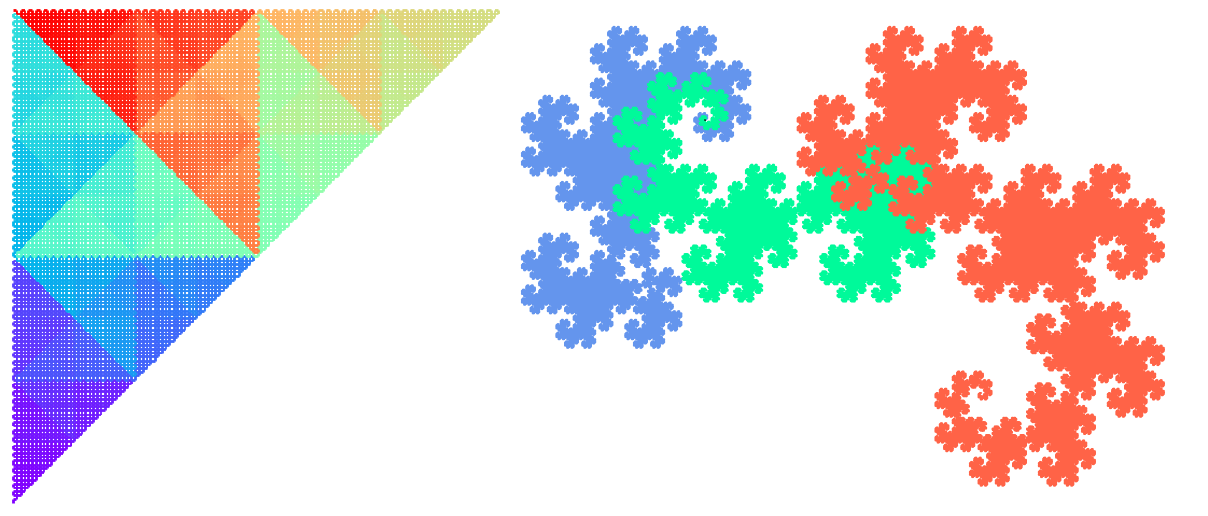}
  \caption{On the left is the curve for $\boldsymbol{q}$ depicted after 14 iterations. It starts purple and slowly fades into red following the colours of the rainbow. On the right is the curve for $\boldsymbol{v}$ after 17 iterations. Blue is the term $\boldsymbol{v}_{n-2}$, green is $-\protect\overleftarrow{\boldsymbol{v}_{n-2}}$ and red is $\boldsymbol{v}_{n-1}$. The red and green pieces together form a mirror paperfolding dragon, just like the red and green curve in Figure \ref{fig:Rho fold} form a triangle.}
  \label{fig:alternative fold}
\end{figure}
To strengthen this connection even further, both $\boldsymbol{q}_n$ and $\boldsymbol{w}_n$ are also observed to be a folded continued fraction of certain functions, respectively, of
\begin{align*}
    x\sum_{n=0}^{\infty} (-1)^nx^{2^n} \quad \text{and} \quad 1 - \pfrac{x}{1}- \pfrac{x^2}{1}- \pfrac{x^4}{1}- \pfrac{x^8}{1} - \dotsm.
\end{align*}
\subsection[The folding sequence of rho as a Mahler function]{The folding sequence of $\boldsymbol{\rho}$ as a Mahler function}\label{Subsection where P(x) is introduced}
The folding recursion for $\boldsymbol{w}$ gives not much structure on the two intertwining sequences consisting solely of $\pm 1$. The Toeplitz construction is another method to generate the paperfolding sequence and also gives a Mahler equation for its generating function. In this subsection, we show that $\boldsymbol{w}$ posesses the same structure. The two main strategies to generate the paperfolding sequence are \cite[Example 5.1.6 and Exercise 5.7]{allouche2003automatic}:
\begin{enumerate}
    \item \textbf{The folding approach.} Let $\boldsymbol{p}_0 = [\:]$ and $\boldsymbol{p}_n = [\boldsymbol{p}_{n-1},\: 1,\:  \overleftarrow{\boldsymbol{p}_{n-1}}]$ for $n \ge 1$. Then the paperfolding sequence is $\lim_{n \to \infty} \boldsymbol{p}_n$. This definition is used to prove the paperfolding lemma.
    \item \textbf{The Toeplitz construction.} Let $\boldsymbol{p}_0 = [\:]$ again be the empty word and for $n\ge 1$,
    \begin{align*}
        \boldsymbol{p}_n = [1,\: \boldsymbol{p}_{n-1}[0],\: -1,\: \boldsymbol{p}_{n-1}[1],\: 1, \boldsymbol{p}_{n-1}[2],\: -1,\: \boldsymbol{p}_{n-1}[3],\:1,\:\boldsymbol{p}_{n-1}[4],\: \dots].
    \end{align*}
    Here, $\boldsymbol{p}_{n-1}[m]$ is the $m$th element of $\boldsymbol{p}_{n-1}$, and the $1$ and $-1$ alternate.
\end{enumerate}
To prove these definitions are equivalent, first show by induction that in the first approach, the subsequence of terms of even index alternate. That is, $\boldsymbol{p}_n[0], \boldsymbol{p}_n[2], \boldsymbol{p}_n[4], \dots$ are $1,-1,1,-1,1,\dots$. Then prove the subsequences $\boldsymbol{q}_n$ of $\boldsymbol{p}_n$ of terms of odd index satisfy $\boldsymbol{q}_n = [\boldsymbol{q}_{n-1},\: 1,\:  \overleftarrow{\boldsymbol{q}_{n-1}}]$. Let $P(x)$ be the generating function of the paperfolding sequence. Then the Toeplitz construction induces that
\begin{align}\label{Equation P not Mahler}
    P(x) = xP(x^2) + \frac{1}{1+x^2}.  
\end{align}
Multiplying both sides with $1+x^2$ shows that $P(x)$ is a Mahler function:
\begin{align*}
    (1+x^2)P(x) = (x+x^3)P(x^2) + 1.  
\end{align*}
The limit of the iteration of equality \eqref{Equation P not Mahler} gives a more direct formula:
\begin{align*}
    P(x) = \sum_{n=0}^{\infty} \frac{x^{2^n-1}}{1+x^{2^n}}.
\end{align*}
The function $P(x)$ is not rational and not even algebraic over $\mathbb{C}[x]$. If $0 < |\alpha| < 1$ is algebraic, then $P(\alpha)$ is a transcendental number \cite{france1981arithmetic, nishioka1997algebraic}.\medskip

Now we want to mimic this idea for $\rho$ and $\boldsymbol{w}$ using Theorem \ref{theorem Continued Fraction rho x, -x}. To recap, the words $\boldsymbol{w}_n$ of the folded continued fraction of $\rho_n(x)$ satisfy $\boldsymbol{w}_0 = \boldsymbol{w}_1 = [\:]$, and for $n\ge2$,
\begin{align*}
    \boldsymbol{w}_n =  [\boldsymbol{w}_{n-2},\: (-1)^nx, \: - \overleftarrow{\boldsymbol{w}_{n-2}},\: (-1)^nx,\:  \boldsymbol{w}_{n-1}].
\end{align*}
 Each word only contains $x$ and $-x$, and the sequence of $\boldsymbol{w}_n$ converges parity partially. Replace $x$ with $1$ and $-x$ with $-1$, to get
\begin{align*}
    \boldsymbol{w}_n =  [\boldsymbol{w}_{n-2},\: (-1)^n,\:  - \overleftarrow{\boldsymbol{w}_{n-2}},\: (-1)^n,\:  \boldsymbol{w}_{n-1}].
\end{align*}
From this recursion, we will compute two generating functions. First introduce some notation:
\begin{notation}
Let $N \ge 0$ and $a = (a_n)_{n=0}^{N}$ and $b = (b_n)_{n=0}^{\infty}$ be sequences. finite and infinite. Then the generating function of $a$, the polynomial $\sum_{n=0}^N a_nx^n$, is denoted by $\operatorname{GF}(a)$ and $\operatorname{GF}(b)$ is written for the power series $\sum_{n=0}^{\infty} b_nx^n$.
\end{notation}
Next, let $\mathcal{F}(x) = \lim_{n \to \infty} \operatorname{GF}(\boldsymbol{w}_{2n})$ and $\mathcal{G}(x) = \lim_{n \to \infty} \operatorname{GF}(\boldsymbol{w}_{2n+1})$, to state the Mahler equations.
\begin{theorem}\label{theorem Mahler equations words Cofrac rho}We have\begin{align*} 
    \mathcal{F}(x) &= x^2\mathcal{F}(x^4)-\frac{2x^6}{1+x^8} +  \frac{1}{1+x^4}  + \frac{x}{1+x^2}\quad\text{and}\\
    \mathcal{G}(x) &= x^4\mathcal{G}(x^4)-\frac{1-x^8}{1+x^8} +\frac{x^2}{1+x^4} -  \frac{x}{1+x^2}.
\end{align*}\end{theorem}
This result has a lengthy proof, but there has to be a starting point somewhere. Recall that $k_n$ denotes the length of the word $\boldsymbol{w}_n$. Thus, by Lemma \ref{Lemma defining r_n}, $k_0 = k_1 = 0$, $k_n = k_{n-1} + 2k_{n-2} + 2$ for $n \ge 2$. We need a refinement modulo 4.
\begin{lemma}\label{Lemma word length rho}
For all even $n \ge 2$, $k_{n} \equiv 2 \mod 4$ and for all odd $n \ge 1$, $k_{n} \equiv 0 \mod 4$.
\end{lemma}
\begin{proof}
Use that $k_n$ is even and apply induction on $k_n = 2k_{n-2} + k_{n-1} + 2 \equiv k_{n-1} + 2 \mod 4$.\end{proof}
\begin{prop}\label{prop words w at odd indices}
Let $n \ge 0$ and $0 \le m < k_n$. Then:
\begin{enumerate}
    \item If $n$ is even, then $\boldsymbol{w}_n[m] = 1$ if $m \equiv 1 \mod 4$ and $\boldsymbol{w}_n[m] = -1$ if $m \equiv 3 \mod 4$.
    \item If $n$ is odd, then $\boldsymbol{w}_n[m] = -1$ if $m \equiv 1 \mod 4$ and $\boldsymbol{w}_n[m] = 1$ if $m \equiv 3 \mod 4$.
\end{enumerate}\end{prop}
\begin{proof}
Use induction on $n$; the cases $n = 0$ and $n=1$ follow easily. Now assume the statement holds for $n-1$ and $n-2$. Then recall Lemma \ref{Lemma word length rho} and the recursion 
\begin{align*}
    \boldsymbol{w}_n =  [\boldsymbol{w}_{n-2},\: (-1)^n,\:  -\overleftarrow{\boldsymbol{w}_{n-2}},\: (-1)^n,\:  \boldsymbol{w}_{n-1}].
\end{align*}
For all five concatenated pieces of the recursion, the statement will be proven. The first, $\boldsymbol{w}_{n-2}$, follows from $n-2$ and $n$ sharing the same parity. The first $(-1)^n$ appears on position $k_{n-2}$, which is even, hence does not interfere. Now consider the third part, $-\overleftarrow{\boldsymbol{w}_{n-2}}$. If $n$ is even, then $k_n \equiv 2 \mod 4$, so the position of the first term of $-\overleftarrow{\boldsymbol{w}_{n-2}}$ is 3 modulo 4. Then $-\overleftarrow{\boldsymbol{w}_{n-2}}[4k] = -\boldsymbol{w}_{n-2}[k_n-1-4k] = -\boldsymbol{w}_{n-2}[1] = -1$ is on positions $3$ modulo 4 and $-\overleftarrow{\boldsymbol{w}_{n-2}}[4k+2] = -\boldsymbol{w}_{n-2}[k_n-1-4k-2] = -\boldsymbol{w}_{n-2}[3] = 1$ on positions 1 modulo 4. If $n$ is odd, $k_n \equiv 0 \mod 4$, $-\overleftarrow{\boldsymbol{w}_{n-2}}[0]$ is at a position 1 modulo 4. Then $-\overleftarrow{\boldsymbol{w}_{n-2}}[4k] = -\boldsymbol{w}_{n-2}[k_n-1-4k] = -\boldsymbol{w}_{n-2}[3] = 1$ is on positions $1$ modulo 4 and $-\overleftarrow{\boldsymbol{w}_{n-2}}[4k+2] = -\boldsymbol{w}_{n-2}[k_n-4k-3] = -\boldsymbol{w}_{n-2}[1] = 1$ on positions 3 modulo 4. The second $(-1)^n$ appears on an odd position $2k_{n-2} + 1$. By Lemma \ref{Lemma word length rho}, $2k_{n-2} + 1$ is 1 modulo 4, and $(-1)^n$ is thus exactly as required. The last part $\boldsymbol{w}_{n-1}$ begins on position $2k_{n-2} + 1$ which is $2$ modulo 4. If $m \equiv 1 \mod 4$, $\boldsymbol{w}_{n-1}[m]$ is placed on a position that is 3 modulo 4 in $\boldsymbol{w}_n$. Similarly, if $m \equiv 3 \mod 4$, $\boldsymbol{w}_{n-1}[m]$ is added at a position $1$ modulo 4 in $\boldsymbol{w}_n$.\end{proof}

\begin{corollary}\label{corollary GF odd indices} We have
\begin{align*}
    \lim_{n \to \infty} \operatorname{GF}([0,\:\boldsymbol{w}_{2n}[1],\:0,\:\boldsymbol{w}_{2n}[3],\:\dots,\:0,\:\boldsymbol{w}_{2n}[k_{2n}-1]]) &= \frac{x}{1+x^2} \quad\text{and}\\
    \lim_{n \to \infty} \operatorname{GF}([0,\:\boldsymbol{w}_{2n+1}[1],\:0,\:\boldsymbol{w}_{2n+1}[3],\:\dots,\:0,\:\boldsymbol{w}_{2n+1}[k_{2n+1}-1]]) &= \frac{-x}{1+x^2}.
\end{align*}
\end{corollary}
To continue, it is useful to introduce the terms of $\boldsymbol{w}_n$ of even index with a small twist. Define 
\begin{align*}
    \boldsymbol{e}_n := [\boldsymbol{w}_n[0],\: -\boldsymbol{w}_n[2],\: \boldsymbol{w}_n[4],\: \dots,\: (-1)^m\boldsymbol{w}_n[2m],\: \dots,\: -(-1)^{\frac{k_n}{2}}\boldsymbol{w}_n[k_n-2]].
\end{align*}
For example, as $\boldsymbol{w}_4 = [1,\: 1,\: 1,\: -1,\: -1,\: 1,\: -1,\: -1,\: 1,\: 1]$, we have $\boldsymbol{e}_4 := [1,\:-1,\:-1,\:1,\:1]$. 
\begin{prop}\label{props recursion w and e} Let $n \ge 0$. If $n$ is even, then $\boldsymbol{w}_{n} = -\boldsymbol{e}_{n+1}$, and if $n$ is odd, then $[1,\:\boldsymbol{w}_{n}] = \boldsymbol{e}_{n+1}$.
\end{prop}
\begin{proof}
As $k_n$ is even, $\#\boldsymbol{e}_n = \frac{k_n}{2}$, and $k_n = \frac{2^{n+1}+(-1)^n}{3} - 1$ by Lemma \ref{Lemma defining r_n}. Then, if $n$ is even, $k_{n+1} = 2k_{n}$, and if $n$ is odd, $k_{n+1} = 2k_{n}+1$. Thus, the number of terms on both sides agrees for both parities of $n$. The statement holds for $n=0$ and $n=1$ and proceed by induction. If $n$ is even, then $k_{n-1} \equiv 0 \mod 4$, and if $n$ is odd, $k_{n-1} \equiv 2 \mod 4$, so $\frac{k_{n-1}}{2} \equiv n \mod 2$. Thus, $n + 1+\frac{k_{n-1}}{2} \equiv 1 \mod 2$ and $(-1)^{n + 1+\frac{k_{n-1}}{2}} = -1$. Then
\begin{align*}
    &(-1)^{\frac{k_{n-1}}{2}}[\boldsymbol{w}_{n-1}[k_{n-1}-2],\:-\boldsymbol{w}_{n-1}[k_{n-1}-4],\: \dots,\: -(-1)^{\frac{k_{n-1}}{2}}\boldsymbol{w}_{n-1}[0]] \\
    &\quad = (-1)^{\frac{k_{n-1}}{2}}\overleftarrow{[-(-1)^{\frac{k_{n-1}}{2}}\boldsymbol{w}_{n-1}[0],\: (-1)^{\frac{k_{n-1}}{2}}\boldsymbol{w}_{n-1}[2],\: \dots,\:\boldsymbol{w}_{n-1}[k_{n-1}-2]]}\\
    &\quad = (-1)^{\frac{k_{n-1}}{2}}\cdot \big(-(-1)^{\frac{k_{n-1}}{2}}\overleftarrow{\boldsymbol{e}_{n-1}}\big) = -\overleftarrow{\boldsymbol{e}_{n-1}}
\end{align*}
is needed to prove that
\begin{align*}
    \boldsymbol{e}_{n+1} &= [\boldsymbol{w}_n[0],\: -\boldsymbol{w}_n[2],\: \boldsymbol{w}_n[4],\: \dots,\: (-1)^m\boldsymbol{w}_n[2m],\: \dots,\: -(-1)^{\frac{k_n}{2}}\boldsymbol{w}_n[k_n-2]] \\
    &=[\boldsymbol{w}_{n-1}[0],\: \dots,\: -(-1)^{\frac{k_{n-1}}{2}}\boldsymbol{w}_{n-1}[k_{n-1} - 2],\: (-1)^{\frac{k_{n-1}}{2}}(-1)^{n+1},\: -(-1)^{\frac{k_{n-1}}{2}}\overleftarrow{-\boldsymbol{w}_{n-1}}[1],\: \dots,\:\\
    &\quad \overleftarrow{-\boldsymbol{w}_{n-1}}[k_{n-1}-1],\: -\boldsymbol{w}_{n}[0], \boldsymbol{w}_{n}[2],\: \dots ,\: (-1)^{\frac{k_n}{2}}\boldsymbol{w}_{n}[k_n-2]]\\
    &=[\boldsymbol{e}_{n-1},\: (-1)^{n+1+\frac{k_{n-1}}{2}},\: (-1)^{\frac{k_{n-1}}{2}}\boldsymbol{w}_{n-1}[k_{n-1}-2],\: \dots,\: -\boldsymbol{w}_{n-1}[0],\: -\boldsymbol{e}_{n}]\\
    &= [\boldsymbol{e}_{n-1},\:-1,\: -\overleftarrow{\boldsymbol{e}_{n-1}},\: -\boldsymbol{e}_{n}]. 
\end{align*}
Thus, if $n$ is even, 
\begin{align*}
    -\boldsymbol{e}_{n+1} &= -[\boldsymbol{e}_{n-1},\:-1,\: -\overleftarrow{\boldsymbol{e}_{n-1}},\: -\boldsymbol{e}_{n}]\\
    &= -[-\boldsymbol{w}_{n-2}],\: -1,\: \overleftarrow{\boldsymbol{w}_{n-2}},\: -[1,\: \boldsymbol{w}_n]] \\
    &= [\boldsymbol{w}_{n-2},\: 1,\: -\overleftarrow{\boldsymbol{w}_{n-2}},\: 1,\: \boldsymbol{w}_n] = \boldsymbol{w}_{n+1};
\end{align*}
and if $n$ is odd,
\begin{align*}
    \boldsymbol{e}_{n+1} &= [\boldsymbol{e}_{n-1},\:-1,\: -\overleftarrow{\boldsymbol{e}_{n-1}},\: -\boldsymbol{e}_{n}]\\
    &= [[1,\:\boldsymbol{w}_{n-2}],\: -1,\: -\overleftarrow{[1,\:\boldsymbol{w}_{n-2}]},\: \boldsymbol{w}_n] \\
    &= [1,\boldsymbol{w}_{n-2},\: -1,\: -\overleftarrow{\boldsymbol{w}_{n-2}},\: -1,\: \boldsymbol{w}_n] = [1,\:\boldsymbol{w}_{n+1}].\qedhere
\end{align*}
\end{proof}
The Hadamard product of two analytic functions $f$ and $g$ is the termwise multiplication of their power series and is denoted by $f \star g$. We review this notion in greater details in Chapter 3.
\begin{lemma}\label{lemma Hadamard product mathcal F,G}
We have
\begin{align*}
    \frac{1}{1+x} \star \mathcal{F}(x) = \mathcal{F}(x) - \frac{2x}{1+x^2} \quad\text{and}\quad \frac{1}{1+x} \star \mathcal{G}(x) = \mathcal{G}(x) + \frac{2x}{1+x^2}.
\end{align*}
\end{lemma}
\begin{proof}
Let $\mathcal{F}(x) = \sum_{m=0}^{\infty}a_mx^m$ and $n \ge 0$. Then for even $n$, the left side of the equation is $1 \cdot a_n$ and the right side is $a_n + 0 = a_n$. If $n = 2m+1$ is odd, then Proposition \ref{prop words w at odd indices} gives that the left side is $-1 \cdot(-1)^m$ and the right side is $(-1)^m - 2(-1)^m = -(-1)^m$. For $\mathcal{G}$ a similar argument is used.
\end{proof}
\begin{prop}\label{Prop Simple relations F and G} The following functional equations are valid:
\begin{align*}
    \mathcal{F}(x)& =-x^2\mathcal{G}(x^2)-\frac{2x^4}{1+x^4} + 1+ \frac{x}{1+x^2} \quad \text{and}\quad \mathcal{G}(x) = -\mathcal{F}(x^2)+\frac{2x^2}{1+x^4} -  \frac{x}{1+x^2}.
\end{align*}
\end{prop}
\begin{proof}
From the definition of $\boldsymbol{e}_n$ and Corollary \ref{corollary GF odd indices} one obtains that, for each $n \ge 0$,
\begin{align*}
    \operatorname{GF}(\boldsymbol{w}_n)(x) &= \operatorname{GF}([\boldsymbol{w}_n[0],0, \boldsymbol{w}_n[2],0, \dots, \boldsymbol{w}_n[k_n-2], 0])(x)\\ &+ \operatorname{GF}([0, \boldsymbol{w}_n[1],0, \boldsymbol{w}_n[3], \dots,0, \boldsymbol{w}_n[k_n-1]])(x) \\
    &= \frac{1}{1+x^2} \star \operatorname{GF}(\boldsymbol{e}_n)(x^2) + (-1)^n\operatorname{GF}([0,1,0,-1,0,1,\dots,0,(-1)^{k_n}]).
\end{align*}
Then inserting Corollary \ref{corollary GF odd indices} and Proposition \ref{props recursion w and e} gives
\begin{align*}
    \mathcal{F}(x)& = \lim_{n \to \infty}\operatorname{GF}(\boldsymbol{w}_{2n})(x) \\
    &= \lim_{n \to \infty}\frac{1}{1+x^2} \star  \operatorname{GF}(\boldsymbol{e}_{2n})(x^2) + (-1)^{2n}\operatorname{GF}([0,1,0,-1,0,1,\dots,0,(-1)^{k_{2n}}])\\
    &=\frac{1}{1+x^2} \star \lim_{n \to \infty} \operatorname{GF}([1, \boldsymbol{w}_{2n-1}])(x^2) +  \frac{x}{1+x^2} =\frac{1}{1+x^2} \star (1+x^2\mathcal{G}(x^2)) +  \frac{x}{1+x^2}\\
    &=\frac{1}{1+x^2}\star x^2\mathcal{G}(x^2) + 1+ \frac{x}{1+x^2}=x^2\bigg(\frac{-1}{1+x^2}\star G(x^2)\bigg) + 1+ \frac{x}{1+x^2}\\
    &=-x^2\bigg(\mathcal{G}(x^2)+\frac{2x^2}{1+x^4}\bigg) + 1+ \frac{x}{1+x^2} = -x^2\mathcal{G}(x^2)-\frac{2x^4}{1+x^4} + 1+ \frac{x}{1+x^2},
\end{align*}
where the second-to-last step is due to Lemma \ref{lemma Hadamard product mathcal F,G}. The other identity follows similarly:
\begin{align*}
    \mathcal{G}(x)& = \lim_{n \to \infty}\operatorname{GF}(\boldsymbol{w}_{2n+1})(x) \\&= \lim_{n \to \infty}\frac{1}{1+x^2} \star  \operatorname{GF}(\boldsymbol{e}_{2n+1})(x^2) + (-1)^{2n+1}\operatorname{GF}([0,1,0,-1,0,1,\dots,0,(-1)^{k_{2n+1}}])\\
    &=\frac{1}{1+x^2} \star \lim_{n \to \infty} \operatorname{GF}(-\boldsymbol{w}_{2n})(x^2) -  \frac{x}{1+x^2} =\frac{1}{1+x^2} \star (-\mathcal{F}(x^2)) -  \frac{x}{1+x^2} \\
    &= -\bigg(\mathcal{F}(x^2)-\frac{2x^2}{1+x^4}\bigg) -  \frac{x}{1+x^2}=-\mathcal{F}(x^2)+\frac{2x^2}{1+x^4} -  \frac{x}{1+x^2}.\qedhere\end{align*}
\end{proof}
Now the individual Mahler equations for $\mathcal{F}$ and $\mathcal{G}$ can be computed easily. 
\begin{proof}[Proof of Theorem \ref{theorem Mahler equations words Cofrac rho}] Combine the two equalities of Proposition \ref{Prop Simple relations F and G} to deduce
\begin{align*}
    \mathcal{F}(x) &= -x^2\mathcal{G}(x^2)-\frac{2x^4}{1+x^4} + 1+ \frac{x}{1+x^2}\\
    &= -x^2\bigg(-\mathcal{F}(x^4)+\frac{2x^4}{1+x^8} -  \frac{x^2}{1+x^4}\bigg) - \frac{2x^4}{1+x^4} + 1+ \frac{x}{1+x^2}\\
   &= x^2\mathcal{F}(x^4)-\frac{2x^6}{1+x^8} -  \frac{x^4}{1+x^4}  + \frac{x}{1+x^2}
\end{align*}
and
\begin{align*}
    \mathcal{G}(x) &= -\mathcal{F}(x^2)+\frac{2x^2}{1+x^4} -  \frac{x}{1+x^2} \\
    &=-\bigg(-x^4\mathcal{G}(x^4)-\frac{2x^8}{1+x^8} + 1+ \frac{x^2}{1+x^4}  \bigg)+\frac{2x^2}{1+x^4} -  \frac{x}{1+x^2} \\
    &=x^4\mathcal{G}(x^4)-\frac{1-x^8}{1+x^8} +\frac{x^2}{1+x^4} -  \frac{x}{1+x^2}.\qedhere
\end{align*}\end{proof}
Like $F$ and $G$, the functions $\mathcal{F}$ and $\mathcal{G}$ can be connected with a simple Mahler function.
\begin{corollary}
Let $\mathcal{I}(x) = \mathcal{F}(x^3) - x^2\mathcal{G}(x^3)$. Then $\mathcal{I}(x)$ is 2-Mahler.
\end{corollary}
\begin{proof}
Combining the functions obtained in Proposition \ref{Prop Simple relations F and G} gives
\begin{align*}
    \mathcal{F}(x^3) - x^2\mathcal{G}(x^3) &= x^2\big(\mathcal{F}(x^6) - x^4\mathcal{G}(x^6)\big) -\frac{2x^8+2x^{12}}{1+x^{12}} + \frac{x^3+x^5}{1+x^6} + 1.\qedhere
\end{align*}
\end{proof}
In the same vain, the transcendence of $\mathcal{F}$ and $\mathcal{G}$ can be studied as was done for $F$ and $G$. First, an upper bound of the transcendence degree is obtained.
\begin{corollary}
The transcendence degree of $\{\mathcal{F}(x), \mathcal{G}(x), \mathcal{F}(x^2), G(x^2), \mathcal{F}(x^4), \mathcal{G}(x^4), \dots \}$ is at most $2$.
\end{corollary}
\begin{proof}
Apply Proposition \ref{Prop Simple relations F and G}:
\begin{align*}
    2 &\ge \operatorname{tr\,deg}(\mathcal{F}(x), \mathcal{G}(x)) = \operatorname{tr\,deg}(\mathcal{F}(x), \mathcal{G}(x), \mathcal{F}(x^2), \mathcal{G}(x^2)) \\
    &= \operatorname{tr\,deg}(\mathcal{F}(x), \mathcal{G}(x), \mathcal{F}(x^2), \mathcal{G}(x^2), \mathcal{F}(x^4), \mathcal{G}(x^4))=\dots = \operatorname{tr\,deg}(\mathcal{F}(x), \mathcal{G}(x), \mathcal{F}(x^2), \mathcal{G}(x^2), \dots).\:\:\:\:\qedhere
\end{align*}
\end{proof}
Proving that $\mathcal{F}$ and $\mathcal{G}$ are transcendental over $\mathbb{C}[x]$ gives a first lower bound. A standard approach would be to apply Theorem 1.3 of \cite{nishioka2006mahler}: If $\mathcal{F}(x)$ is algebraic over $\mathbb{C}$, then  $\mathcal{F}(x)$ is in $\mathbb{C}(x)$. Then a contradiction is reached by showing that there are no coprime polynomials $a(x)$ and $b(x)$ such that $\frac{a(x)}{b(x)}$ satisfies the equations of Theorem \ref{theorem Mahler equations words Cofrac rho}. Fortunately, a short-cut exists for both $\mathcal{F}$ and $\mathcal{G}$.
\begin{prop}
The functions $\mathcal{F}$ and $\mathcal{G}$ are both transcendental over $\mathbb{C}[x]$.
\end{prop}
\begin{proof}
By Theorem \ref{theorem rho and c equal}, $\rho(8) = 2\lambda(\frac{1}{2})$, which by \cite[Theorem 1.2]{adamczewski2010non} is transcendental, and hence not a quadratic irrational. Then any continued fraction of $\rho(8)$ is not eventually periodic by \cite[Theorem 2.48]{borwein2014neverending}. Then the coefficients of $8\mathcal{F}(x)$ are not eventually periodic and $8\mathcal{F}(x)$ is not rational and thus transcendental by \cite[Theorem 1.3]{nishioka2006mahler}. A similar argument is used for $\mathcal{G}$.\end{proof}
Such tricks are of little use when tackling the algebraic independence of $\mathcal{F}$ and $\mathcal{G}$, so more sophisticated methods have to be used. To start, tweak the formulas slightly and change the sign of $\mathcal{G}$. Namely, define
$$
\tilde{\mathcal{F}}(x):=\mathcal{F}(x)-\frac x{1+x^2}\quad\text{and}\quad
\tilde{\mathcal{G}}(x):=-\mathcal{G}(x)+\frac x{1+x^2}.
$$
Feeding in this new pair into Theorem \ref{theorem Mahler equations words Cofrac rho} gives
$$
\tilde{\mathcal{F}}(x)=x^2\tilde{\mathcal{G}}(x^2)+\frac1{1+x^4}\quad\text{and}\quad
\tilde{\mathcal{G}}(x)=\tilde{\mathcal{F}}(x^2)-\frac{x^2}{1+x^4}.
$$
These equalities imply that both $\tilde{\mathcal{F}}(x)$ and $\tilde{\mathcal{G}}(x)$ are even functions. This allows us to write $\mathcal{I}(x^2) = x\tilde{\mathcal{F}}(x^3)$ and $\mathcal{J}(x^2) = x^2\tilde{\mathcal{G}}(x^3)$, leading to the system of two $\{0,\pm1\}$-series:
\begin{align}\label{Mahler equations I and J}
    \mathcal{I}(x)=\mathcal{J}(x^2)+\frac{x}{1+x^6}\quad\text{and}\quad
\mathcal{J}(x)=\mathcal{I}(x^2)-\frac{x^5}{1+x^6}.
\end{align}

By reversing the construction, $\mathcal{I}(x)$ and $\mathcal{J}(x)$ being algebraically independent over $\mathbb{C}(x)$ would imply the same for $\mathcal{F}(x)$ and $\mathcal{G}(x)$. As most theorems for algebraic independence of Mahler functions desire homogeneous equations or that the coefficients in front of the functions are all constants, the entire route to $\mathcal{I}$ and $\mathcal{J}$ has to be taken. A simplified version of a theorem by Nishioka will be used:

\begin{theorem}[Nishioka, Theorem 3.2.2 in \cite{nishioka2006mahler}]\label{theorem Nishioka constants for functions}
Let $m \ge 1$ and $f_1(x),\dots,f_n(x) \in \mathbb{C}[[x]]$ satisfy
\begin{align*}
    \begin{pmatrix}
    f_1(x)\\\vdots\\f_n(x)
    \end{pmatrix}=A    
    \begin{pmatrix}
    f_1(x^k)\\\vdots\\f_n(x^k)
    \end{pmatrix}+
    \begin{pmatrix}
    g_1(x)\\\vdots\\g_n(x)
    \end{pmatrix}
\end{align*}
where $k \ge 2$, $A$ is an $n \times n$ matrix over $\mathbb{C}$ and $g_i(x) \in \mathbb{C}(x)$. If $f_1,\dots,f_m$ are algebraically dependent over $\mathbb{C}$, there are $c_1,\dots,c_m \in \mathbb{C}$ not all all zero such that $\sum_{i=0}^m c_i\cdot f_i(x) \in \mathbb{C}(x)$. 
\end{theorem}
The theorem is unhelpful when dealing with the Mahler equations \eqref{Mahler equations I and J}, so we iterate them once:
$$
\mathcal{I}(x)=\mathcal{I}(x^4)+\frac{x}{1+x^6}-\frac{x^{10}}{1+x^{12}}, \quad
\mathcal{J}(x)=\mathcal{J}(x^4)-\frac{x^5}{1+x^6}+\frac{x^2}{1+x^{12}}.
$$
\begin{theorem}
$\mathcal{I}(x)$ and $\mathcal{J}(x)$ are algebraically independent over $\mathbb{C}(x)$.
\end{theorem}
\begin{proof}
Assume the opposite. Apply Nishioka's theorem (Theorem \ref{theorem Nishioka constants for functions}) with $m=2$, $k=4$, 
$$A=\begin{pmatrix}1&0\\0&1\end{pmatrix}, \quad g_1(x) = \frac{x}{1+x^6}-\frac{x^{10}}{1+x^{12}}\quad\text{and}\quad g_2(x) = -\frac{x^5}{1+x^6}+\frac{x^2}{1+x^{12}}.$$
Then there are $c_1, c_2 \in \mathbb{C}$, not both zero, and coprime polynomials $a(x)$ and $b(x)$ such that $c_1\mathcal{I}(x)+c_2\mathcal{J}(x) = \frac{a(x)}{b(x)}$. In other words, $$\frac{a(x^4)}{b(x^4)} = \frac{a(x)}{b(x)} + c_1g_1(x) + c_2g_2(x).$$
The equation can be transformed into an equation of polynomials: $A(x) = B(x) + C(x)$, where 
\begin{align*}
    A(x) &= a(x^4)b(x)(1+x^6)(1+x^{12}), \\
    B(x) &= a(x)b(x^4)(1+x^6)(1+x^{12})\quad \text{and}\\
    C(x) &= b(x)b(x^4)\big( c_1(x-x^{10}+x^{13}-x^{16}) + c_2(x^2-x^5+x^8-x^{17})\big)\\
    &=b(x)b(x^4)x(1-x^3)\big( c_1(1+x^3+x^6+x^{12}) + c_2x(1+x^6+x^9+x^{12})\big)\\
    &= b(x)b(x^4)x(1-x^3)c(x).
\end{align*}
Using that $a(x)$ and $b(x)$ are coprime and thus $a(x^4)$ and $b(x^4)$ are coprime as well, we deduce that $b(x)\mid b(x^4)(1+x^6)(1+x^{12})$ and $b(x^4)\mid b(x)(1+x^6)(1+x^{12})$.\medskip

As $b(x^4)\mid b(x)(1+x^6)(1+x^{12})$, $x$ cannot divide $b(x)$, and so $b(0) \ne 0$. Assume $\zeta$ is a root of $b(x)$ that is not a root of $1-x^6$. Then, as $\zeta \ne 0$ and $\zeta^6\ne1$, the four fourth roots of $\zeta$ are zeros of $b(x^4)$ but not of $(1+x^6)(1+x^{12})$. Using that $b(x^4)\mid b(x)(1+x^6)(1+x^{12})$, these four roots are all roots of $b(x)$. Iterating this process gives an unbounded number of roots of $b(x)$. This contradiction means that $b(x)\mid (1-x^6)^k$ for some $k \ge 0$. By the same argument, $b(x)$ cannot have two identical roots that are zeros of $1-x^6$. Thus, $b(x)\mid (1-x^6)$ and $b(x^4) \mid (1-x^{24})$.\medskip

From $b(x)\mid (1-x^6)$ and $b(x)\mid b(x^4)(1+x^6)(1+x^{12})$, it follows that $b(x)\mid b(x^4)$. Combine this with $b(x^4)\mid b(x)(1+x^6)(1+x^{12})$ to conclude that, for $y \in \mathbb{C}$, $b(y) = 0$ if and only if $b(y^4) = 0$ and $(1+y^6)(1+y^{12})\ne 0$. In particular, for $y =-\zeta_3^k$, where $\zeta_3 = e^{\frac{2\pi i}{3}}$ and $k=0,1,2$, we have $b(\zeta_3^k) = 0$ if and only if $b(-\zeta_3^k)=0$. Thus, $\deg{b(x)}$ is even and $\deg{b(x^4)}$ is divisible by 8.\medskip

Let $\zeta$ be a root of $(1+x^6)(1+x^{12})$. If one of the multiplicities $\operatorname{ord}_{\zeta}(A), \operatorname{ord}_{\zeta}(B)$ and $\operatorname{ord}_{\zeta}(C)$ is smaller than the others and $m = \min(\operatorname{ord}_{\zeta}(A), \operatorname{ord}_{\zeta}(B), \operatorname{ord}_{\zeta}(C))$, the limit
\begin{align*}
    \lim_{x \to \zeta}\bigg( \frac{A(\zeta)}{(x - \zeta)^m} - \frac{B(\zeta)}{(x - \zeta)^m} - \frac{C(\zeta)}{(x - \zeta)^m} \bigg)
\end{align*}
 involves two zero terms and one non-zero term so cannot be zero. This contradicts $A(x) + B(x) =C(x)$. Let us apply this fact a few times. Since $b(x)\mid (1-x^6)$, $b(\zeta) \ne 0$ and $\zeta(1-\zeta^3) \ne 0$ as well.
\begin{itemize}
    \item If $b(\zeta^4) = 0$, then $a(\zeta^4) \ne 0$ as $a(x^4)$ and $b(x^4)$ are coprime. Therefore, $\operatorname{ord}_{\zeta}(A) = 1$ and $\operatorname{ord}_{\zeta}(B) \ge 2$. Then $\operatorname{ord}_{\zeta}(C) = 1$, and hence, $c(\zeta) \ne 0$ as $b(\zeta^4) = 0$.
    \item If $b(\zeta^4) \ne 0$, then $\operatorname{ord}_{\zeta}(C) = \operatorname{ord}_{\zeta}(c)$. As $\operatorname{ord}_{\zeta}(A) \ge 1$ and $\operatorname{ord}_{\zeta}(B)\ge1$, it follows that $\operatorname{ord}_{\zeta}(C)\ge 1$. Thus, $c(\zeta) = 0$ as $b(\zeta^4) \ne 0$.
\end{itemize}
To summarise, exactly one of $b(\zeta^4)$ and $c(\zeta)$ is zero. If $c(\zeta) = 0$ for such a root $\zeta$ of $(1+x^6)(1+x^{12})$, then if $\zeta$ is a root of $1+x^{12}$ or $1+x^6$, respectively,  
$$
c_1(\zeta^3+\zeta^6) =- c_2\zeta(\zeta^6+\zeta^9) \quad \text{and} \quad c_1(1+\zeta^{9}) =- c_2\zeta(\zeta^3+\zeta^{12}).
$$
In both cases, using $\zeta \ne 0$, $1+\zeta^3 \ne 0$ and $1+\zeta^9 \ne 0$, we conclude that $c_1 =- c_2\zeta^4$. Then $c(x)$ has at most four roots in common with $(1+x^6)(1+x^{12})$ as $c_1$ and $c_2$ are constants, not simultaneously zero. Then $b(x^4)$ has at least $18-4=14$ roots in common with $(1+x^6)(1+x^{12})$, and so $\deg{b(x)} \ge \big\lceil\frac{14}{4}\big\rceil = 4$. As $b(x) \mid  b(x^4)$ and $b(x)$ and $(1+x^6)(1+x^{12})$ are coprime, $\deg{b(x^4)}\ge18$. Since $\deg{b(x^4)}$ is divisible by 8, it is at least 24. As $b(x^4)\mid (1-x^{24})$ this implies that $b(x^4) = 1-x^{24}$ and $b(x) = 1-x^6$. Now divide the equation $A(x) = B(x) + C(x)$ by $b(x^4)$ to obtain
\begin{align*}
    a(x^4)= a(x)(1+x^6)(1+x^{12})+ x(1-x^3)(1-x^6)c(x).
\end{align*}
Thus, $a(1) = 4a(1)$, and so $a(1) = 0$. A contradiction, as $a(x)$ and $b(x)$ are coprime and $b(1) = 0$. Thus, $a(x)$ and $b(x)$ do not exist, implying that $\mathcal{I}(x)$ and $\mathcal{J}(x)$ are algebraically independent.
\end{proof}
\begin{corollary}
$\mathcal{F}(x)$ and $\mathcal{G}(x)$ are algebraically independent over $\mathbb{C}(x)$.
\end{corollary}

\section{Algebraic folded continued fractions}
In the previous section, we gave an example of a folded continued fraction that is a quotient of two $2$-Mahler functions. In this section, folded continued fractions are defined more generally, several examples are given and different choices are discussed. In particular, algebraic functions produced by a folded continued fraction are examined with the help of a few examples.
\subsection{Defining folded continued fractions}
The concept of folded continued fraction can be moulded in many ways, and any such definition is more or less temporarily useful. In this thesis, a relatively small area is chosen where interesting examples can turn up. Other cases like $\sum_{n=0}^{\infty} 2^{-F_n}$ are fascinating but of a different nature. 
\begin{definition}
A \textbf{recursion of words} is an infinite sequence $(\boldsymbol{w}_n)_{n=0}^{\infty}$ of finite sequences $\boldsymbol{w}_n$ called \textbf{words} such that:
\begin{itemize}
    \item Each word $\boldsymbol{w}_n$ is finite, and all its letters are $x$ or $-x$.
    \item There is a decent amount of convergence so that taking a limit $\lim_{n \to \infty} \textbf{w}_n$ is meaningful.
    \item A recursion itself means the following: There are $r, N \ge 0$ such that for each $n \ge N$, 
    \begin{align*}
        \boldsymbol{w}_n = [\boldsymbol{v}_{n, 1},\: \boldsymbol{v}_{n, 2},\: \dots,\: \boldsymbol{v}_{n, r}],
    \end{align*}
    where each $\boldsymbol{v}_{n, i}$ is equal to
    \begin{align*}
        x,\: -x,\:\boldsymbol{w}_{n-1},\: -\boldsymbol{w}_{n-1},\: \overleftarrow{\boldsymbol{w}_{n-1}},\: -\overleftarrow{\boldsymbol{w}_{n-1}},\: \dots,\: \boldsymbol{w}_{n-N},\: -\boldsymbol{w}_{n-N},\: \overleftarrow{\boldsymbol{w}_{n-N}}\text{ or } -\overleftarrow{\boldsymbol{w}_{n-N}}.   
    \end{align*}
    Here, $x$ and $-x$ are called \textbf{constant words} and those depending on $\boldsymbol{w}_{n-i}$ \textbf{non-constant words}.
\end{itemize}
\end{definition}
As before, we associate a continued fraction with each word, and the paperfolding sequence $(\boldsymbol{p}_n)_{n=0}^{\infty}$ is a recursion of words. Clearly, there is much of a personal taste in this definition, as it does not include $(-1)^n\boldsymbol{w}_{n-1}$. Thus, $(\boldsymbol{w}_n)_{n=0}^{\infty}$ for $\rho$ is not a recursion of words, but as shown below, that does not exclude $\rho$ entirely. Despite the strict definition, many examples are quite similar.
\begin{example}
Recall that the original paperfolding sequence satisfies $\boldsymbol{p}_n = [\boldsymbol{p}_{n-1},\: x,\: -\overleftarrow{\boldsymbol{p}_{n-1}}]$ for $n \ge 1$ and $\boldsymbol{p}_0 = [\:]$. Consider a few variations: 
\begin{itemize}
    \item Let $\boldsymbol{q}_0 = [x]$ and $\boldsymbol{q}_n = [\boldsymbol{q}_{n-1},\: x,\: -\overleftarrow{\boldsymbol{q}_{n-1}}]$, then $\boldsymbol{q}_{n+1} = \boldsymbol{p}_n$, it is just a shift.
    \item If $\boldsymbol{q}_0 = [\:]$ and $\boldsymbol{q}_n = [\boldsymbol{q}_{n-1},\: -x,\: -\overleftarrow{\boldsymbol{q}_{n-1}}]$, then $\boldsymbol{q}_n = -\boldsymbol{p}_n$.
    \item If $\boldsymbol{q}_0 = [\:]$ and $\boldsymbol{q}_n = [\boldsymbol{q}_{n-1},\: x,\: -\overleftarrow{\boldsymbol{q}_{n-1}},\: x]$, then $[x;\:\boldsymbol{q}_n] = \boldsymbol{p}_{n+1}$.
    \item If $\boldsymbol{q}_0 = [\:]$, $\boldsymbol{q}_1 = [\:]$ and $\boldsymbol{q}_n = [\boldsymbol{q}_{n-2},\: x,\: -\overleftarrow{\boldsymbol{q}_{n-2}}]$, then $\boldsymbol{q}_{2n} = \boldsymbol{q}_{2n-1} = \boldsymbol{p}_{n}$. Thus, twice as many iterations are needed to reach the same continued fraction.
    \item If $\boldsymbol{q}_0 = [\:]$ and $\boldsymbol{q}_n = [\boldsymbol{q}_{n-1},\: x,\: -\overleftarrow{\boldsymbol{q}_{n-1}},\: x,\: \boldsymbol{q}_{n-1},\: -x,\: -\overleftarrow{\boldsymbol{q}_{n-1}}]$. Then only half the number of iterations is needed to achieve the same continued fraction $\boldsymbol{p}$.
\end{itemize}\end{example}
In the same way, the folding sequence for $\rho$ can be written without the $(-1)^n$ term. Recall that $\boldsymbol{w}_n = [\boldsymbol{w}_{n-2},\: (-1)^nx,\: -\overleftarrow{\boldsymbol{w}_{n-2}},\: (-1)^nx,\: \boldsymbol{w}_{n-1}]$. To remove the $(-1)^n$ from the recursion, define $\boldsymbol{w}_n^+ = (-1)^n\boldsymbol{w}_n$. It satisfies $\boldsymbol{w}_n = [\boldsymbol{w}_{n-2}^+,\: x,\: -\overleftarrow{\boldsymbol{w}_{n-2}^+},\: x,\: -\boldsymbol{w}_{n-1}^+]$. As such, disallowing $(-1)^nx$ as a possible term in the recursion does not exclude $\rho$ from the definition.\medskip

Many recursions of words are uninteresting. When only including constant words, a rational function is obtained. With precisely one non-constant word, the limit is eventually periodic, hence a quadratic irrationality. There are also other ways to obtain eventually periodic words, for example by defining $\boldsymbol{w}_0 = [x]$ and $\boldsymbol{w}_n = [\boldsymbol{w}_{n-1},\: \boldsymbol{w}_{n-1}]$, which converges to $[x;\:x,\:x,\:x,\:x,\:x,\:\dots]$.\medskip

Let $k_n$ be $\#\boldsymbol{w}_n-1$. Then in the case of at least two non-constant words, the degrees of the polynomials $p_{k_n}(x)$ and $q_{k_n}(x)$ increase exponentially with each iteration. If $p_{k_n}(x)$ and $q_{k_n}(x)$ converge (partially) to power series, they fall into two categories:
\begin{enumerate}
    \item \textbf{Fast converging} recursions. These are recursions for which the number of correct coefficients of $\lim_{n \to \infty} p_{k_n}$ and $\lim_{n \to \infty} q_{k_n}$ after $n$ iterations grows exponentially with respect to $n$. The power series for $\rho$ is an example of this type.
    \item \textbf{Slowly converging} recursions have, instead of exponential, only a linear in $n$ number of correct coefficients. The next subsection contains an example of this type.
\end{enumerate}
Although countless slowly converging recursions of words exist, the set of fast converging recursions of words appears to be limited. All examples we have found are related to either $x\sum_{n =0}^{\infty}x^{2^n}$ or $\frac{H(x)}{H(x^2)}$ with slight variations in signs or linear combinations.
\subsection{A cubic folded continued fraction}\label{cubic example subsection}
In this subsection, a folded continued fraction is shown to be a root of a cubic equation over $\mathbb{Q}[x]$. First, the matrix of continuants of $\boldsymbol{w}_n$ is shown to satisfy a specific recursion. Then it is established that a particular power series satisfies the same recursion, leading to the equality.\medskip

Define the recursion of words $(\boldsymbol{w}_n)_{n=0}^{\infty}$ by $\boldsymbol{w}_0 = [\:]$ and $\boldsymbol{w}_n = [\boldsymbol{w}_{n-1}, \:\boldsymbol{w}_{n-1},\: x,\: -\overleftarrow{\boldsymbol{w}_{n-1}},\: -\overleftarrow{\boldsymbol{w}_{n-1}}]$. Furthermore, let $k_n := \#\boldsymbol{w}_n -1$.

\begin{lemma}\label{lemma matrix relation p, q}
Let $n \ge 1$ and $m = k_n$. Then
\begin{align*}
    \begin{pmatrix}
    p_{k_{n+1}} & p _{k_{n+1}-1} \\ q _{k_{n+1}} & q _{k_{n+1}-1}
    \end{pmatrix} = 
    \begin{pmatrix}
    x(p_m^2+ p_{m-1}q_m)^2 & 1-xq_m(p_m+ q_{m-1})(p_m^2+p_{m-1}q_m) \\
    1+xq_m(p_m+ q_{m-1})(p_m^2+p_{m-1}q_m)  & -xq_m^2(p_m+q_{m-1})^2
    \end{pmatrix}.
\end{align*} 
\end{lemma}
\begin{proof}
The recursion for $\boldsymbol{w}_n$ gives $k_{n+1} = \#\boldsymbol{w}_{n+1} - 1 = 4\#\boldsymbol{w}_{n} = 4k_n-4$ is even for all $n \ge 1$. Thus, $p_{m-1}q_m - p_mq_{m-1} = (-1)^m = 1$ by Lemma \ref{lemma cross product p_n and q_n}. Recall that the matrix of continuants of $-\overleftarrow{\boldsymbol{w}_n}$ is
\begin{align*}
\begin{pmatrix}
        (-1)^{m+1}p_m & (-1)^mp_{m-1}\\
    (-1)^mq_m & (-1)^{m-1}q_{m-1}
\end{pmatrix} = 
\begin{pmatrix}
        -p_m & p_{m-1}\\
    q_m & -q_{m-1}
\end{pmatrix}.
\end{align*}
Next, the $[\boldsymbol{w}_{n-1},\: x,\: -\overleftarrow{\boldsymbol{w}_{n-1}}]$ part of the recursion is computed like for paperfolding sequence: 
\begin{align*}
    &\begin{pmatrix}
         p_m & p_{m-1} \\ q_m & q_{m-1}
    \end{pmatrix}
    \begin{pmatrix}
        x & 1 \\ 1 & 0
    \end{pmatrix}
    \begin{pmatrix}
        -p_m & q_m \\ p_{m-1} & -q_{m-1}
    \end{pmatrix} 
    = 
        x\begin{pmatrix}
        -p_m^2 &  p_mq_m\\
        -p_mq_m  & q_m^2
    \end{pmatrix} +
    \begin{pmatrix}
    0 & 1 \\ 1 & 0 
    \end{pmatrix}.
\end{align*}
These cancellations make the rest of the computation slightly easier. We obtain,
\begin{align*}
    &    \begin{pmatrix}
    p_{k_{n+1}} & p _{k_{n+1}-1} \\ q _{k_{n+1}} & q _{k_{n+1}-1}
    \end{pmatrix} = \begin{pmatrix}
         p_m & p_{m-1} \\ q_m & q_{m-1}
    \end{pmatrix}^2
    \begin{pmatrix}
        x & 1 \\ 1 & 0
    \end{pmatrix}
    \begin{pmatrix}
        -p_m & q_m \\ p_{m-1} & -q_{m-1}
    \end{pmatrix}^2\\&\quad=
      x \begin{pmatrix}
         p_m & p_{m-1} \\ q_m & q_{m-1}
    \end{pmatrix}
    \begin{pmatrix}
        -p_m^2 &  p_mq_m\\ -p_mq_m  & q_m^2
    \end{pmatrix}
    \begin{pmatrix}
        -p_m & q_m \\ p_{m-1} & -q_{m-1}
    \end{pmatrix}\\ &\quad +\begin{pmatrix}
         p_m & p_{m-1} \\ q_m & q_{m-1}
    \end{pmatrix}
    \begin{pmatrix}
        0 &  1\\ 1  & 0
    \end{pmatrix}
    \begin{pmatrix}
        -p_m & q_m \\ p_{m-1} & -q_{m-1}
    \end{pmatrix}.
\end{align*}
Computing the two parts separately gives the required result, since
\begin{align*}
    &x \begin{pmatrix}
         p_m & p_{m-1} \\ q_m & q_{m-1}
    \end{pmatrix}
    \begin{pmatrix}
        -p_m^2 &  p_mq_m\\ -p_mq_m  & q_m^2
    \end{pmatrix}
    \begin{pmatrix}
        -p_m & q_m \\ p_{m-1} & -q_{m-1}
    \end{pmatrix} \\&\quad=
    x\begin{pmatrix}
         p_m & p_{m-1} \\ q_m & q_{m-1}
    \end{pmatrix}
    \begin{pmatrix}
    p_m(p_m^2+p_{m-1}q_m)  &
    -p_mq_m(p_m+q_{m-1})  \\
    q_m(p_m^2 +p_{m-1}q_m) &
    - q_m^2(p_m+q_{m-1})
    \end{pmatrix}\\&\quad=
    x\begin{pmatrix}
    (p_m^2+ p_{m-1}q_m)^2 & -q_m(p_m^2+p_{m-1}q_m)(p_m +q_{m-1}) \\
    q_m(p_m+ q_{m-1})(p_m^2+p_{m-1}q_m)  & -q_m^2(p_m+q_{m-1})^2
    \end{pmatrix}
\end{align*} 
and
\begin{align*}
&\begin{pmatrix}
         p_m & p_{m-1} \\ q_m & q_{m-1}
    \end{pmatrix}
    \begin{pmatrix}
        0 &  1\\ 1  & 0
    \end{pmatrix}
    \begin{pmatrix}
        -p_m & q_m \\ p_{m-1} & -q_{m-1}
    \end{pmatrix}=
    \begin{pmatrix}
         p_m & p_{m-1} \\ q_m & q_{m-1}
    \end{pmatrix}   
    \begin{pmatrix}
        p_{m-1} & -q_{m-1} \\ -p_m & q_m
    \end{pmatrix} \\ &\quad=
    \begin{pmatrix}
    p_mp_{m-1}-p_mp_{m-1} & -p_mq_{m-1}+p_{m-1}q_m \\
    q_mp_{m-1}-p_mq_{m-1} & -q_{m-1}q_m + q_mq_{m-1}
    \end{pmatrix}=
    \begin{pmatrix}
    0 & 1 \\ 1 & 0
    \end{pmatrix}.\qedhere
\end{align*}\end{proof}
The limits of $p_{k_n}$, $p_{{k_n}-1}$, $q_{k_n}$ and $q_{{k_n}-1}$ are observed to converge to power series; explicitly, to
\begin{align*}
    A(x) &= \sum_{n=0}^{\infty}a_nx^n = x + 2 x^3 + 7 x^5 + 30 x^7 + 143 x^9 + 728 x^{11} +\dotsm,\\
    B(x) &= \sum_{n=0}^{\infty}b_nx^n =1 - x^2 - 3 x^4 - 12 x^6 - 55 x^8 - 273 x^{10}-1428x^{12}+\dotsm,\\
    C(x) &= \sum_{n=0}^{\infty}c_nx^n =1+ x^2+ 3 x^4 + 12 x^6 + 55 x^8 + 273 x^{10}+1428x^{12}+\dotsm,\\
    D(x) &= \sum_{n=0}^{\infty}d_nx^n =-x^3 - 4 x^5 - 18 x^7 - 88 x^9 - 455 x^{11} + \dotsm.
\end{align*}
Here, $a_n$ and $d_n$ vanish for even $n$ while $b_n$ and $c_n$ vanish for odd $n$. A well-educated guess suggests that the non-zero terms of $A(x)$, $C(x)$ and $D(x)$ are given by A006012, A001764 and the negation of A006629 on the OEIS \cite{OEIS}, respectively. Then $a_{2n+1} = \frac{1}{n+1}\binom{3n+1}{n}$, $c_{2n} = \frac{1}{2n+1}\binom{3n}{n}$ and $d_{2n+3} = -\frac{2}{n+2}\binom{3n+3}{n}$. Meanwhile, $B(x)=2 - C(x)$. These formulas are actually used to define $A(x),\dots,D(x)$. The OEIS lists the following results: $xC(x)^2 = A(x)$, $D(x) = -x^3C(x)^4$, $C(x) = 1 + x^2C(x)^3$ \cite{OEIS}. This implies that $D(x) = -xA(x)^2$. A few more identities are needed.
\begin{lemma}\label{lemma Combining xC = A + D}
$D(x) = xC(x)-xC(x)^2$, $xC(x) = A(x) + D(x)$ and $A(x)^2 + B(x)C(x) = C(x)$. 
\end{lemma}
\begin{proof}
All three identities follow easily by manipulating the identities found on the OEIS.
\end{proof}
Write $A$ for $A(x)$, $B$ for $B(x)$ etc. Now we prove that $A, B, C$ and $D$ indeed relate to $\lim_{n \to \infty} \boldsymbol{w}_n$.
\begin{lemma}\label{lemma matrix relation A,B,C,D}
We have
\begin{align*}
    \begin{pmatrix}
A & B\\ C&D
\end{pmatrix} = 
\begin{pmatrix}
x(A^2+BC)^2 & 1-xC(A+D)(A^2+BC) \\
1+xC(A+D)(A^2+BC) & -xC^2(A+D)^2
\end{pmatrix}.
\end{align*}
\end{lemma}
\begin{proof} We have 
\begin{align*}
    &\begin{pmatrix}
x(A^2+BC)^2 & 1-x^2C(A+D)(A^2+BC) \\
xC(A+D)(A^2+BC) + 1 & -xC^2(A+D)^2
\end{pmatrix} \\&\quad=
\begin{pmatrix}
xC^2 & -xC^3+ 1 \\
1+x^2C^3 & -x^3C^4
\end{pmatrix} =     \begin{pmatrix}
A & B\\ C&D
\end{pmatrix}.\qedhere
\end{align*}\end{proof}
\begin{theorem}
Near the origin, the continued fraction $\boldsymbol{w}_n$ converges to $\frac{A(x)}{C(x)}$ as $n \to \infty$.
\end{theorem}
\begin{proof}
We aim at proving by induction that, after $n$ recursions, the first $n$ coefficients of $p_{k_n}$, $p_{k_n-1}$, $q_{k_n}$ and $q_{k_n-1}$ are the same as the first $n$ coefficients of $A(x), B(x), C(x)$ and $D(x)$. Since $\boldsymbol{w}_0 = [\:]$, the corresponding matrix is $\begin{pmatrix}0 & 1 \\ 1 & 0\end{pmatrix}$ as required. If $n \ge 1$, Lemma \ref{lemma matrix relation p, q} and Lemma \ref{lemma matrix relation A,B,C,D} show that each iteration gives one extra term in the $x$-expansions of the matrix entries.
\end{proof}
\begin{corollary}
Near the origin, the continued fraction $\boldsymbol{w}_n$ converges to $xC(x)$ as $n \to \infty$.
\end{corollary}
\begin{proof}
$\frac{A(x)}{C(x)} = \frac{xC(x)^2}{C(x)} =xC(x)$.
\end{proof}

As $C(x) = 1 + x^2C(x)^3$, we get $\big(xC(x)\big)^3-xC(x)+x = 0$, hence the continued fraction $\boldsymbol{w}$ corresponds to a root of the polynomial $t^3-t+x \in \mathbb{C}[x][t]$. For real $|x| < \sqrt{\frac{4}{27}}$, this polynomial has three real roots, and the continued fraction represents the middle one,
\begin{align*}
    \sqrt[3]{3} \bigg(\frac{\sqrt[3]{2\sqrt{81x^2-12}-18x}}{3} + \frac{1}{\sqrt[3]{2\sqrt{81x^2-12}-18x}}\bigg).
\end{align*}
Sadly, such identities hold only within a certain domain of the complex plane without any nonzero integers. The series for $C(x)$ converges within the disk $|x| < \sqrt{\frac{4}{27}}$ \cite{OEIS}, and experimental evidence supports that  the continued fraction is not any more algebraic for rational $|x| \ge \sqrt{\frac{4}{27}}$. For example, when $x = 1$, the continued fraction gives $1.43229968255951445829\dots$, which does not fit any polynomial equation up to degree 60 with reasonably sized coefficients. Moreover, $x = \pm \sqrt{\frac{4}{27}}$ is also exactly the point where $t^3-t+x$ has a double root, so this point is not a surprise. Inside the radius of convergence, the convergence of the continued fraction is slow. While the number of terms of the continued fraction grows fourfold with each iteration, the number of correct digits increases at a constant rate. In particular, the example is a slowly converging folded continued fraction. In the domain $|x|>\sqrt{\frac{4}{27}}$, the continued fraction seems to be a smooth function with a power series expansion at infinity: 
\begin{align*} y^{-1} (1 + y^2 - y^4 + y^8 - 2y^{12} + 2y^{14} + y^{16} - 4y^{18} + O(y^{20})), \end{align*}
where $y = x^{-1}$.  The convergence towards this power series is fast: the number of correct coefficients in the power series grows as quickly as the number of terms in the continued fraction. The folding curve of this recursion is boring: It is self-crossing, has many gaps and is less self-similar than the dragon curve or the curve for $\rho$ as Figure \ref{fig:Curve cubic example} illustrates. 
\begin{figure}[H]
 \centering
  \includegraphics[width=80mm]{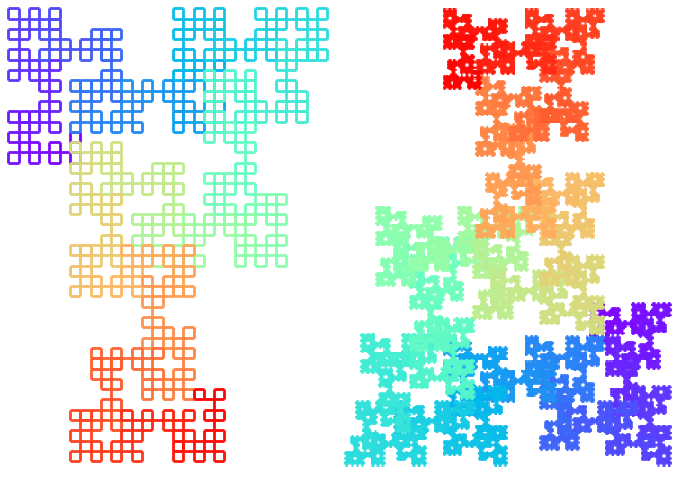}
  \caption{The folding curve of the cubic folded continued fraction starts purple and slowly fades to red following the colours of the rainbow. The left curve is drawn after 6 iterations and the right after 8 iterations.}
  \label{fig:Curve cubic example}
\end{figure}
Many recursions of words relate to this cubic example. An alternative recursion, $\boldsymbol{v}_0 = [\:]$ and $\boldsymbol{v}_n = [\boldsymbol{v}_{n-1},\: -\overleftarrow{\boldsymbol{v}_{n-1}},\: -\overleftarrow{\boldsymbol{v}_{n-1}},\: \boldsymbol{v}_{n-1},\: x]$, gives the exact same cubic and power series $A, B, C$ and $D$ as introduced earlier. Both converge to the same values within the disk $|x| < \sqrt\frac{4}{27}$, where the matrix converges, but differ outside this disk. For example,
\begin{align*}
    \boldsymbol{w}(1) &= 1.4322996825595144583\dots=[ 1, 2, 3, 5, 5, 3, 5, 4, 2, 5,3, \dots]; \\
    \boldsymbol{v}(1) &= 0.60281150560716119896 \dots =[ 0, 1, 1, 1, 1, 13, 1, 1, 1, 2, 6,\dots];\\ 
    \boldsymbol{w}(10) &=10.099000099980200960\dots=[ 10, 10, 9, 1, 9, 9, 1, 9, 10, 9, 1,\dots];\\ 
    \boldsymbol{v}(10) &=9.9009999000199950014 \dots=[ 9, 1, 9, 9, 1, 9, 9, 1, 9, 9, 1,\dots].
\end{align*}
The general pattern of $\boldsymbol{w}$ and $\boldsymbol{v}$ also differ wildly as can be seen in Figures \ref{fig:Curve cubic example} and \ref{fig: Alternative cubic curve}.
\begin{figure}[H]
 \centering
  \includegraphics[width=150mm]{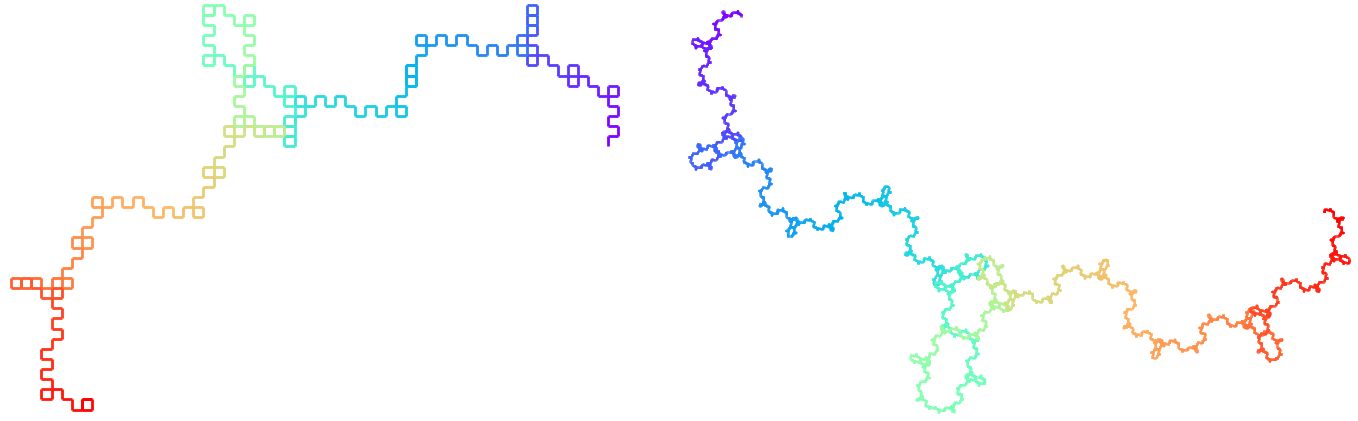}
  \caption{The folding curve of the alternative cubic folded continued fraction after 5 and 8 iterations starts purple and slowly fades to red following the colours of the rainbow.}
  \label{fig: Alternative cubic curve}
\end{figure}
Only this folded continued fraction does not `really' converge: We have checked the matrix convergence only along $(k_n)_{n=0}^{\infty}$. Another choice for such a subset, say $k_0-1, k_1-1, k_2-1,\dots$, leads to another function for $|x| <\sqrt{\frac{4}{27}}$, namely, to $\frac{B(x)}{D(x)}$ in this case.
\subsection{Other examples of folded continued fraction identities}
The behaviour of the folded continued fraction converging to the cubic function is, in contrast with $\rho$ and $\boldsymbol{p}$, not rare. The continued fractions of many more recursions of words tend to algebraic functions.
\begin{example}\label{example rational function folding word}
Take the recursion of words given by $\boldsymbol{w}_0 = [\:]$ being empty and $$\boldsymbol{w}_n = [\boldsymbol{w}_{n-1},\:-\overleftarrow{\boldsymbol{w}_{n-1}},\:x,\: -x,\: -\overleftarrow{\boldsymbol{w}_{n-1}}, \:\boldsymbol{w}_{n-1},\:x].$$ Then, letting $k_n = \#\boldsymbol{w}_n -1$, experimentally gives
\begin{align*}
    \lim_{n \to \infty} p_{k_n} &= 2x\sum_{n=0}^{\infty}x^{2n} = \frac{2x}{1-x^2} \quad \text{and} \\
    \lim_{n \to \infty} q_{k_n} &= 1+ 2x^2\sum_{n=0}^{\infty}x^{2n} = \frac{1+x^2}{1-x^2}.
\end{align*}
Thus, the continued fraction approaches $\frac{2x}{1+x^2}$, and it is natural to expect that its radius of convergence is 1. However, this turns out to be false as the convergence holds for $|x| < 0.22\dots$. The best explanation we have for this phenomenon is that the power series has many large incorrect terms that dominate the slow convergence.\medskip

The folding curve this sequence induces is self-intersecting but, more interestingly, looks globally like the cubic in Figure \ref{fig: Alternative cubic curve}. Only on a local scale, the curves differ. The theme of globally similar folding curves appears more often when searching graphical illustrations of underlying examples. Even when, on first glance, two folded continued fractions differ, their folding curves can be quite similar.
\begin{figure}[H]
 \centering
  \includegraphics[width=150mm]{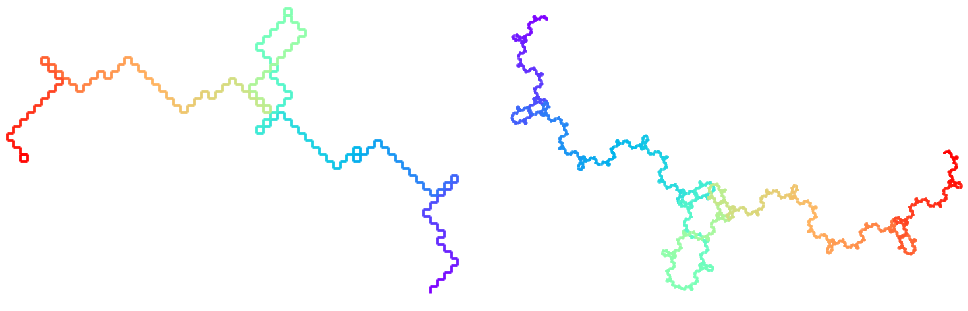}
  \caption{The curve as defined in Example \ref{example rational function folding word} after 4 and 8 iterations. It starts purple and slowly fades to red following the colours of the rainbow.}
  \label{fig:Curve rational example}
\end{figure}
This example also shows that analytically continuing such relations beyond the radius of convergence of the power series for $\lim_{n \to \infty} {p_{k_n}}$ and $\lim_{n \to \infty} {q_{k_n}}$ is impossible. For a nonzero integer $x$, the rational limit $\frac{2x}{1+x^2}$ has a finite continued fraction and differs from the infinite continued fraction $\boldsymbol{w}$. 
\end{example}
\begin{example}\label{CurveExampleQuintic}
The recursion defined by $\boldsymbol{w}_0$ being the empty word and $$\boldsymbol{w}_n = [\boldsymbol{w}_{n-1},\:\boldsymbol{w}_{n-1},\: \boldsymbol{w}_{n-1},\:-x,\:-\overleftarrow{\boldsymbol{w}_{n-1}},\:-\overleftarrow{\boldsymbol{w}_{n-1}},\:-\overleftarrow{\boldsymbol{w}_{n-1}},\:x]$$ looks like an even further exaggeration of the original folded continued fraction and the cubic folded continued fraction. For some reason, many terms cancel out again, giving a relationship. The partial quotients, $p(x)$ and $q(x)$, converge to the two power series
\begin{align*}
    p(x) &= -1 - 2x^2 - 17x^4 - 195x^6 - 2570x^8 - 36720x^{10} - 553168x^{12} - 8650756x^{14} -  O(x^{16}),\\
    q(x) &= -x - 6x^3 - 61x^5 - 756x^7 - 10406x^9 - 152880x^{11} - 2348164x^{13} - 37250298x^{15} + O(x^{17}).
\end{align*}
Both sequences of coefficients do not show up on the OEIS in any form and have no cheap closed form. Call their quotient $t(x)$. Then $t(x)$ has a Laurent expansion with a simple pole around 0:
\begin{align*}
    t(x) = x^{-1} - 4x - 20x^3 - 197x^5 - 2410x^7 - 32939x^9 - 481780x^{11} - 7377385x^{13} - O(x^{15});
\end{align*}
it is recognisable as an algebraic function of degree 5 because it satisfies
\begin{align}\label{Quintic polynomial}
    xt(x)^5-t(x)^4+2xt(x)^3+2t(x)^2-3xt(x)+x^2-1=0.
\end{align}
The continued fraction convergences to the corresponding root of this polynomial within a disk of approximate radius 0.230 around the origin. As polynomial \eqref{Quintic polynomial} is quintic over $\mathbb{Z}[t(x)]$, we note that its Galois group is isomorphic to the the symmetric group $S_5$ for generic rational values of $x$ and thus non-solvable. Therefore, there is no finite expression in radicals for $t(x)$.\medskip

As a variant, define $\boldsymbol{v}_n = [\boldsymbol{w}_n,\: \boldsymbol{w}_n,\: \boldsymbol{w}_n]$ such that  $\boldsymbol{w}_n = [ \boldsymbol{v}_{n-1},\: -x,\:-\overleftarrow{ \boldsymbol{v}_{n-1}},\: x]$ and
\begin{align*}
    \boldsymbol{v}_n = [\boldsymbol{v}_{n-1},\: -x,-\overleftarrow{ \boldsymbol{v}_{n-1}},\: x, \boldsymbol{v}_{n-1},\: -x,\:-\overleftarrow{ \boldsymbol{v}_{n-1}},\: x,\: \boldsymbol{v}_{n-1},\: -x,\:-\overleftarrow{ \boldsymbol{v}_{n-1}},\: x].
\end{align*}
Clearly, $\boldsymbol{v}_n$ is the beginning of $\boldsymbol{w}_{n-1}$ and $\boldsymbol{w}_n$ is the beginning of $\boldsymbol{v}_n$, so that $\lim_{n\to \infty}\boldsymbol{w}_n = \lim_{n\to \infty}\boldsymbol{v}_n$. However, the partial quotients of $\boldsymbol{w}$ converge to different power series than $\boldsymbol{v}$ around the origin; $\boldsymbol{w}$ is a root of a quintic polynomial and $\boldsymbol{v}$ of a cubic. In contrast with the cubic example, the folding curve of $\boldsymbol{w}$ is regular, see Figure \ref{fig:Curve Quintic example}. It is not self-crossing and has transparent self-symmetry.
\begin{figure}[H]
 \centering
  \includegraphics[width=120mm]{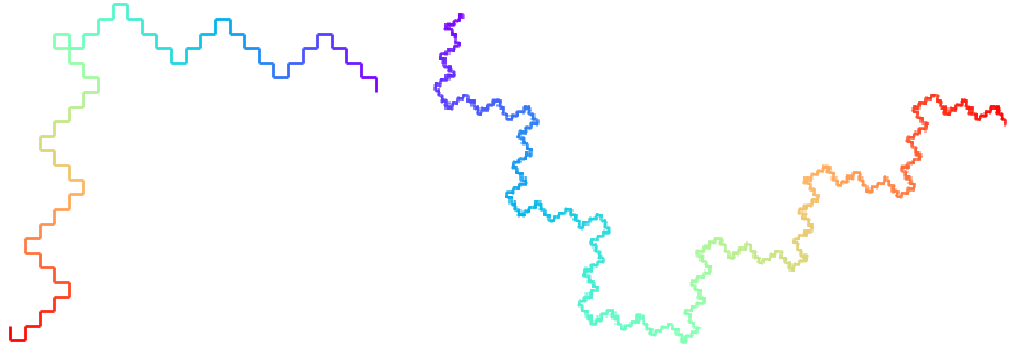}
  \caption{The curve as defined in Example \ref{CurveExampleQuintic} after 3 and 6 iterations starts purple and slowly fades to red following the colours of the rainbow.}
  \label{fig:Curve Quintic example}
\end{figure}
\end{example}
Many examples of such continued fractions exist. From above, one may falsely suspect that some sort of balance in the number of terms $\boldsymbol{w}_n$ and $-\overleftarrow{\boldsymbol{w}_n}$ is required. This is not the case: The recursion  
\begin{align*}
    \boldsymbol{w}_n = [\boldsymbol{w}_{n-1},\:x,\:\boldsymbol{w}_{n-1},\:x,\:-x,\:\boldsymbol{w}_{n-1},\:-x,\boldsymbol{w}_{n-1}]
\end{align*}
gives a continued fraction whose convergent relates to the cubic example in Subsection \ref{cubic example subsection}.
\section{An analogue of a result of Cohn}
In \cite{cohn1996symmetry}, Cohn explores the fact that the function $F(x) = \sum_{n=0}^{\infty} {x^{-2^n}}$ has a folded continued fraction form from a different perspective: All partial quotients of the regular infinite continued fraction of $F(x)$ are in $\mathbb{Z}[x]$. Moreover, the terms $x^{2^n}$ can also be interpreted as an iteration of the function $f(x) = x^2$. Let $f^n(x)$ denote $f(f^{n-1}(x))$ for $n \ge 2$; then $F(x) = \sum_{n=0}^{\infty} \frac{1}{f^n(x)}$.
\begin{definition}
A regular continued fraction with partial quotients in $\mathbb{C}(x)$ is called \textbf{specializable} if all partial quotients are polynomials in $\mathbb{Z}[x]$. \end{definition}
Now, the question is for which polynomials $f(x) \in \mathbb{Z}[x]$ the continued fraction $F(x) = \sum_{n=0}^{\infty} \frac{1}{f^n(x)}$ is specializable. The answer to this question is given by the following theorem:
\begin{theorem}[Cohn \cite{cohn1996symmetry}]\label{Theorem Cohn Spec}
For each polynomial $f \in \mathbb{Z}[x]$ of degree at least two, $\sum_{n=0}^{\infty} \frac{1}{f^n(x)}$ has a specializable continued fraction if and only if $f$ satisfies one of the following fourteen congruences:
\begin{multicols}{2}
\begin{enumerate}
    \item $f(x) \equiv 0 \mod x^2$;
    \item $f(x) \equiv -x \mod x^2$;
    \item $f(x) \equiv 1 \mod x^2(x-1)$;
    \item $f(x) \equiv -1 \mod x^2(x+1)$;
    \item $f(x) \equiv x^3-x^2-x+1 \mod x^2(x-1)^2$;
    \item $f(x) \equiv -x^3+2x^2-x+1 \mod x^2(x-1)^2$;
    \item $f(x) \equiv -x^3+3x^2-2x+1 \mod x^2(x-1)^2$;
    \item $f(x) \equiv x^3+x^2-x-1 \mod x^2(x+1)^2$;
    \item $f(x) \equiv -x^3-2x^2-x-1 \mod x^2(x+1)^2$;
    \item $f(x) \equiv -x^3-3x^2-2x-1 \mod x^2(x+1)^2$;
    \item $f(x) \equiv x^2-x+1 \mod x^2(x-1)^2$;
    \item $f(x) \equiv x^2-2x+1 \mod x^2(x-1)^2$;
    \item $f(x) \equiv -x^2-x-1 \mod x^2(x+1)^2$;
    \item $f(x) \equiv -x^2-2x-1 \mod x^2(x+1)^2$.
\end{enumerate}
\end{multicols}
\end{theorem}
Because $\rho_n(x)$ is $1+ \pfrac{x}{1} + \pfrac{f(x)}{1} + \pfrac{f^2(x)}{1} + \dots  +\pfrac{f^n(x)}{1}$ with $f(x) = x^2$, $\rho$ is specializable by Theorem \ref{theorem Continued Fraction rho x, -x}. This motivates asking Cohn's question for an irregular continued fraction shape of $F(x)$.
\begin{obs}
Let $f$ be a polynomial in $\mathbb{Z}[x]$ of degree at least two. Then the irregular continued fraction $x + \pfrac{f(x)}{1} + \pfrac{f^2(x)}{1} + \dots  +\pfrac{f^n(x)}{1}$ has a specializable continued fraction for all $n \ge 0$ if $f$ satisfies one of the following conditions:
\begin{enumerate}
    \item $f(x)  \equiv 0 \mod x^2$;
    \item $f(x)  \equiv -1 \mod x(x+1)$;
    \item $f(x)  \equiv -x-1 \mod x(x+1)$;
    \item $f(x) = x^2-2$.
\end{enumerate}
\end{obs}
The observation, in contrast with Theorem \ref{Theorem Cohn Spec}, does not contain not an `if and only if' condition. For some polynomials, a large iteration for $n$ is needed to discard them, and those that remain do not belong to obvious classes. A naive computation requires dealing with polynomials of degree $(\deg{f})^n$, so even testing $n=10$ is hard for polynomials of degree 4. All polynomials that may produce specializable continued fractions are $-x-2 \mod 2x(x+2)$, $-2x-2 \mod 2x(x+2)$, or $-x \mod x^2$.\medskip

The continued fraction for $f = x^2-2$ is quite remarkable. It appears to be the only polynomial of degree at least 2 and not a multiple of $x^2$ such that the corresponding regular continued fraction is specializable with partial quotients equal to linear polynomials. Like $\rho$, it has two limits. These are
\begin{align*}
[1 + x,\: -x,\: x,\: -x,\: x,\: -x,\: \dots] \quad \text{and}\quad [1,\: -x,\: x,\: -x,\: x,\: -x,\: \dots].
\end{align*}
Save for the first terms, the partial quotients alternate between $x$ and $-x$. Thus, both parity partial limits of the continued fraction are quadratic elements of $\mathbb{Q}[x]$. Due to these results, iterating polynomials is natural in combination with folded continued fractions. To strengthen this claim, we study a special class of recursions of words. 
\begin{definition}
Let $(\boldsymbol{w}_n)_{n=0}^{\infty}$ be a recursion of words defined by $\boldsymbol{w}_0=[\:]$ and the recursion
\begin{align*}
    \boldsymbol{w}_n = [\boldsymbol{w}_{n-1},\: \pm x,\: -\overleftarrow{\boldsymbol{w}_{n-1}},\: \pm x,\: \boldsymbol{w}_{n-1},\: \pm x,\: -\overleftarrow{\boldsymbol{w}_{n-1}},\: \pm x,\: \dots,\: \pm x,\: \boldsymbol{w}_{n-1} \text{ or } -\overleftarrow{\boldsymbol{w}_{n-1}}],
\end{align*}
where each sign is chosen independently and the parts $\boldsymbol{w}_{n-1}$ and $-\overleftarrow{\boldsymbol{w}_{n-1}}$ alternate. Then $(\boldsymbol{w}_n)_{n=0}^{\infty}$ is called \textbf{special}.
\end{definition}
To understand why these special recursions are indeed special, we give a few examples where  $k_n = \#\boldsymbol{w}_n-1$ and $\tilde{p}$ and $\tilde{q}$ denote $p_{k_{n-1}}$ and $q_{k_{n-1}}$, respectively. For $\boldsymbol{w}_n=[\boldsymbol{w}_{n-1},\: x,\: -\overleftarrow{\boldsymbol{w}_{n-1}}]$,
\begin{align*}
    p_{k_n} = x\tilde{p}^2 \quad\text{and} \quad q_{k_n} = x\tilde{p}\tilde{q}+1.
\end{align*}
For $\boldsymbol{w}_n=[\boldsymbol{w}_{n-1},\: x,\: -\overleftarrow{\boldsymbol{w}_{n-1}},\: x,\: \boldsymbol{w}_{n-1}]$,
\begin{align*}
    p_{k_n} = \tilde{p}(-x^2\tilde{p}^2+1)\quad \text{and}\quad q_{k_n} = \tilde{q}(-x^2\tilde{p}^2+1) +x\tilde{p}.
\end{align*}
For $\boldsymbol{w}_n=[\boldsymbol{w}_{n-1},\: x,\: -\overleftarrow{\boldsymbol{w}_{n-1}},\: x,\: \boldsymbol{w}_{n-1},\: x,\: -\overleftarrow{\boldsymbol{w}_{n-1}}]$,
\begin{align*}
    p_{k_n} = \tilde{p}(x^3\tilde{p}^3+2x\tilde{p}) \quad\text{and}\quad q_{k_n} = \tilde{q}(x^3\tilde{p}^3+2x\tilde{p}) +x^2\tilde{p}^2+1.
\end{align*}
For $\boldsymbol{w}_n=[\boldsymbol{w}_{n-1},\: x,\: -\overleftarrow{\boldsymbol{w}_{n-1}},\: x,\: \boldsymbol{w}_{n-1},\: x,\: -\overleftarrow{\boldsymbol{w}_{n-1}},\: x,\: \boldsymbol{w}_{n-1}]$,
\begin{align*}
    p_{k_n} = \tilde{p}(x^4\tilde{p}^4-3x^2\tilde{p}^2+1) \quad\text{and}\quad q_{k_n} = \tilde{q}(x^4\tilde{p}^4-3x^2\tilde{p}^2+1)- 3x^3\tilde{p}^3+2x\tilde{p}.
\end{align*}
For $\boldsymbol{w}_n=[\boldsymbol{w}_{n-1},\: x,\: -\overleftarrow{\boldsymbol{w}_{n-1}},\: x,\: \boldsymbol{w}_{n-1},\: x,\: -\overleftarrow{\boldsymbol{w}_{n-1}},\: x,\: \boldsymbol{w}_{n-1},\: x,\: -\overleftarrow{\boldsymbol{w}_{n-1}}]$,
\begin{align*}
    p_{k_n} = \tilde{p}(x^5\tilde{p}^5+4x^3\tilde{p}^3+3x\tilde{p}) \quad\text{and}\quad q_{k_n} = \tilde{q}(x^5\tilde{p}^5+4x^3\tilde{p}^3+3x\tilde{p}) +x^4\tilde{p}^4+3x^2\tilde{p}^2+1.
\end{align*}
Observe that the equation for $p_{k_n}$ only contains $\tilde{p}$, and the equation for $q_{k_n}$ only contains $\tilde{p}$ and $\tilde{q}$ and is linear in the latter. Furthermore, both $p_{k_n}$ and $q_{k_n}$ have a common term, which is a polynomial in $x\cdot\tilde{p}$ multiplied by $\tilde{p}$ and $\tilde{q}$, respectively; $p_{k_n}$ has no other terms while $q_{k_n}$ has an extra term, which is a polynomial in $x\cdot\tilde{p}$. These observations can be summarised as follows: 

\begin{obs}
A recursion of words that only contains $x, -x, \boldsymbol{w}_{n-1}, -\boldsymbol{w}_{n-1}, \overleftarrow{\boldsymbol{w}_{n-1}}$ and $-\overleftarrow{\boldsymbol{w}_{n-1}}$ and satisfies $\boldsymbol{w}_0 = [\:]$ gives a recursion for $p_{k_n}$ only in terms of $x$ and $p_{k_{n-1}}$ if and only if the recursion of words is special. Moreover, if the recursion of words is special, there are polynomials $P, Q \in \mathbb{Z}[y]$ such that $k_n = \#\boldsymbol{w}_n-1$, $p_{k_n} = p_{k_{n-1}}P(xp_{k_{n-1}})$ and $q_{k_n} = q_{k_{n-1}}P(xp_{k_{n-1}}) + Q(xp_{k_{n-1}})$ for $n \ge 0$, $p_{-1}=0$ and $q_{-1}=1$. 
\end{obs}
Assuming the observation, the corresponding continued fraction is a sum of an iterated polynomial:
\begin{align*}
    \frac{q_{k_n}}{p_{k_n}} = \frac{\tilde{q}P(x\tilde{p})+Q(x\tilde{p})}{\tilde{p}P(x\tilde{p})} = \frac{\tilde{q}}{\tilde{p}}+\frac{Q(x\tilde{p})}{\tilde{p}P(x\tilde{p})}= \frac{q_{k_{n-1}}}{p_{k_{n-1}}}+\frac{Q(x\tilde{p})}{\tilde{p}P(x\tilde{p})}. 
\end{align*}
Hence the continued fraction can be written as
\begin{align*}
\lim_{n \to \infty} \boldsymbol{w}_n = \Big(\lim_{n \to \infty} \frac{q_{k_n}}{p_{k_n}}\Big)^{-1} =
    \bigg(\sum_{n=0}^{\infty}\frac{Q\big(x\widehat{P}^{n}(1)\big)}{\widehat{P}^{n+1}(1)} \bigg)^{-1},
\end{align*}
where $\widehat{P}(y) = yP(x y)$. For example, if $\boldsymbol{w}_n=[\boldsymbol{w}_{n-1}, x, -\overleftarrow{\boldsymbol{w}_{n-1}}, x, \boldsymbol{w}_{n-1}, x, -\overleftarrow{\boldsymbol{w}_{n-1}}, x, \boldsymbol{w}_{n-1}]$, then $\widehat{P}(y) = x^4y^5 - 3x^2y^3 +y$ and $Q(y) = -3y^3+2y$. As $\deg{P} \ge \deg{Q}$, this sum converges for all $|x|>1$.\medskip

Special folded continued fractions also have a very clear connection with real-life paper folding, like the paperfolding dragon. The paper is folded in a number of pieces equal to the number of non-constant terms. The constant terms direct how the paper is folded.\medskip

Finally, when all $\pm x$ are chosen to be $x$, $P$ and $Q$ assume a simple form, save for perhaps signs. Coefficients of equal parity are either zero or of the form $\binom{n-k}{k}$ and $\binom{n-k-1}{k}$, where $n$ is the number of non-constant terms in the recursion. For example, when there are $n=20$ non-constant terms,
\begin{align*}
     P(y) &= y^{19} + 18y^{17} + 136y^{15} + 560y^{13} + 1365y^{11} + 2002y^9 + 1716y^7 + 792y^5 + 165y^3 + 10y,\\
     Q(y) &= y^{18} + 17y^{16} + 120y^{14} + 455y^{12} + 1001y^{10} +1287y^8 + 924y^6 + 330y^4 + 45y^2 + 1.
\end{align*}
The non-zero coefficients of ${P}(y)$ and the coefficients of $Q(y)$ correspond to $\binom{20-k}{k}$ for $k = 0,1,\dots,9$ and $\binom{19-k}{k}$ for $k = 0,1,\dots,9$, respectively. Such a connection could be established by induction, but we will not do this in the current thesis.
\section{Mahler functions, remarkable identities and Fibonacci numbers}
We now move on to a somewhat different theme. Let $F_n$ denote the $n$th Fibonacci number, $L_n$ the $n$th Lucas number and $\phi \approx 1.618\dots$ the golden ratio. Then
\begin{align*}
    F_0 = 0, \quad F_1 = 1\quad\text{and}\quad F_n &= F_{n-2}+ F_{n-1} \quad \text{for all $n\ge2$} \quad \text{and} \\
    L_0 = 2, \quad L_1 = 1\quad\text{and}\quad L_n &= L_{n-2}+ L_{n-1} \quad \text{for all $n\ge2$}. 
\end{align*}
Binet's formula is an elementary tool to work with Fibonacci and Lucas numbers.
\begin{theorem}[Binet's formula]
 For all $n \ge 0$, $F_n = \frac{\phi^n - (-\phi^{-1})^n}{\sqrt{5}}$ and $L_n = \phi^n + (-\phi)^{-n}$.
\end{theorem}
It implies that
\begin{align*}
    F_{2^n} = \begin{cases}
    1 &  \text{if } n = 0 \\
    \frac{1 - \big((\phi^{-1})^{2^n}\big)^2}{\sqrt{5}(\phi^{-1})^{2^n}} & \text{if } n \ge 1,
    \end{cases} \quad \text{and} \quad     L_{2^n} = \begin{cases}
    1 &  \text{if }n = 0 \\
    \frac{1 + \big((\phi^{-1})^{2^n}\big)^2}{(\phi^{-1})^{2^n}} & \text{if }n \ge 1.
    \end{cases}
\end{align*}
In 1974, I. J. Good \cite{good1974reciprocal} published a remarkable identity involving the Fibonacci numbers, namely
\begin{align*}
    \sum_{n = 0}^{\infty} \frac{1}{F_{2^n}} = \frac{7 - \sqrt{5}}{2}.
\end{align*}
In particular, this is an algebraic number. Meanwhile, Becker and T\"opfer \cite{becker1994transcendency} proved that $\sum_{n = 0}^{\infty} \frac{1}{F_{2^n+1}}$ is transcendental and Nishioka \cite{nishioka1997algebraic} showed that the numbers $\sum_{n = 0}^{\infty} \big(\frac{1}{F_{2^n+1}}\big)^l$ for $l \ge 1$ are all algebraically independent of each other. The powers of 2 in the indices of the Fibonacci numbers already suggest that Good's identity has something to do with Mahler functions. Such identities have been fairly popular, being cited dozens of times, and articles are still published about these sums four decades later. Good's identity can be constructed in the following way:
Let $f(x) = \frac{\sqrt{5}x}{1-x^2}$. Then $$\frac{\sqrt{5}}{1-x} = \frac{\sqrt{5}}{1-x^2} + f(x) =  \frac{\sqrt{5}}{1-x^4} + f(x) + f(x^2) =  \frac{\sqrt{5}}{1-x^8} + f(x) + f(x^2)+f(x^4) = \dotsm,$$
and so 
$$\frac{\sqrt{5}}{1-x} = \lim_{n\to\infty} \frac{\sqrt{5}}{1-x^{2^n}} + \sum_{m=0}^{n-1} f\big(x^{2^m}\big) = \sqrt{5}+ \sum_{n=0}^{\infty} f\big(x^{2^n}\big) \quad \text{if}\quad |x|<1.$$ 
Evaluating $\sum_{n=0}^{\infty} f\big(x^{2^n}\big)$ at $\phi^{-1}$ gives the desired result using Binet's formula.\medskip

More general, one finds these identities by choosing a $k \ge 2$ and a rational function $f(x)$. By taking $g(x) := f(x) - f(x^k)$, it follows from telescoping that $f(x) - \lim_{n \to \infty} f\big(x^{k^n}\big) = \sum_{n=0}^{\infty}g\big(x^{k^n}\big)$. If almost all of $g(x), g(x^k), g\big(x^{k^2}\big), \dots$ evaluate to a combination of Fibonacci and Lucas numbers by Binet's formula, and the limit and sum converge, a new identity is born.
\begin{example}
Let $f(x) = \frac{1}{\sqrt{5}(1+x^2)}$. Then $g(x) = f(x) - f(x^2) = \frac{x^2(x^2-1)}{\sqrt{5}(1+x^2)(1+x^4)} = \frac{x^2-1}{x\sqrt{5}}\frac{x^2}{1+x^4}\frac{x}{1+x^2}$ can be evaluated at $\phi^{-1}$ to get the infinite sum
$$
\sum_{n=0}^{\infty}\frac{F_{2^n}}{L_{2^n}L_{2^{n+1}}} = \frac{3\sqrt{5}+5}{30} = \frac{\sqrt{5}}{10} + \frac{1}{6}.
$$
\end{example}
This construction does not only work for sums. For example, defining $g(x)$ via $f(x) = g(x)f(x^k)$ leads to closed forms for infinite products, while extracting $g(x)$ and $h(x)$ from $f(x) = g(x)f(x^k) + h(x)$ evaluates infinite sums of finite products. For the last mentioned type, both $g\big(x^{k^n}\big)$ and $h\big(x^{k^n}\big)$ have to assume nice forms using Binet's formula for almost all $n$. We have the following example:
\begin{example}[Example of \cite{Hideyuki}]
Let $k=2$, $f(x) = 1-x^2$ and $g(x) = \frac{-x^2}{1-x^4}$. Then $f(x) = g(x)f(x^2) + 1$, implies that 
$$f(x) = \sum_{n=0}^{\infty}\prod_{m=0}^{n-1} g\big(x^{2^m}\big).$$ 
Using that $\phi^{-2^n} = -\sqrt{5}F_{2^{n+1}}$, we evaluate at $x = \phi^{-1}$ to get
\begin{align*}
    \frac{\sqrt{5}-3}{2} = f(\phi^{-1})-1 = \sum_{n=0}^{\infty} \frac{1}{g(\phi^{-1}) g(\phi^{-2}) \dotsm g\big(\phi^{-2^{n-1}}\big)}-1 
    = \sum_{n=1}^{\infty} \frac{1}{\big(-\sqrt{5}\big)^n F_2 F_4\dotsm F_{2^n}}.
\end{align*}\end{example}
This construction uses an elementary technique and produces a short proof. Different methods were used in \cite{Hideyuki}. Similarly, identities with irregular continued fractions can be constructed, which seem to be not yet discovered. The set up is the following: Let $f(x)$, $g(x)$ and $h(x)$ be rational functions satisfying $f(x) = g(x) + \frac{h(x)}{f(x^2)}$. Then fix $f(x)$ and one of $g(x)$ and $h(x)$, compute the other of $g(x)$ and $h(x)$, and test whether Binet's formula can be applied to $g\big(x^{k^n}\big)$ and $h\big(x^{k^n}\big)$ for almost all $n$. Any continued fraction with a unique limit then produces an identity.\medskip

For $h(x)=1$, multiple rational functions appear as solutions, but they always seem to have two limits, and the trick fails. For example, $f(x) = \frac{x+1}{x-1}$ and $g = \frac{2x}{x^2-1}$ gives 
$$
\lim_{n \to \infty } \frac{2}{\sqrt{5}F_1} + \pfrac{1}{\frac{2}{\sqrt{5}F_2}} + \pfrac{1}{\frac{2}{\sqrt{5}F_4}} +\dots+ \pfrac{1}{\frac{2}{\sqrt{5}F_{2^{n-1}}}}+ \pfrac{1}{\frac{2}{\sqrt{5}F_{2^n}}}.
$$
The $\frac{2}{\sqrt{5}F_{2^n}}$ is small, making $\frac{2}{\sqrt{5}F_{2^{n-1}}} + \big(\frac{2}{\sqrt{5}F_{2^n}}\big)^{-1}$ large etc. Therefore, depending on the parity of the continued fraction, the result is larger or smaller than 1, hence, the convergence fails.\medskip

On the other hand, setting $f(x) = \frac{(1+x)^2}{x}$, $g(x) = \frac{1+x^2}{x}$ and $h(x)=2\frac{(1+x^2)^2}{x^2}$ we obtain
$$
1 + \frac{2}{\sqrt{5}} = f(\phi) = g(\phi)+\pfrac{h(\phi)}{L_2} + \pfrac{2L_2^2}{L_4} + \pfrac{2L_4^2}{L_8}+\pfrac{2L_8^2}{L_{16}}+ \dotsm,
$$
which can be alternatively written as 
$$
\frac{7}{5} = L_1+\pfrac{2L_1^2}{L_2} +  \pfrac{2L_2^2}{L_4} + \pfrac{2L_4^2}{L_8}+\pfrac{2L_8^2}{L_{16}}+ \dotsm.
$$
Below is a table with a few continued fractions that converge to some number in $\mathbb{Q}$. Far more continued fractions do not converge in $\mathbb{Q}$ but in $\mathbb{Q}(\sqrt{5})$, but due to the supply of such identities, only these six examples are shown which have `nicer' values.
\begin{center}
 \begin{tabular}{||c| c| c| c| c||} 
 \hline
 $f(x)$ & $g(x)$ & $h(x)$ & Continued Fraction & Value \\ [0.5ex] 
 \hline\hline
    $\frac{(1-x)^2}{x}$ & $\frac{1+x^2}{x}$ & $-2\frac{(1-x^2)^2}{x^2}$ &  $L_1-\pfrac{10F_1^2}{L_2} -  \pfrac{10F_2^2}{L_4} - \dotsm$ & -9\\
  \hline
    $\frac{(1-x^2)^2}{x^2}$ & $\frac{(1+x^2)^2}{x^2}$ & $-4\frac{(1-x^4)^2}{x^4}$ &  $L_1^2-\pfrac{20F_2^2}{L_2^2} -  \pfrac{20F_4^2}{L_4^2} - \dotsm$ &-3\\
 \hline
   $\frac{1+x^4}{x^2}$ & $\frac{(1+x^2)^2}{x^2}$ & $-2\frac{1+x^8}{x^4}$ &  $L_1^2-\pfrac{2L_4}{L_2^2} -  \pfrac{2L_8}{L_4^2} - \dotsm$ &-1\\
 \hline
 $\frac{(1+x)^2}{x^2+1}$ & $1+\frac{1+x^2}{x}$& $\frac{(1+x^2)^2}{x^2}$ & $L_1+1+\pfrac{L_1^2}{1+L_2} +  \pfrac{L_2^2}{1+L_4} +\dotsm$ &2.2 \\
 \hline
   $\frac{1+x^4}{x^2}$ & $\frac{(1-x^2)^2}{x^2}$ & $2\frac{1+x^8}{x^4}$ &  $5F_1^2+\pfrac{2L_4}{5F_2^2} +  \pfrac{2L_8}{5F_4^2} + \dotsm$ &7\\
 \hline
   $\frac{(1+x^2)^2}{x^2}$ & $1+\frac{(1-x^2)^2}{x^2}$ & $3\frac{(1+x^4)^2}{x^4}$ &  $1+5F_1^2+\pfrac{3L_2^2}{1+5F_2^2} +  \pfrac{3L_4^2}{1+5F_4^2} + \dotsm$& 9\\ 
 [1ex] 
 \hline
\end{tabular}
\end{center}
\chapter{Hadamard products of Mahler functions}\label{Chapter Hadamard}
The space of $k$-Mahler functions is known to be closed under several operations, for example addition and multiplication. In this section, its closeness under another operation, the Hadamard product, is studied. To do this efficiently, we show that the usual definition of a Mahler function can be weakened. Recall that Mahler function were assumed to be analytical and thus have a power series expansion around the origin.
\begin{theorem}\label{Theorem Weaker condition Mahler}
    The condition that $A_0(q)A_d(q) \ne 0$ in the definition of a $k$-Mahler function $f(q) \in \mathbb{C}[[q]]\setminus\{0\}$ (Definition \ref{definition Mahler function}) can be weakened to ``not all $A_i(q)$ are zero''.
\end{theorem}
\begin{proof}
Without loss of generality, $A_d(q)$ can be assumed to be 0, and if $A_0(q)=0$, then by the new condition, there is a smallest number $0 \le e < d$ such that $A_{e}(q) \ne 0$. As such, we have the equation
\begin{align}\label{Mahler relation theorem exisitence}
    A(q) + A_e(q)f(q^{k^e}) + A_{e+1}(q)f(q^{k^{e+1}}) + \dots + A_d(q)f(q^{k^d}) = 0.
\end{align}
Note that $\mathbb{C}[[q]]$ can be written as a direct sum of $\mathbb{C}[[q^{k^e}]]$, $q\mathbb{C}[[q^{k^e}]],\dots,q^{k^e-1}\mathbb{C}[[q^{k^e}]]$. As $ 0 \ne f(q) \in \mathbb{C}[[q]]$, $f(q^{k^e})$ is in the first of these otherwise disjoint vector spaces and in none of the others, projecting the Mahler equation \eqref{Mahler relation theorem exisitence} onto each of these vector spaces gives a new relation for $f(q)$. As $d$ is minimal, the relations are linearly independent, and hence, as \eqref{Mahler relation theorem exisitence} is another relation for $f(q)$, all but one is zero. Let this sole projection be to $q^j\mathbb{C}[[q^{k^e}]]$ for a $0 \le j < k^e$. Then $q^j$ divides the polynomial coefficients of the Mahler equation, and so by updating the relation by dividing the relation by $q^j$, all new polynomial coefficients are in $\mathbb{C}[q^{k^e}]$. Evaluating the new equation in $q^{\frac{1}{k^e}}$ gives a usual $k$-Mahler equation.
\end{proof}
\begin{corollary}\label{corollary dimension is sufficient}
If $f$ is analytic and $\dim_{\mathbb{C}(q)}(f(q), f(q^k), f(q^{k^2}), \dots)$ is finite, then $f$ is Mahler.
\end{corollary}
\begin{proof}
By the assumption, there is an equation like \eqref{Mahler relation theorem exisitence}, and so $f$ is Mahler by Theorem \ref{Theorem Weaker condition Mahler}.
\end{proof}
This corollary enables us to show the space of $k$-Mahler functions is closed under several operations.
\begin{example}
That the space of $k$-Mahler functions is closed under addition, can be seen as follows. For a $k$-Mahler function $f$ of degree $d_f$, the vector space $\mathbb{C}(q)[1, f(q), f(q^k), f\big(q^{k^2}\big), \dots]$ has dimension at most $d_f+1$. If $g$ is a $k$-Mahler function of degree $d_g$, then
\begin{align*}
    &\dim_{\mathbb{C}(q)}\big(1, (f+g)(q), (f+g)(q^k), (f+g)\big(q^{k^2}\big),\dots \big)  \\
    &\quad\le\dim_{\mathbb{C}(q)} \big(1, f(q), f(q^k), f\big(q^{k^2}\big), \dots, g(q), g(q^k), g\big(q^{k^2}\big), \dots \big) \le d_f + d_g + 1
\end{align*}
is finite. Thus, $f+g$ satisfies a $k$-Mahler equation by corollary \ref{corollary dimension is sufficient}.
\end{example}
\begin{definition}
The \textbf{Hadamard product} of two power series $f(q) = \sum_{n=0}^{\infty}c_nq^n$ and $g(q) = \sum_{n=0}^{\infty}d_nq^n$ in $\mathbb{C}[[q]]$ is defined as $\sum_{n=0}^{\infty} c_nd_nq^n$ and denoted by $f \star g$.
\end{definition}
Several well-known subsets of $\mathbb{C}[[q]]$ are closed under the Hadamard product: the spaces of monomials, polynomials, rational functions \cite[Theorem 7]{jungen1931series} and D-functions (functions that satisfy a linear differential equation with polynomial coefficients) \cite[Theorem 9]{jungen1931series}. To the best of our knowledge, this question has not yet been answered for the entire space of $k$-Mahler functions, but Allouche and Shallit proved it for a subspace of so-called $k$-regular functions \cite{allouche1992ring}. This construction is studied in Subsection \ref{k-regular functions subsection}. To start, recall a simple observation:
\begin{prop}\label{prop Mahler recursion}
Let $f(q) = \sum_{n=0}^{\infty}x_nq^n$ be a $k$-Mahler function that satisfies the $k$-Mahler equation
\begin{align*}
    A(q) + A_0(q)f(q) + A_1(q)f(q^k)+ \dots + A_d(q)f\big(q^{k^d}\big) = 0
\end{align*}
for some $A(q),A_0(q),\dots,A_d(q) \in \mathbb{C}[q]$ with $A_0(q)A_d(q) \ne 0$ and $N = \max(\deg{A_0}(q),\dots,\deg{A_d(q)})$. If $A_i(q) = \sum_{j=0}^Na_{i, j}q^j$ and $n > \deg{A}$, then we have $$a_{0, 0}x_n + \dots + a_{0, N}x_{n-N} + a_{1,0}x_{\frac{n}{k}} + \dots + a_{1, N}x_{\frac{n-N}{k}} + \dots + a_{d, N} + \dots + a_{1, N} x_{\frac{n-N}{k^d}} = 0.$$
\end{prop}
A rational function $f(q) = \frac{r(q)}{s(q)}$ is $k$-Mahler as $s(q)f(q)-r(q) = 0$.
Moreover, the Hadamard product is commutative, associative and distributive over addition. Its identity is the rational function $1 + q + q^2 + \dots = \frac{1}{1-q}$. Only, the Hadamard product of $k$-Mahler functions is not always $k$-Mahler.
\begin{example}\label{CounterExample}
Let $k=2$, $f(q) = q + q^2 + q^4 + q^8 + \dots$ and $g(q) = \frac{1}{1-2q} = 1 + 2q + 4q^2 + 8q^3 + \dotsm$. Then $f$ satisfies $f(q) = q + f(q^2)$, so both are 2-Mahler. Their Hadamard product is
\begin{align*}
    \sum_{n=0}^{\infty} x_nq^n := (f \star g)(q) = 2q + 4q^2 + 16q^4 + 256q^8 + \dots + 2^{2^n}q^{2^n} + \dotsm.
\end{align*}
This means that
\begin{align*}
    x_n = \begin{cases}
    2^n &\text{if $n$ is a power of 2},\\
    0 &\text{otherwise}.
    \end{cases}
\end{align*}
If $f \star g$ is 2-Mahler, Proposition \ref{prop Mahler recursion} gives a recursion for some $N$ and $d$. If $n$ is a large power of 2, $0 \le i \le N$ and $0 \le j \le d$, then $\frac{n-i}{2^j}$ is a power of 2 if and only if $i = 0$. Thus, for large $m$ only a small number of non zero terms remain, and the recursion becomes
\begin{align*}
    a_{0, 0}x_{2^m} + a_{1, 0}x_{2^{m-1}} + \dots + a_{d, 0}x_{2^{m-d}} = 0. 
\end{align*}
As such, the sequence $(x_{2^m})_{m=0}^{\infty} = \big(2^{2^m}\big)_{m=0}^{\infty}$ satisfies a linear recursion while $\frac{x_{2^m}}{x_{2^{m-1}}} = 2^{2^{m-1}}$ tends to infinity. A contradiction. Thus, $f \star g$ is not 2-Mahler.
\end{example}
Any $f(q)$ with arbitrary long sequences of zeros in its $q$-expansion does the job in Example \ref{CounterExample}, so $F$, $G$, $H$ and $I$ from Chapter 1 work as well. One could argue that $\frac{1}{1-2q}$ is not a `legitimate' Mahler function as it is not defined in the entire unit disk. However, examples without this drawback exist. 
\begin{example}\label{Extended counter example}
If $k \ge 2$, $\alpha$ is non-zero and not a root of unity, $g(q) = \frac{1}{1-\alpha q}$ and $f = q + q^k + q^{k^2} + \dotsm$ satisfies $f(q) = q + f(q^k)$, then $(f \star g)(q) = \alpha q+\alpha^kq^k + \alpha^{k^2}q^{k^2} + \dotsm$ and the recursion
\begin{align*}
    a_{0, 0}x_{k^m} + a_{1, 0}x_{k^{m-1}} + \dots + a_{d, 0}x_{k^{m-d}} = 0. 
\end{align*}
is obtained as in Example \ref{CounterExample} for large $m$. For such large $m$, $\alpha^{k^{m-d}}$ is a root of $a_{0, 0}q^{k^d}+ a_{1, 0}q^{k^{d-1}} + \dots + a_{d, 0}$ because $x_{k^m} = \alpha^{k^m}$. As $\alpha$ is not a root of unity, all $\alpha^{k^{m-d}}$ differ from each other, giving that the polynomial has an infinite number of roots and hence is the zero polynomial. Thus, the recursion is trivial and $f \star g$ is not $k$-Mahler.\medskip

Not all numbers on the unit circle are roots of unity or transcendental. For example, Lehmer's polynomial of degree 10 \cite{lehmer1933factorization} has 8 roots on the unit circle that are not roots of unity. In general, the projection of an algebraic number onto the unit circle is algebraic but often not a root of unity.
\end{example}
Meanwhile, there are many pairs Mahler whose Hadamard product is again Mahler. 
\begin{example}
We call a $k$-Mahler function \textbf{special} if, as defined in Proposition \ref{prop Mahler recursion}, all $A_i(q) = a_i$ are constant and $A(q) = \sum_{i=0}^Mb_iq^i$ with $b_i = 0$ if $i$ is not a $k$th power. Examples of special Mahler functions are $f$ in Example \ref{CounterExample} and $g(q) = q + q^2 + 2q^4 + 3q^8 + \dots + F_{n+1}q^{2^n} + \dotsm $, where $F_n$ is the $n$th Fibonacci number, which satisfies $g(q) = g(q^2)+g(q^4) + q$.\medskip

There is a natural bijection between the special $k$-Mahler functions and rational functions defined at $q=0$ that preserves the Hadamard product: $x_1q + x_kq^k + x_{k^2}q^{k^2} + \dots\mapsto x_1 + x_{k}q + x_{k^2}q^2+\dotsm$. Thus, the recursion for $x_n$ transforms into a linear recursion. As the rational functions are closed under the Hadamard product \cite[Theorem 7]{jungen1931series}, the special $k$-Mahler functions are as well. 
\end{example}
\section[The space of k-regular functions]{The space of $\boldsymbol{k}$-regular functions}\label{k-regular functions subsection}
Much larger spaces of Mahler functions that are closed under the Hadamard product exist. In this section, we study a space introduced by Allouche and Shallit \cite{allouche1992ring} and Becker \cite{becker1994k}.
\begin{definition}
The \textbf{$\boldsymbol{k}$-kernel} of a sequence $(x_n)_{n=0}^{\infty}$ is the set of all the subsequences of the form $(x_{k^en + r})_{n=0}^{\infty}$ where $e \ge 0$ and $0 \le r \le k^e-1$. 
\end{definition}
\begin{definition}
A sequence $(x_n)_{n=0}^{\infty}$ is called \textbf{$\boldsymbol{k}$-automatic} if its $k$-kernel is finite and \textbf{$\boldsymbol{k}$-regular} if its $k$-kernel is finitely generated.
\end{definition}
Naturally, the definitions of $k$-automatic and $k$-regular sequences translate to their generating functions. We give a few examples. The $k$-kernel of $\frac{1}{1-q}$ consists of one element: $(1,1,\dots)$. So $\frac{1}{1-q}$ is $k$-automatic and $k$-regular. Next, $\frac{q}{(1-q)^2}$ has a $k$-kernel generated by $(1,1,\dots)$ and $(0,1,2,\dots)$ and is  $k$-regular but not $k$-automatic. Lastly, $\frac{1}{1-2q}$ is neither $k$-automatic nor $k$-regular. By \cite[Theorem 2.10]{allouche1992ring}, the coefficients of a $k$-regular sequence have to grow polynomially.\medskip

More generally, a $k$-regular sequence is $k$-automatic if and only if it assumes a finite number of values \cite[Theorem 2.3]{allouche1992ring}. Moreover, both the sum and Hadamard product of two $k$-regular sequences are $k$-regular \cite[Theorem 2.5]{allouche1992ring}. There does not exist an easy characterisation of $k$-regular functions in terms of $k$-Mahler functional equations, but at least there is a one-sided relationship.
\begin{theorem}[Theorem 1 in \cite{becker1994k}]\label{Becker1}
A $k$-regular function satisfies a homogeneous $k$-Mahler equation.
\end{theorem}
The inclusion is strict, as $f(q) = \frac{1}{1-2q}$ demonstrates. It satisfies $(1-2q)f(q)-(1-2q^2)f(q^2) = 0$, but it is not $k$-regular as noted above. Becker showed a partial inverse:
\begin{definition}
A $k$-Mahler function as in Proposition \ref{prop Mahler recursion} with $A(q) = 0$ (so homogeneous) and $A_0(q) = 1$ is called a \textbf{$\boldsymbol{k}$-Becker function} or $k$-Becker for short.
\end{definition}
\begin{theorem}[Theorem 2 in \cite{becker1994k}]\label{Becker2}
A $k$-Becker function is $k$-regular.
\end{theorem}
Again, the inclusion is strict. Take $f(q) = q$, which is $k$-regular. If $f$ was $k$-Becker, the linear term of $A_i(q)f\big(q^{k^i}\big) = A_i(q)q^{k^i}$ would be zero and $A_0(q)f(q) = q$ would not cancel out in the Mahler equation. An \textbf{inhomogeneous $\boldsymbol{k}$-Becker function} is a not necessarily homogeneous $k$-Becker function.
\begin{prop}
An inhomogeneous $k$-Becker function is $k$-regular.
\end{prop}
\begin{proof}
Let $f$ be an inhomogeneous $k$-Becker function that satisfies the Mahler functional equation
\begin{align*}
    f(q) + A_1(q)f(q^k) + A_2(q)f\big(q^{k^2}\big) + \dots + A_d(q)f\big(q^{k^d}\big) + A(q) = 0
\end{align*}
for some polynomials $A_1,A_2,\dots,A_d,A \in \mathbb{C}[q]$ and $A_d \ne 0$. If $A(q) = 0$, we are done by Theorem \ref{Becker2}. Recall that polynomials are $k$-regular, and that the space of $k$-regular sequences is closed under addition and multiplication \cite[Theorem 3.1]{allouche1992ring}. Our goal is to prove that $f$ is a sum of $k$-regular functions and thus $k$-regular with several reductions.


\textit{Step 1.} If $e = \deg{A}$ and $A(q) = \sum_{n=0}^{e} a_nq^n$, there are Mahler functions $g_0(q),\dots,g_a(q)$ such that
\begin{align*}
    g_0(q) + A_1(q)g_0(q^k) + A_2(q)g_0\big(q^{k^2}\big) + \dots + A_d(q)g_0\big(q^{k^d}\big) + a_0 &= 0 \\
    \vdots\qquad\qquad &\quad \vdots \\
    g_e(q) + A_1(q)g_e(q^k) + A_2(q)g_e\big(q^{k^2}\big) + \dots + A_d(q)g_e\big(q^{k^d}\big) + a_n &= 0 
\end{align*}
and $f(q) = \sum_{n=0}g_e(q)$, and so, without loss of generality, $A(q)$ can be assumed to be a monomial. 

\textit{Step 2.} If $A(q) = cq^n $ is a monomial for some $c \in \mathbb{C} \setminus \{0\}$ and $n\ge 0$, then 
\begin{align*}
    q^{-n}f(q) + q^{-n}A_1(q)f(q^k) + q^{-n}A_2(q)f\big(q^{k^2}\big) + \dots + q^{-n}A_d(q)f\big(q^{k^d}\big) + c = 0.
\end{align*}
Now $A(q)$ can be assumed to be constant due to the map $f(q)$ to $g(q) := f(q)q^n$ which gives
\begin{align*}
    g(q) + q^{nk-n}A_1(q)g(q^k) + q^{nk^2-n}A_2(q)f\big(q^{k^2}\big) + \dots + q^{nk^d-n}A_d(q)f\big(q^{k^d}\big) + c = 0.
\end{align*}

\textit{Step 3.} Assume that $A(q) = c$, a non-zero constant. Then
\begin{align*}
    f(q) + A_1(q)f(q^k) + A_2f\big(q^{k^2}\big) + \dots + A_df\big(q^{k^d}\big) + c &= 0 \quad\text{and} \\
    f(q^k) + A_1(q^k)f\big(q^{k^2}\big) + A_2(q^k)f\big(q^{k^3}\big) + \dots + A_d(q^k)f\big(q^{k^{d+1}}\big) + c &= 0.
\end{align*}
Subtracting these two equations gives a homogeneous equation for $f$. This completes the reduction.\medskip

Thus, an inhomogeneous $k$-Becker function can be written as a sum of homogeneous $k$-Becker functions scaled by powers of $q$. If $n \ge 1$ and $f$ $k$-Becker, then $f$ is $k$-regular by Theorem \ref{Becker2}. As $q^n$ is a polynomial, it is $k$-regular, and so their product $q^nf(q)$ is $k$-regular. This completes the proof.\end{proof}
Again, this characterisation is incomplete. Take $P(q)$ as introduced in Subsection \ref{Subsection where P(x) is introduced}. It satisfies $(1+q^2)P(q) = (q+q^3)P(q^2) + 1$ and is $2$-automatic by Theorem 6.5.4 of \cite{allouche2003automatic} and thus $k$-regular, but it is unclear whether $P(q)$ is (inhomogeneous) $k$-Becker.
\section{Complete Hadamard functions}
The identity function with respect to the Hadamard product is $\frac{1}{1-q}$, so the Hadamard product of any $k$-Mahler function and $\frac{1}{1-q}$ is again $k$-Mahler. Such functions are quite rare. We name them:
\begin{definition}
A \textbf{complete Hadamard function} is a function $f(q) \in \mathbb{C}[[q]]$ such that the Hadamard product of $f$ and every $k$-Mahler function is $k$-Mahler.
\end{definition}
Simple examples of complete Hadamard functions are polynomials as their Hadamard product with a power series is again a polynomial. In this section, an elegant classification of the rational functions with this property is given.
\begin{theorem}\label{theorem complete Hadamard rational functions}
A rational function $f(q) = \frac{r(q)}{s(q)}$ with $r, s \in \mathbb{C}[q]$ coprime and $s(0) \ne 0$ is complete Hadamard if and only if all roots of $s(q)$ are roots of unity.\end{theorem}
Note that the condition $s(0) \ne 0$ implies that $f$ has a power series expansion. To start, we give a few basic observations on complete Hadamard functions:
\begin{lemma}\label{lemma small tricks compelte Hadamard}
Let $f(q)$ be complete Hadamard and $A \in \mathbb{C}[q]$. Then
\begin{enumerate}
    \item $f(q)$ is $k$-Mahler;
    \item The space of complete Hadamard functions forms a ring with multiplication being defined the Hadamard product and the multiplicative identity being $\frac{1}{1-q}$.
    \item $A(q)f(q)$ is complete Hadamard.
\end{enumerate}\end{lemma}
\begin{proof}
Let $h(q)$ be $k$-Mahler.
\begin{enumerate}
    \item $f(q) = f(q) \star \frac{1}{1-q}$ is $k$-Mahler.
    \item Let $g(q)$ be complete Hadamard. The space of complete Hadamard functions is closed under both addition and the Hadamard product. Then $(f + g) \star h = f \star h + g \star h$ is the sum of two $k$-Mahler functions and thus $k$-Mahler. As $g$ is complete Hadamard, $g \star h$ is $k$-Mahler, and hence $f \star (g\star h) = (f \star g) \star h$ as well. The other axioms follow trivially.
    \item If $\alpha \in \mathbb{C}$, $(\alpha f) \star h = \alpha (f \star h)$ is also $k$-Mahler, and thus $\alpha f(q)$ is complete Hadamard.  If $\tilde{h}(q) := \frac{h(q)-h(0)}{q}$, $\tilde{h}(q)$ is Mahler and analytic. Then 
    $$(qf)\star(h) = (qf)\star(g(0) + q\tilde{h}) = (qf)\star(g(0)) + (qf)\star(q\tilde{h}) = 0 + q(f\star \tilde{h})$$
    is $k$-Mahler, and thus $qf(q)$ is complete Hadamard. As the ring of complete Hadamard functions is closed under addition, iterating these two facts gives that $A(q) f(q)$ is complete Hadamard. \qedhere
\end{enumerate}\end{proof}
A more involved complete Hadamard function is $\frac{1}{(1-q)^n}$ for every $n \ge 0$.
\begin{lemma}\label{lemma Derivative is complete Hadamard}
The derivative $f'(q)$ of a $k$-Mahler function $f$ is $k$-Mahler.
\end{lemma}
\begin{proof}
Differentiating the $k$-Mahler equation of $f$ gives
\begin{align*}
    A'(q) + A_0'(q)f(q) + A_0(q)f'(q) + \dots + A_d'(q)f\big(q^{k^d}\big) + A_d(q)k^dq^{k^d-1}f'\big(q^{k^d}\big)= 0.
\end{align*}
Thus, $\dim_{\mathbb{C}(q)}(1, f(q), f'(q), f(q^k), f'(q^k), f\big(q^{k^2}\big), f'\big(q^{k^2}\big),\dots)\le 2d+2$, and so $f'(q)$ is $k$-Mahler by Corollary \ref{corollary dimension is sufficient} as $\dim_{\mathbb{C}(q)}(1, f'(q), f'(q^k),\dots )$ is finite.
\end{proof}
\begin{lemma}\label{lemma Derivative is complete Hadamard 2}
The derivative of a complete Hadamard function is complete Hadamard.
\end{lemma}
\begin{proof}
Let $f(q) = \sum_{n=0}^{\infty}a_nq^n$ be complete Hadamard and $g(q) = \sum_{n=0}^{\infty}b_nq^n$ be $k$-Mahler. Then
\begin{align*}
    f'(q) \star g(q) = \Big(\sum_{n=0}^{\infty}(n+1)a_{n+1}q^n\Big) \star \Big(\sum_{n=0}^{\infty}b_nq^n \Big) = \sum_{n=0}^{\infty}(n+1)a_{n+1}b_nq^n = \frac{d}{dq}(f(q)\star qg(q))
\end{align*}
is $k$-Mahler by applying Lemmas \ref{lemma small tricks compelte Hadamard} and \ref{lemma Derivative is complete Hadamard}.\end{proof}
\begin{lemma}
For all $n \ge 0$, $\frac{1}{(1-q)^n}$ is complete Hadamard.
\end{lemma}
\begin{proof}
As $\frac{1}{1-q}$ is complete Hadamard, its $(n-1)$st derivative $\frac{(n-1)!}{(1-q)^n}$ is as well by Lemma \ref{lemma Derivative is complete Hadamard 2}.\end{proof}
Now Theorem \ref{theorem complete Hadamard rational functions} can be proven. First, one direction:
\begin{prop}\label{prop complete Hadamard not root of unity}
Let $r(q)$ and $s(q)$ be coprime polynomials and $f(q) = \frac{r(q)}{s(q)}$ such that $s(0) \ne 0$ and $s$ has a root $\alpha$ that is not a root of unity. Then $f(q)$ is not complete Hadamard.
\end{prop}
\begin{proof}
If $f(q)$ were complete Hadamard, then $\frac{p(q)}{1-\alpha^{-1}q} := f(q) \frac{s(q)}{1-\alpha^{-1}q}$ would be complete Hadamard by Lemma \ref{lemma small tricks compelte Hadamard} and as $\frac{s(q)}{1-\alpha^{-1}q}$ is a polynomial. Write
$$
\frac{p(q)}{1-\alpha^{-1}q} = \frac{p(\alpha^{-1})}{1 - \alpha^{-1}} + u(q) = p(\alpha^{-1})\sum_{n=0}^{\infty}\alpha^{-n}q^n + u(q)
$$
for some polynomial $u(q)$. Thus, like in Example \ref{Extended counter example}, the Hadamard product of $\sum_{n=0}^{\infty}p(\alpha^{-1})\alpha^nq^n$ with $g(q)$ satisfying $g(q) = g(q^k) + q$ gives rise to a linear recurrence and a polynomial with an infinite number of roots. Thus $f \star g$ is not Mahler and $f(q)$ cannot be complete Hadamard.
\end{proof}
And the other way around:
\begin{lemma}\label{prop p(q)/(1-cq)}
For all $m \ge 1$, $\frac{1}{1 - q^m}$ is complete Hadamard.
\end{lemma}
\begin{proof}
Assume that $g(q)  =\sum_{n = 0}^{\infty} c_n q^n$ is $k$-Mahler. Then $$\frac{1}{1-q^m} \star g(q) = \sum_{n = 0}^{\infty} \sum_{i = 0}^m \zeta_m^i \alpha^n q^n = \sum_{i = 0}^m g(\zeta_m^i q).$$ Thus, as for all $0 \le i \le m$ 
\begin{align*}
    \dim_{\mathbb{C}(q)}\big(1, g(\zeta_m^i q), g((\zeta_m^i)^k q^k), g((\zeta_m^i)^{k^2} q^{k^2}), \dots\big)
\end{align*}
is finite, $\{g(\zeta_m^i q^{k^j}) : i, j \ge 0\}$ is finitely generated. Now Corollary \ref{corollary dimension is sufficient} gives that $\frac{1}{1-q^m} \star g(q)$ is $k$-Mahler and thus the proposition.
\end{proof}

\begin{prop}\label{prop complete Hadamard all roots of unity}
Let $r(q)$ and $s(q)$ be coprime polynomials and $f(q) = \frac{r(q)}{s(q)}$ such that all roots of $s$ are roots of unity. Then $f$ is complete Hadamard.
\end{prop}
\begin{proof}
If all roots of $s(q)$ are roots of unity, let $m$ be the least common multiple of these orders. Then $s(q)$ divides $(1-q^m)^l$ for some $l \ge 1$. If $l=1$, then this case has been treated in the proof of Lemma \ref{prop p(q)/(1-cq)}. If $l > 1$ and $g$ is $k$-Mahler, then $\frac{1}{1-q^m} \star g$ has a $k$-Mahler equation and its coefficients follow a recursion by Proposition \ref{prop Mahler recursion}. As only the coefficients $x_{\frac{n-i}{k^j}}$ matter when $i$ is a multiple of $m$, it is seen that all $A_i(q)$ and $A(q)$ are polynomials in $\mathbb{C}[q^m]$. Then $(\frac{1}{1-q^m} \star g)(q^{\frac{1}{m}})$ is also $k$-Mahler and
\begin{align*}
    \bigg(\frac{1}{(1-q^m)^l} \star g\bigg)(q) = \Bigg(\bigg(\bigg(\frac{1}{1-q^m} \star g\bigg)(q^{\frac{1}{m}})\bigg) \star \frac{1}{(1-q)^l}\Bigg)(q^m)
\end{align*}
is $k$-Mahler, and thus $\frac{1}{(1-q^m)^l}$ is complete Hadamard. Thus, $\frac{r(q)}{s(q)} = \frac{r(q)}{(1-q^m)^l} \frac{(1-q^m)^l}{s(q)}$ is complete Hadamard by Lemma \ref{lemma small tricks compelte Hadamard} as $\frac{(1-q^m)^l}{s(q)}$ is a polynomial by construction.\end{proof}
\begin{proof}[Proof of Theorem \ref{theorem complete Hadamard rational functions}]
Combine Propositions \ref{prop complete Hadamard not root of unity} and \ref{prop complete Hadamard all roots of unity}.
\end{proof}
The natural follow-up question is whether there are irrational complete Hadamard functions, but their existence is highly dubious. 


\bibliographystyle{plain}
\clearpage
\addcontentsline{toc}{chapter}{Bibliography}
\bibliography{main}

\end{document}